\RequirePackage{fix-cm}

\documentclass[smallextended]{svjour3}
\usepackage{amsmath,amssymb,
enumerate,
color,alltt,theorem}
\usepackage[english]{babel}
\usepackage{graphicx}
\usepackage{caption}
\usepackage{subcaption}
\usepackage{hyperref}
\newcommand{\bbox}{\operatorname{Box}}
\newcommand{\LBox}{\operatorname{LBox}}
\newcommand{\RBox}{\operatorname{RBox}}
\newcommand{\sep}{\operatorname{Sep}}
\newcommand{\logsep}{\operatorname{lSep}}
\newcommand{\cau}{\operatorname{Cau}}
\newcommand{\cont}{\operatorname{cont}}
\newcommand{\Mea}{\operatorname{Mea}}
\newcommand{\logmea}{\operatorname{lMea}}
\newcommand{\Len}{\operatorname{Len}}
\newcommand{\loglen}{\operatorname{lLen}}
\newcommand{\Disc}{\operatorname{Disc}}
\newcommand{\gdisc}{\operatorname{GDisc}}
\newcommand{\logGdisc}{\operatorname{lGDisc}}
\newcommand{\res}{\operatorname{Res}}
\newcommand{\LCF}{\operatorname{lc}}
\newcommand{\tcoeff}{\operatorname{tcoeff}}

\newcommand{\mult}{\operatorname{mul}}

\newcommand{\tO}{\tilde{O}}
\newcommand{\Z}{\mathbb{Z}}
\newcommand{\Q}{\mathbb{Q}}
\newcommand{\R}{\mathbb{R}}
\newcommand{\C}{\mathbb{C}}

\newcommand{\sr}{\mathrm{sr}}
\newcommand{\sign}{\operatorname{sign}}

\newcommand{\Crit}{\operatorname{Crit}}
\newcommand{\Sing}{\operatorname{Sing}}
\newcommand{\sDisc}{\operatorname{sDisc}}

\newcommand{\Ker}{\operatorname{Ker}}
\newcommand{\Ima}{\operatorname{Im}}
\newcommand{\Match}{\operatorname{Match}}
\newcommand{\eps}{\varepsilon}

\newcommand{\hide}[1]{}
\definecolor{grey}{rgb}{0.75,0.75,0.75}
\definecolor{orange}{rgb}{1.0,0.5,0.5}
\definecolor{brown}{rgb}{0.5,0.25,0.0}
\definecolor{pink}{rgb}{1.0,0.5,0.5}
\newtheorem{notation}[definition]{Notation}
\newtheorem{exemple}[definition]{Example}
\newtheorem{algorithm}{Algorithm}
\renewenvironment{proof}[1][\noindent{\it Proof}]{\noindent{\it #1}. } {\hspace*{\fill}\qedsymbol\medskip}
\newcommand{\qedsymbol}{$\square$}

\begin{document}
\title{
Bounds for polynomials on algebraic 
numbers and 
application to curve topology}

\author{
Daouda Niang Diatta \and
S\'eny Diatta \and
Fabrice Rouillier \and
Marie-Fran{\c c}oise Roy\and
Michael Sagraloff\\
 }

 \institute{D. N. Diatta \at
              {Universit{\'e} Assane Seck de Ziguinchor, Senegal}   
           \and
           S. Diatta \at
               {Universit{\'e} Assane Seck de Ziguinchor, Senegal}
               \and
           F. Rouillier \at
          {INRIA Paris, IMJ-PRG - Sorbonne Universit{\'e}s, France}
           \and * M.-F. Roy \at
            {IRMAR, Universit{\'e} de Rennes I, Campus de Beaulieu, 35042 Rennes CEDEX, France, 33(0)223236020}\\
            marie-francoise.roy@univ-rennes1.fr
             \and 
             M. Sagraloff \at
           {HAW Landshut, Germany}
}

\date{Received: date / Accepted: date}

\maketitle

\begin{abstract}
Let $P \in \Z [X, Y]$ be a given square-free polynomial  of total degree $d$ with integer coefficients of bitsize less than 
 $\tau$, and let $$V_{\mathbb{R}}
 (P) := \{ (x,y) \in \mathbb{R}^2, P (x,y)
 = 0 \}$$ be the real planar algebraic curve implicitly defined as the vanishing set of $P$. We give a deterministic 
 algorithm to compute the topology of $V_{\mathbb{R}} (P)$ in terms of a 
 simple straight-line planar graph
 $\mathcal{G}$ that is isotopic to $V_{\mathbb{R}} (P)$.
 The upper bound on the bit complexity of our algorithm is in $\tO (d^5 \tau
 + d^6)$\footnote{The expression "the complexity is in $\tO(f(d,\tau))$" with $f$  a polynomial in $d,\tau$ is an abbreviation for  the expression "there exists a positive integer $c$ such that the complexity is in $O((\log d \log \tau)^c f(d,\tau))$".} ;
 which matches the current record bound
 for the problem of computing the topology of a planar algebraic curve.
However, compared to existing algorithms with comparable complexity, our method 
 does not consider any change of coordinates, and more importantly the returned 
 simple planar graph 
 $\mathcal{G}$ yields the cylindrical algebraic decomposition information of the curve
 in the original coordinates.

Our result is based on two main ingredients: First, we 
derive amortized quantitative bounds on the roots of polynomials with algebraic coefficients as well as adaptive methods for computing the roots of 
bivariate polynomial systems
 that actually exploit this amortization. 
 The results we obtain are more general  that the previous literature.
Our second ingredient is a novel approach for the computation of the 
 local topology of the curve in a neighborhood of all 
 singular
 points.

\end{abstract}


\section{Introduction}
Let $P \in \Z [X, Y]$ be a given square-free polynomial  of total degree $d$ with integer coefficients of bitsize less than 
 $\tau$. 
The problem of computing the topology of the planar algebraic curve
\[
V_{\mathbb{R}}
 (P) := \{ (x,y) \in \mathbb{R}^2, P (x,y)
 = 0 \},
\]
 that is, the computation of 
a simple  straight-line  planar graph 
isotopic to
$V_{\mathbb{R}} (P)$ inside $\mathbb{R}^2$, is a classical problem in
algorithmic real algebraic geometry with many
applications in Computer Aided Geometric Design. It is extensively studied in
the context of symbolic 
or semi-numerical computation; for instance, see~\cite{AM1,AM2,AMW,BEKS,BCGY,CLPPRT,DMR,EKW,GE1,GI,KS,KoS,MSW2,WM} for recent references. 

Almost all 
 algorithms are based on some variant of
Cylindrical Algebraic Decomposition (C.A.D.): Decompose the $X$-axis into a finite number of
open intervals and points above which the curve has a cylindrical structure (i.e.~the curve decomposes into disjoint function graphs above each of these intervals).
The \emph{special values} are the projections of the $X$-critical  
points and vertical asymptotes
onto the $X$-axis.
Taking
points between two special values defines 
\emph{regular} fibers.

   \begin{figure}[h]
 \center{\resizebox{5cm}{4cm}
 {\includegraphics{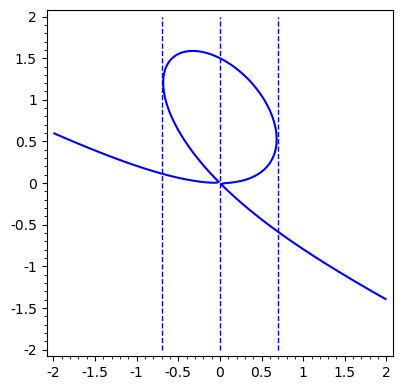}}}
\caption{Example of a C.A.D. with three special values}
  \end{figure}
  
Computing a 
simple planar graph 
isotopic to $V_{\mathbb{R}} (P)$ inside $\mathbb{R}^2$ then
 essentially amounts to connecting the points of a regular fiber to the
points of its neighboring special fibers.

There are two main difficulties:
\begin{itemize}
\item[(1)] computing the points of the special fibers,
which amounts for computing the real roots of univariate polynomials with real
algebraic coefficients, which are not square-free;
\item[(2)] computing the number 
 curve segments
  that go to each of the
 points of the special fiber, to the left and to the right.
\end{itemize}

A usual strategy for  dealing with (2)
(see \cite{BPRbook2,DMR,EKW,GE1,GI,DBLP:phd/de/Kerber2009,KoS,MSW2}),
consists in putting the curve in 
\emph{generic position}. This is typically achieved by considering a 
\emph{shearing} that maps $X$ to $X+s\cdot Y$ for some (small) integer $s$. 
For most values of $s$,
the sheared curve has no asymptotes and each special fiber contains at most one $X$-critical or singular point, which considerably eases (2). 
Since a shearing does not change the topology of the curve, the 
simple planar graph 
returned by such algorithms is still isotopic to the curve. However,  
these algorithms do not compute a C.A.D. of the curve itself but only of the sheared curve. This  
is
critical in some applications, for instance, 
when the variables $X$ and $Y$ represent physical values (such as length) that are always positive, or probabilities that are between $0$ and $1$. The topology of the curve restricted to a quadrant or a square can be easily extracted from our output but would be complicated to recover after a shearing of the form $X\mapsto X+s\cdot Y$. Another advantage of keeping the initial variables is that the sparseness of the input equation is preserved.

There exist algorithms~\cite{EKW,DBLP:phd/de/Kerber2009} that go one step further by performing a shearing of the curve in the first step and an inverse shearing in the second step in order to eventually compute an isotopic 
simple planar graph 
whose vertices are located on the curve. Such algorithms compute the C.A.D. of the given curve, however the bit complexity of these methods falls clearly behind the best algorithms~\cite{KoS,MSW2} for computing the topology of a planar algebraic curve, which achieve the complexity bound\footnote{The algorithm from~\cite{KoS} is deterministic, whereas~\cite{MSW2} uses randomization. Both algorithm consider a shearing of the original curve and only return the C.A.D. information of the sheared curve.} $\tO(d^5\tau+d^6)$. Our algorithm achieves the same complexity bounds, but it never performs any coordinate transformation and yields the C.A.D. information of the original curve.

\begin{theorem}[Topology]\label{finaltopology}
 Let $P \in \Z [X, Y]$ be a given square-free polynomial of total degree $d$
and integer coefficients of bitsize bounded by $\tau$. There is a deterministic  
algorithm\footnote{We do not only prove existence of such an algorithm, but also present the algorithm in this paper.} that uses $\tO(d^5\tau+d^6)$ bit operations to compute the topology of the curve $$V_{\mathbb{R}} (P)= \{(x, y) \in \mathbb{R}^2 \mid P (x, y) = 0\},$$ in terms of a 
simple  planar graph 
 $\mathcal{G}$ that is isotopic to $V_{\mathbb{R}}(P)$ inside $\mathbb{R}^2$. In addition, $\mathcal{G}$ yields the C.A.D. information of the curve $V_{\mathbb{R}}(P)$.\footnote{We remark that our algorithm returns a purely combinatorial representation of the C.A.D., however isolating intervals of all points in all special and regular fibers are computed in sub-steps of the algorithm. By refining these intervals, an isotopic 
 simple planar graph
  whose vertices are arbitrarily close to the curve can be 
 obtained.}
 \end{theorem}

For our result, we use two main ingredients: The first one is a
new
efficient algorithm for computing the roots of a bivariate system in triangular form. 

\begin{theorem}[Bivariate Root Isolation]\label{firstmain}
Let $R\in\Z[X]$ and $F\in \Z[X,Y]$ be polynomials of total degrees $N$ and $n$ and with integer coefficients of bitsize less than $\Lambda$ and $\tau$ respectively, such that  
$$V_{\mathbb{C}}(R,F)=\{(x,y)\in \mathbb{C}^2\mid R(x)=F(x,y)=0\}$$
is finite. Using a total number of bit operations bounded by
\[
\tO(N^2\Lambda+N^3+n^5\tau+n^6+n\cdot\max(n^2,N)\cdot (N\tau+n\Lambda+Nn)),
\]
we can compute
\begin{itemize}
\item[(a)] $\deg F(x,Y)$ as well as $\deg \gcd (F(x,Y),\partial_Y F(x,Y))$ for all complex roots $x$ of $R$,
\item[(b)] isolating disks ${\mathcal D}_{x,y}\subset\mathbb{C}$ for all complex roots $y$ of all polynomials $F(x,Y)$ for all complex roots $x$ of $R$, as well as the corresponding multiplicities 
$
\mult(y, F(x,Y))$
\item[(b')] for a given subset $V\subset\{(x,y)\in\mathbb{C}^2:R(x)=0\text{ and }F(x,y)=0\}$ and an $L\in\mathbb{N}$, we can further refine all isolating disks ${\mathcal D}_{x,y}$, with $(x,y)\in V$, to a size less than $2^{-L}$, in an additional number of bit operations bounded by 
$$ \tO(N^2\Lambda+N^3+n\cdot\max(n^2,N)\cdot(N\tau+n\Lambda+Nn)+L\cdot(N\cdot\mu+n^2\cdot\sum_{x \in \pi_X(V)}\mu_{x})),$$
where $\mu_{x}:=\max_{(x,y) \in V}
\mult(y, F(x,Y))$, 
$\pi_X(V)$ is the projection of $V$ on the $X$-axis and $\mu=\max_{x \in \pi_X(V)}\mu_{x}$.
\end{itemize}
\end{theorem}

The results in Theorem~\ref{firstmain} are significantly more general than the  results in the existing literature. 
For instance, recent work~\cite{ST2018} makes two strong assumptions, namely that the degree of the polynomials $F(x,Y)$ stays invariant for all roots $x$ of $R$ and that each $F(x,Y)$ has only simple roots. In~\cite{KS,KoS,MSW2}, the polynomials $F(x,Y)$ are allowed to have multiple roots, but the degree of $F(x,Y)$ is assumed to be the same for all $x$. In addition, the analysis restricts to the special case, where $R$ is either the resultant polynomial $\res(F,\partial_YF,Y)$ or its derivative.
None of these restrictions are needed in our results.

Using Theorem~\ref{firstmain} as our main tool, we obtain an efficient 
extension of the C.A.D.
for an algebraic curve;  see Theorem ~\ref{sing-fibers0} and Theorem ~\ref{sing-fibers}.

Our second ingredient is an algorithm for computing the number of branches reaching a singular point inside an adjacency box (see Section \ref{sec:top} for a precise definition) from information computed on the boundary of the box (see Figure \ref{description_lists}, Figure \ref{matching} and Algorithm \ref{algoconnect}). In order to achieve the complexity bounds we are aiming at, we need to accept that some of the desired information stays ambiguous\footnote{For instance, we aim to compute the location of the intersections of the curve with the four boundary edges of the adjacency box. However, we were not able to show how to distinguish between the special case, where the curve passes exactly the corner of a box, and the generic case, where the curve intersects one of the neighboring edges close to the corner, using only $\tO(d^6+d^5\tau)$ bit operations. In such cases, the information at the corner of the box stays ambiguous.}, however we will prove 
that these ambiguities do not prevent us from computing the correct connectivity; see Proposition \ref{ambigousnoproblem}.\\

For both of the above ingredients, we make essential use of amortized quantitative bounds for polynomials with algebraic coefficients by considering adaptive algorithms that make it possible to exploit this amortization; see Proposition \ref{generalunivariate} and Proposition  \ref{generalbivariate}. Finally a precise combinatorial description of the information needed to draw the
simple planar graph 
isotopic to the curve is given.
Since we perform our computations without any change of variables, vertical lines and asymptotes need to be dealt with.

\paragraph{Organization of the paper}

The detailed description and complexity analysis of each
step of our algorithm  computing the isotopy type is given in Section \ref{sec:top}.

The two preceding sections are devoted to univariate and bivariate results about roots of polynomials. In each of the two sections, a part is devoted to recalling well known results, but there are also several new results, particularly amortized quantitative bounds about algebraic numbers weighted with multiplicities (see Proposition \ref{generalunivariate} and Proposition \ref{generalbivariate}) that play a key role in the proof of Theorem \ref{firstmain} 
 or in the results from Section \ref{sec:top}.

\section{Univariate results} \label{univariateres}

\subsection{Quantitative geometry of the roots}

In our complexity analysis we are going to use some quantitative results related to the geometry of the roots. We first fix some convenient notations and definitions.

\begin{definition}
\label{def0}
  For a polynomial $f = \sum^n_{i = 0} a_i X^i \in \mathbb{C} [X], \text{with }a_n \neq 0$, 
$\LCF(f)=a_n$,
 is the  \emph{leading coefficient} of $f$.
   The \emph{length} of $f$ is defined as
    \begin{equation}
  \Len (f) := \sum_{i = 0}^n | a_i | .
  \end{equation}
  The following result is straightforward: 
  \begin{equation}
   |f (x) | \le \Len (f)\cdot \max (1, | x |)^{n}.\label{subst1}
    \end{equation}
    The \emph{norm} of $f$ is defined as
    \begin{equation}
 \| f \| := \sqrt{\sum_{i = 0}^n | a_i |^2} .
   \end{equation}
   Given $x \in \mathbb{C}$, we denote by
 $\mult(x,f)$ the \emph{multiplicity} of $x$ as a root of $f$.
  We denote by  $\mathrm{V}_{\mathbb{C}}(f)$
   the set of the distinct roots of $f$ in $\mathbb{C}$ so that
  $$f = \LCF(f) \cdot\prod_{x\in  \mathrm{V}_{\mathbb{C}}(f)} (X - x)^{\mult(x,f)}.$$

 The \emph{square-free part} $f^\star_\mathbb{C}$ of $f$ is defined as
    $$f^\star_\mathbb{C} = 
    \prod_{x\in  \mathrm{V}_{\mathbb{C}}(f)} (X - x).$$

  The \emph{Mahler measure} of $f$ is defined as
      \begin{equation}\label{mahler}
      \Mea (f) := | \LCF(f) | \cdot\prod_{x\in  \mathrm{V}_{\mathbb{C}}(f)} \max (1, | x|)^{\mult(x,f)}.    
      \end{equation}
      Note the multiplicativity of  the Mahler measure 
       \begin{equation} \label{multiplicative}
       \Mea (f\cdot g)=\Mea (f)\cdot \Mea (g).
       \end{equation}

The norm of $f$, its length and its Mahler measure are related as follows  (see for example \cite[Propositions 10.8 and 10.9]{BPRbook2}):
  \begin{equation}
  \label{comparebis}
  2^{-n} \Len (f) \leqslant \Mea (f) \leqslant \| f \| .
   \end{equation}
  The \emph{separation of $f$ at $x$} is defined as
  \begin{equation}
 \sep(x,f):= \min_{y \in\mathrm{V}_{\mathbb{C}}(f), \atop  y \not{=} x} |y-x |.
   \end{equation}
 \end{definition}

Given a domain $\mathbb{D}$ and  two polynomials $f,g$ in $\mathbb{D}[X]$ of respective degrees $n,n'$, we denote by $\res(f,g)$ their resultant and by $\sr_{k}(f,g)$ their $k$-th subresultant coefficient, with  $\sr_{0}(f,g)=(-1)^{n(n-1)/2}\res(f,g)$ (see for example \cite[Section 4.2]{BPRbook2}).
If $\mathbb{D}=\mathbb{C}$  it is well known (see for example \cite[Theorem 4.16]{BPRbook2}) that
 \begin{equation}\label{equ:res}
\res(f,g)=\LCF(f)^{n'} \prod_{x \in V_{\mathbb{C}}(f)} g(x)^{\mult(x,f)}.
  \end{equation}
 The \emph{discriminant}  $\Disc(f)$ is the element of $\mathbb{D}$ such that 
  \begin{equation}\label{equ:dis}
   \LCF(f)\cdot \Disc(f)=\sr_{0}(f,f')=(-1)^{n(n-1)/2}\res(f,f').
   \end{equation}
   If $\mathbb{D}=\mathbb{C}$, denoting by $x_1,\ldots, x_{\deg(f)}$ its complex roots  
    \begin{equation}\label{equ:disc}
\Disc(f)=\prod_{i<j } (x_i-x_j)^2
  \end{equation}
   (see for example \cite[Sections 4.1 and 4.2]{BPRbook2}).
   
   If $\varphi$ is a ring homomorphism between  $\mathbb{D}$ and  $\mathbb{D'}$, the following properties are  easy to prove
   \begin{itemize}
   \item if $\varphi(a_n)\not=0$, $\varphi(\Disc(f))=\Disc(\varphi(f))$, 
   \item if $\varphi(a_n)=0$, $\varphi(\Disc(f))=\varphi(a_{n-1})^2 \cdot \Disc(\varphi(f))$, 
  \end{itemize}

 The \emph{$k$-th subdiscriminant coefficient} $\sDisc_{k}(f)$ is the element of $\mathbb{D}$ such that 
   \begin{equation}\label{equ:subdis}
   \LCF(f)\cdot \sDisc_{k}(f)=\sr_{k}(f,f').
     \end{equation}
If $\mathbb{D}=\mathbb{C}$,  
$m$ is the biggest index such that 
$\sDisc_{n-m}(f)\not= 0$ 
if and only if $f$ has
$m$
 distinct roots in $\mathbb{C}$ (see for example \cite[Sections 4.1 and 4.2]{BPRbook2}).

Now we introduce the generalized discriminant which will play a key role in the statement of several of our quantitative results.
 We are not aware of a clear identification of this notion earlier in the literature.
 
 \begin{notation} \label{not:fk}
 Let $f  \in \mathbb{C} [X]$ be a polynomial of degree $n$.
 Defining
\begin{equation}
f^{[i]}:=\frac{f^{(i)}}{i!} \label{fk},
\end{equation}
the Taylor formula writes as
$$f(\alpha+X)=\sum_{i=0}^{n} f^{[i]}(\alpha)X^{i}$$
\end{notation}

  \begin{definition} \label{def:gdisc}
     Let $\mathbb{D}$ be a domain, and let $f$  in $\mathbb{D}[X]$ of degree $n$.
The \emph{generalized discriminant} of $f$,   $\gdisc(f)$ is the element of $\mathbb{D}$
such that
$$ \LCF(f)\cdot \gdisc(f)=\tcoeff(\res_X(f,\sum_{k=1}^n U^{k-1}f^{[k]} ))$$  where $\tcoeff(g)$ is the tail coefficient of the polynomial $g\in\mathbb{D}[U]$, i.e. the coefficient of its non-zero term of lowest degree in $U$. 
   \end{definition}
   
   Notice that the  generalized discriminant is never $0$.  In the special case where all the roots of $f$ are simple we have:
   \begin{equation}
    \gdisc(f)=(-1)^{n(n-1)/2}\cdot \Disc(f) 
   \end{equation}
  
     \begin{proposition} \label{prop:gdisc}
    If $\mathbb{D}=\mathbb{C}$, 
 $$\begin{array}{rcl} \gdisc(f) &= &\displaystyle{\LCF(f)^{n-2}\cdot
  \prod_{x \in V_{\mathbb{C}}(f)} f^{[\mult(x,f)]}(x)^{\mult(x,f)}}\\
  &=&\displaystyle{ \LCF(f)^{2n-2}\cdot
  \prod_{x,y \in V_{\mathbb{C}}(f),x\not= y} (x-y)^{\mult(x,f) \mult(y,f)}}.
  \end{array}.$$ 
 \end{proposition}
 
 \begin{proof}
 The equality 
$$\LCF(f)^{n-2}\cdot
  \prod_{x \in V_{\mathbb{C}}(f)} f^{[\mult(x,f)]}(x)^{\mult(x,f)}=  \LCF(f)^{2n-2}\cdot
  \prod_{x,y \in V_{\mathbb{C}}(f), x\not= y} (x-y)^{\mult(x,f) \mult(y,f)},$$  is clear.
  The equality  $$\gdisc(f) = \LCF(f)^{n-2}\cdot
  \prod_{x \in V_{\mathbb{C}}(f)} f^{[\mult(x,f)]}(x)^{\mult(x,f)}$$
  follows from the fact that
  $$\res_X(f,\sum_{k=1}^n U^{k-1}f^{[k]} )=\LCF(f)^{n-1}\cdot
  \prod_{x \in V_{\mathbb{C}}(f)} \sum_{k=1}^n U^{k-1}f^{[k]}(x)^{\mult(x,f)}$$
  by (\ref{equ:res}).
 \end{proof}

\begin{definition}
\label{def}
  For a
  non-zero
   polynomial $f \in \mathbb{C} [X]$, we define
 
 \begin{eqnarray*}
  \loglen (f) & := & \log(\Len(f))
  ,\\
 \logmea (f) & := & 
  \log(|\LCF(f)|)+ \sum_{x\in \mathrm{V}_{\mathbb{C}}(f)}\mult(x,f)\cdot \log (\max (1, |x|))
  ,\\
\logsep (f) & := & 
 \sum_{x\in  \mathrm{V}_{\mathbb{C}}(f)} \mult(x,f)\cdot |\log (\sep (x, f))|
 ,\\
 \logGdisc(f) &:=&
 \sum_{x\in \mathrm{V}_{\mathbb{C}}(f)} \mult(x,f) \cdot
 |\log (|f^{[\mult(x,f)]}(x)|)|.
 \end{eqnarray*}
 \end{definition}

\begin{remark}
Notice that while
 \begin{eqnarray*}
  \loglen (f) & = & \log(\Len(f))
  ,\\
 \logmea (f) & = & \log(\Mea(f))
  \end{eqnarray*}
  we have, because of the absolute values,
 \begin{eqnarray*}
\logsep (f) & \ge & 
  \log(\sep(f))
 ,\\
 \logGdisc(f) &\ge &
  \log(\gdisc(f)).
 \end{eqnarray*}
\end{remark}

  As a consequence of (\ref{comparebis}),
\begin{equation}
\label{comparelog30}
\loglen (f) \le \logmea (f)+\deg(f).
\end{equation}

In the next step, we give a bound for $\logsep(f)$ relating it to 
$\logmea(f)$,
  and  $\logGdisc(f)$.

\begin{proposition}\label{GammaDeltaSigma}
Let $f\in\C[X]$ be a polynomial of degree $n$.
If $|\LCF(f)| \ge 1$,

\begin{equation}\label{GammaDeltaSigmab}
\logsep(f)\in O(n(n+
\logmea(f))+\logGdisc(f))
\end{equation}
\end{proposition}

For the proof of Proposition~\ref{GammaDeltaSigma}, we need the following lemma:

\begin{lemma}\label{6neighbours}
Let $S$ be a finite subset of $\mathbb{C}$. Consider a mapping $\phi: S\mapsto  S$ that maps any element  $x$ of $S$ to an arbitrary element $y\in S\setminus\{x\}$ that minimizes the distance to $x$. Then, each element $y\in S$ has at most $6$ preimages under the mapping $\phi$.
\end{lemma}
\begin{proof}
As the distance between two preimages must be greater than
 or equal to 
 the distances between $y$ and each of its preimages, the claim follows directly from the fact that if $OAB$ is a triangle and $AB\ge\max(OA,OB)$, then the angle 
$\widehat{BOA}$ has a measure at least $2\pi/6$. The value 6 is obtained in the case of a regular hexagon and its center.
\end{proof}

\begin{proof}[Proof of Proposition \ref{GammaDeltaSigma}]
Suppose that $f$ is  monic.
We first prove that, for any root $x$ of $f$,
it holds 
\begin{equation}\label{GDS}
\sep(x,f)^{\mult(x,f)}>|f^{[\mult(y,f)]}(y)|\cdot 2^{-n}\cdot\max(1,|y|)^{-n}\cdot 
\Mea (f) ^{-1}
\end{equation}
where $y\neq x$ is an arbitrary root of $f$ that minimizes the distance to $x$. The inequality 
(\ref{GDS}) follows directly from the following computation:
\begin{align*}
|f^{[\mult(y,f)]}(y)|&=
\prod_{
z \in \mathrm{V}_{\mathbb{C}}(f)\setminus\{y\} }|z-y|^{\mult(z,f)}\\
&=\sep(x,f)^{\mult(x,f)}\cdot
\prod_{
z \in \mathrm{V}_{\mathbb{C}}(f)\setminus\{x,y\} }|z-y|^{\mult(z,f)}\\
&\le \sep(x,f)^{\mult(x,f)}\cdot\Mea(f(X+y)).
\end{align*}
Using the multiplicativity of the Mahler measure (see \ref{multiplicative}) it is easy to prove that
\begin{align*}
\Mea(f(X+y))&\le 2^n\cdot \Mea (f) \cdot\max(1,|y|)^n.
\end{align*}
by checking the case of a monic polynomial of degree $1$.
Using Lemma \ref{6neighbours} and its notation,
each element $y\in \mathrm{V}_{\mathbb{C}}(f)$ has at most $6$ preimages under the mapping $\phi$. We thus conclude that
\begin{align*}
&\prod_{x \in \mathrm{V}_{\mathbb{C}}(f)
} \sep(x,f)^{\mult(x,f)}\\
&\quad>2^{-n^2}\cdot  
\Mea(f)^{-n}\cdot \prod_{x \in \mathrm{V}_{\mathbb{C}}(f)
}\max(1,|\phi(x)|)^{-n}\cdot \prod_{x \in \mathrm{V}_{\mathbb{C}}(f)
} |f^{[\mult(\phi(x),f)]}(\phi(x))|\\
&\quad \ge 2^{-n^2}\cdot 
\Mea(f)^{-n}\cdot\prod_{x \in \mathrm{V}_{\mathbb{C}}(f)} \max(1,|x|)^{-6n}\cdot\prod_{x \in  \mathrm{V}_{\mathbb{C}}(f)
}
\min(1,|
f^{[\mult(\phi(x),f)]}(\phi(x))|)
\\
&\quad \ge 2^{-n^2}\cdot 
\Mea(f)^{-n}\cdot 
 \Mea(f)^{-6n}
\cdot\prod_{x \in \mathrm{V}_{\mathbb{C}}(f)} \min(1,|f^{[\mult(x,f)]}(x)|)^6,
\end{align*}
which shows that 
$$-\sum_{x \in \mathrm{V}_{\mathbb{C}}(f)\atop \sep(x,f)<1}\mult(x,f)\cdot \log (\sep(x,f))\in O(n(n+
\logmea(f))+\logGdisc(f)).$$ 
 It remains to estimate 
 $$\sum_{x \in \mathrm{V}_{\mathbb{C}}(f)\atop \sep(x,f)\ge 1}\mult(x,f)\cdot \log (\sep(x,f)).$$
  It holds that 
\begin{equation} \label{usefulsep}
\sep(x,f)=|x-\phi(x)|\le 2 \cdot\max(1,|x|)\cdot\max(1,|\phi(x)|).
\end{equation}
 Thus, we get 
\begin{align*}
 \prod_{x \in \mathrm{V}_{\mathbb{C}}(f)\atop  \sep(x,f)\ge 1} 
 \sep(x,f)^{\mult(x,f)}&\le 2^{n}\cdot \Mea(f)^{7}
\end{align*}
and finally
\begin{equation}\label{GDS1}
\sum_{x:\sep(x,f)>1} \mult(x,f)\cdot \log(\sep(x,f))\in O(n + \logmea(f)).
\end{equation}
For the non-monic case, the claim follows from the result in the monic case and the fact that if $|\alpha|\ge 1$
\begin{itemize}
\item $\logmea(f)\le \logGdisc(\alpha f)$,
\item $\logGdisc(f)\le \logGdisc(\alpha f)$.
\end{itemize}
\end{proof}

\subsection{Quantitative results for univariate polynomials with integer coefficients}

For polynomials $f$ with integer coefficients, we introduce the notation of the \emph{magnitude} of $f$.

\begin{definition}
 A polynomial $f\in \Z[X]$ is of \emph{magnitude} 
 bounded by
  $(n,\tau)$ if its degree is bounded by $n$ and its bitsize is bounded by $\tau$, that is, the absolute value of each of its coefficients is bounded by $2^\tau$.
\end{definition}

 If $f\in \Z[X]$ is of magnitude 
 bounded by
 $(n,\tau)$,
then
$f^{[k]}$ is of magnitude
bounded by $(n,\tau+n)$ (using Notation \ref{not:fk}), using that $\binom{n}{k}<2^n$ for all $k=0,1,\ldots,n$.

It is clear that
 \begin{equation}
 \label{comparelen}
\Len(f)\le (n+1)\cdot 2^{\tau}.
 \end{equation}

Moreover, (\ref{comparebis}) implies
\begin{equation}
\label{compare}
\Mea (f)
  \leqslant  \sqrt{n+1}\cdot 2^\tau,
\end{equation}
and thus
\begin{equation}
\label{comparelog1} \logmea (f) \in {O} (\tau+\log (n)).
\end{equation}

Our next aim is to give bounds on $\logsep(f)$ that depend on the magnitude $(n,\tau)$ of $f$. In Proposition \ref{GammaDeltaSigma}, we have already derived a bound on $\logsep(f)$ that is related to $n$,
$\logmea(f)$
 and $\logGdisc(f)$.
 Given (\ref{comparelog1}) 
 we are left to bound $\logGdisc(f)$ in terms of $n$ and $\tau$. For this, we first derive a general bound on the product of the absolute values that a sequence of integer polynomials $g_1,\ldots,g_m$ takes at corresponding roots of $f$. A corresponding bound on $\logGdisc$ will then follow by applying our result to the sequence 
 of $f^{[k]}$, $k=\deg(f)-1,\ldots, 0$.

\begin{proposition}
\label{generalunivariate}
Let $f$ be a 
non-zero
polynomial in $\Z[X]$
 of magnitude 
bounded by
$(n_1,\tau_1)$ 
 and let $(g_i)_{1 \leq i \leq m}$ be a sequence of 
 non-zero
 polynomials in $\Z[X]$ each
  of magnitude 
bounded by
$(n_2,\tau_2)$.
\begin{itemize}
\item[a)]
Let $A \subset  \mathrm{V}_{\mathbb{C}}(f)$ be
such that for each $x\in A$, there exists $i(x)\in\{1,\ldots,m\}$  with $g_{i(x)}(x)\not=0$. Then, 
it holds
 \begin{equation}\label{boundgeneralunivar}
 \sum_{x\in A} \mult(x,f)
     \log (|g_{i(x)} (x)|) \in \tilde O(n_1\tau_2+n_2\tau_1).
    \end{equation}
    
    \item[b)] Suppose moreover that, for every root
 $x\in  \mathrm{V}_{\mathbb{C}}(f)$,
there exists an $i$ such that $g_i(x)\not=0$. Denoting by $i(x)$ the smallest value of $i$ such 
 that $g_i(x)\not=0$, it holds
 \begin{equation}\label{boundgeneraluniv}
 \sum_{x\in  \mathrm{V}_{\mathbb{C}}(f)}
    \mult(x,f) |\log (|g_{i(x)}(x)|) | \in \tilde O (n_1 \tau_2+n_2 \tau_1).
    \end{equation}
    \end{itemize}
\end{proposition}

\begin{proof}
\noindent a) 
 For any $x \in A,$ we have
\begin{equation}
\label{basis}
\left|g_{i(x)}(x)\right|\leq2^{\tau'_2}\cdot \max(1,\left|x\right|)^{n_2};
\end{equation}
by applying (\ref{subst1}) and
(\ref{comparelen})
, with $\tau'_2$ defined by  $2^{\tau'_2}=(n_2+1)2^{\tau_2}$.
Then, 
$$
\begin{array}{rcl}
\prod_{x \in A}\left|g_{i(x)}(x)\right|^{\mult(x,f)}  &\leq &\prod_{x\in A} 2^{\tau'_2 \mult(x,f)}\cdot \max(1,\left|x\right|)^{n_2 \mult(x,f)} \\
&\leq &\prod_{x\in  \mathrm{V}_{\mathbb{C}}(f)} 2^{\tau'_2 \mult(x,f)}\cdot \max(1,\left|x\right|)^{n_2 \mult(x,f)},
\end{array}
$$
using (\ref{basis}).
Finally, we obtain
\begin{equation}
\begin{array}{rcl}
\label{less}
\prod_{x\in A}\left|g_{i(x)}(x)\right|^{\mult(x,f)} &\leq& 2^{\tau'_2 \sum \mult(x,f)} 
\Mea(f)
^{n_2} \\ 
&\in & 2^{O(n_1 \log (n_2)+n_2 \tau_1+n_1 \tau_2+n_2 \log (n_1))},
\end{array}
\end{equation}

as $\sum_{x\in A} \mult(x,f)\leq n_1$, 
$\Mea(f) \leq 2^{\tau_1+\log (n_1+1)}$, and  $2^{\tau'_2}=(n_2+1)2^{\tau_2}$.\medskip

\noindent b) We want to prove that 
\begin{equation}
\label{greater}
\prod_{x\in  \mathrm{V}_{\mathbb{C}}(f)}\left|g_{i(x)}(x)\right|^{\mult(x,f)}\in 2^{-O(\tau_1 n_2)}.
\end{equation}
We set $$g (X,U)= g_1 (X)+ U g_2(X)+\cdots +U^{m-1} g_{m} (X)$$ and consider
$$\res_X(f,g)= \LCF(f)^{n_2}\cdot\prod_{x\in  \mathrm{V}_{\mathbb{C}}(f)}g(x,U)^{\mult(x,f)}.$$
 The polynomial $\res_X(f, g)$ is then a polynomial in $U$ with integer coefficients. The coefficient of its term of lowest degree, which is a non-zero  integer
 has absolute value equal to $$|\LCF(f)|^{n_2}\cdot \prod_{x\in  \mathrm{V}_{\mathbb{C}}(f)}\left|g_{i(x)}(x)\right|^{\mult(x,f)}.$$
In particular, this shows that the latter term has absolute value at least $1$.
It is further clear that 
$$|\LCF(f)|^{n_2}\in 2^{O(\tau_1 n_2)},$$
and thus (\ref{greater}) holds.

We now define $$A:=\lbrace x | f(x)=0, |g_{i(x)}(x)|\geq 1 \rbrace.$$ 
Since for any $x \in A,~\left|g_{i(x)}(x)\right| \geq 1,$ it is clear by (\ref{less}) that

\begin{equation}\label{+OA}
0\le \sum_{x \in A}\mult(x,f)\cdot \log (|g_{i(x)}(x)|)\le c \in O(n_1 \tau_2+n_2 \tau_1+n_2 \log (n_1)+n_1 \log (n_2))
\end{equation}

It follows by (\ref{greater})
that 
\begin{equation}\label{-OB}
\sum_{x\in  \mathrm{V}_{\mathbb{C}}(f)}-\mult(x,f) \cdot \log (|g_{i(x)}(x)|)\in O(n_2 \tau_1).
\end{equation}

Hence, 
we obtain
\begin{equation}\label{-OBB}
\sum_{x\in  \mathrm{V}_{\mathbb{C}}(f)} \mult(x,f) |\log (|g_{i(x)}(x)|)|\in O(n_2 \tau_1+n_1 \tau_2).
\end{equation}
\end{proof}

The following result has already been proven in \cite{KoS}
 but we give here a
 much
  simpler proof based on Proposition \ref{generalunivariate}.

\begin{proposition}
\label{corodisc}
Let $f\in\Z[X]$ be a 
non-zero
polynomial of magnitude
bounded by
 $(n,\tau)$, then
 $\logsep(f)\in\tO(n\tau+n^2).$
 \end{proposition}

\begin{proof}
Using Proposition \ref{generalunivariate} b) 
for the sequence $f^{[k]}$, $k=\deg(f)-1,\ldots,0$, we obtain
$$\logGdisc(f)\in\tO(n\tau+n^2).$$
Since  
$\logmea(f)\in O(\tau+\log (n))$
by
 (\ref{comparelog1}), we conclude by 
 Proposition \ref{GammaDeltaSigma}.
\end{proof}

Hereafter, we recall some quantitative and complexity results which will be used in the
complexity analysis of our algorithms.

\begin{proposition}
 \label{Mignotte}(see for example \cite{BPRbook2}) Let 
  $f\in \Z[X]$
  be a non-zero polynomial
  of magnitude
bounded by
 $(n,\tau)$ and $g\in \Z[X]$
 a non-zero polynomial 
  of degree $n_{1}$ dividing $f$.
 Then, $g$ is of magnitude 
bounded by
$(n_1,n_{1} + \tau + \log (n +
 1))$.
\end{proposition}

We need to give details about content and greatest common divisors (gcd) in the ring $\Z[X]$.

\begin{definition}
\label{def:contgcd}
Given  a non-zero-polynomial $f\in\Z[X]$
its
\emph{content} $\cont(f)$  is defined as the gcd in $\Z$ of the coefficients of $f$.  For every  $g\in \Z[X]$ dividing $f$, $\cont(g)$ divides $\cont(f)$ in $\Z$.

If $\cont(f)=1$, $f$ is content-free. Every  $g\in \Z[X]$ dividing a content-free polynomial is content-free.

Given 
two non-zero polynomials $f\in\Z[X]$ and 
$g\in\Z[X]$
their gcd $\gcd(f,g)$ is defined (up to sign) as the unique polynomial in $\Z[X]$ which is proportional to the gcd of $f$ and $g$ in $\Q[X]$ and is of content 
$\gcd(\cont(f),\cont(g))$.
 
 The content of the derivative $f'$ of $f$ is a multiple of the content of $f$, so $\cont(\gcd(f,f'))=\cont(f)$ and $f^\star=f/\gcd(f,f')$ its \emph{square-free part} is content-free.
\end{definition}

\begin{proposition}
 \label{gcd-comp}\cite{BLMPRS16,GG} Let $f, g \in \Z [X]$ be two 
 non-zero 
 polynomials of respective magnitude
bounded by
 $(n_{1},\tau_{1})$ and $(n_{2},\tau_{2})$. Computing
 their gcd has  bit complexity  $$\tO (\max(n_{1},n_{2})\cdot (n_{1}\tau_{2}+n_{2}\tau_{1})).$$
\end{proposition}

\begin{proposition}
 \label{exact_division_comp}\cite[Ex.~10.21]{GG}  Let $f\in \Z [X]$ 
 be a non-zero
 polynomial of
 magnitude
bounded by
 $(n,\tau)$. Given a non-zero polynomial $g \in \Z [X]$
 that divides $f$, computing the quotient of $f$ divided by $g$
 has bit complexity of $\tO (n\tau+n^2)$.
\end{proposition}

\begin{corollary}  \label{square_free}
 Let $f\in \Z [X]$
 be a non-zero
 polynomial of
 magnitude
bounded by
 $(n,\tau)$.  computing the square-free part $f^\star$ of $f$ 
 has bit complexity 
 $\tO (n^2\tau)$.
\end{corollary}

 \begin{notation}
 \label{bitsize}
 Let $p=r/s$ a rational number with $\gcd(r,s)=1$.
  We denote by $\lambda(p)$ its bitsize, defined by the sum of the bitsizes of its numerator $r$ and denominator $s$.
    For an interval $I=[a,b]$, $a<b$ with rational endpoints, we denote by  $|I|=b-a$ its length and by $\lambda(I)$ the maximum of $\lambda(a)$ and $\lambda(b)$.
   \end{notation}

\begin{proposition}
 \label{univariate-evaluation}\cite{BaZa,KS15} Let 
 $f \in \Z [X]$
  be a non-zero
 polynomial of magnitude 
bounded by
$(n,\tau)$.
 Let $r$ a rational number of bitsize $\lambda(r)$. Then, the evaluation of $f$ at $r$
 can be performed using $\tO (n (\tau + \lambda(r)))$ bit operations and the bitsize of
 the output $f (r)$ is $\tO (\tau + n\cdot \lambda(r))$.
\end{proposition}

\subsection{Root isolation}

We now focus on the problem of computing the roots of a given univariate polynomial. Here, we consider the two different but related problems of the computation of disjoint isolating regions and the approximation of the roots to a certain precision. Notice that isolating regions allow us to distinguish between two distinct roots, and thus also to determine the number of distinct roots. However, in general, isolating regions do not allow us to estimate the actual distance between two distinct roots as such regions might be considerably larger than the actual separation of the isolated root.

In order to overcome this issue, we are aiming for the computation of so called well-isolating regions from which we can derive a good estimate for the separation of a root $x$ or, more generally, for the distance from $x$ to any other root $y$.  
This good estimate plays a key role in the complexity for refining roots whenever needed.

\begin{definition}\label{well-isolating} Let $f \in \mathbb{C} [X]$ be a polynomial of degree $n$. Then, we define:
 \begin{itemize}
  \item[(a)] A \emph{well-isolating} interval  ${\mathcal I}=(a,b)$ for a real root $x$ of $f$  contains $x$, 
  and it holds that 
  $32 \cdot n \cdot |b-a|<\sep(x,f)$.
  \item[(b)] A \emph{well-isolating} disk ${\mathcal D}_{r}(m)=\{x\in \mathbb{C} \mid |x-m|\leq r\}$  for a complex root $x$ of $f$ contains $x$,
  and it holds that 
  $64 \cdot n \cdot r<\sep(x,f)$.
\end{itemize}
As a consequence a well-isolating interval (resp. disk) contains only one real (resp. complex) root of $f$.
\end{definition}

In what follows, we often have to deal with approximations of polynomials and to compute approximations of an exact value that a given polynomial takes at a certain point. The following definitions will turn out to be useful in order to specify these computations.

\begin{definition}\label{Lbitapproximation}
For a complex number $a\in\C$  and an integer $L$, we say that a dyadic Gaussian number of the form 
$$\tilde{a}={c}\cdot {2^{-L-1}}
+i\cdot {d}\cdot {2^{-L-1}}\in\Q+i\cdot \Q,$$ 
with $c,d\in\Z$, is an \emph{(absolute) $L$-bit approximation} of $a$ if $|a-\tilde{a}|<2^{-L}$.

For a polynomial $f=a_0+\cdots+ a_n\cdot X^n\in\C[X]$ with arbitrary complex coefficients and an integer $L$, we say that a polynomial $\tilde{f}=\tilde{a}_0+\cdots +\tilde{a}_n\cdot X^n$ 
is an 
(absolute) $L$-bit approximation
of $f$ if for every $i$, $\tilde{a}_i$ is an (absolute) $L$-bit approximation of $a_i$.
\end{definition}

The following result is useful.
  \begin{proposition}
 \label{univariate-evaluationbis}\cite{KS15} Let 
 $f \in \Z [X]$
  be a
non-zero
 polynomial of magnitude 
bounded by
$(n,\tau)$, $x \in\C$ be an arbitrary complex value, and $L$ be a positive integer. 

We can compute
a  dyadic approximation $\tilde{\beta}=
 {b}\cdot {2^{-L-1}}$ of $\beta:=f(x)$ , with $b\in\Z[i]$ and $|\beta-\tilde{\beta}|<2^{-L}$, using $\tO(n(L+n\cdot\log(\max(1, |x|))+\tau))$ bit operations. 
 \end{proposition}

The following propositions 
summarize
 the results on root isolation and approximation for a complex polynomials we use in the paper.

\begin{proposition}
 \label{sagraloff-isolation}(\cite[Thm.~4]{MSW2}\footnote{See also
 \cite{CIsolate,MSW1,Pa,SM16} for comparable results.}) Let
  $f \in \mathbb{C} [X]$
  be a
  non-zero polynomial of degree $n$ with 
  $1\le |\LCF(f)| \le 4$. 
  Suppose that the number $m$ of distinct roots of $f$ is given, then it holds:
	\begin{itemize}
	\item[(a)] Using a number of bit operations bounded by
	\begin{align}\label{complexityisol}
\tO\left(n\cdot(n^2+n\logmea(f)
+\logGdisc(f))\right)
\end{align}
	we can compute, for all $x\in \mathrm{V}_{\mathbb{C}}(f)$,
	the multiplicities $\mult(x,f)$ as well as well-isolating disks ${\mathcal D}_{r(x)}(m(x))\subset\C$  with dyadic centers $m(x)$ and dyadic radii $r(x)$ such that 
 the bitsizes of  all $m(x)$ and $r(x)$ sum up to 
 $$\tO(n+\logmea(f)+\logsep(f^\star_\mathbb{C})) .
 $$

 As input, we need an oracle giving an absolute $L'$-bit approximation of $f$, where $L'$ is bounded by 
\begin{align}\label{complexityprec}
\tO\left(n\logmea(f)+\logsep(f)+\logGdisc(f)\right).
\end{align}

\item[(b)] Let $\mathrm{V}^*\subset \mathrm{V}_{\mathbb{C}}(f)
$ be a subset of the roots of $f$, $\mu=\max_{x\in \mathrm{V}^*}\mult(x,f)$, and let $L$ be a given positive integer. Then, we can further refine the isolating disks ${\mathcal D}_{r(x)}(m(x))$ for all roots $x$ in $\mathrm{V}^*$ to a size less than $2^{-L}$ using 
\begin{align}\label{complexityref}
\tO\left(n\cdot\left(L\cdot\mu+n^2+n\logmea(f)
+\logGdisc(f)\right)\right)
\end{align}
bit operations.\footnote{\cite[Thm.~4]{MSW2} only provides a bound for the refinement of \emph{all} isolating disks (i.e. for $\mathrm{V}^*=\mathrm{V}_{\C}(f)$), however, from the proof of \cite[Thm.~4]{MSW2}, the claimed bound directly follows. In addition, in~\cite[Theorem 4]{MSW2}, the additive term $nL\cdot \mu$ appears in the bound on the needed input precision. We remark that this is a typo and that the actual bound is better by a factor $n$. The proof of~\cite[Theorem 4]{MSW2} clearly shows this fact.} As input, we need an oracle giving an absolute $L'$-bit approximation of $f$, where $L'$ is bounded by 
\begin{align}\label{complexityprec1}
\tO\left(L\cdot \mu+n\logmea(f)+\logsep(f)+\logGdisc(f)\right).
\end{align}
\end{itemize}
\end{proposition}

In the literature, the special case of an integer polynomial $f$ with coefficients of bitsize at most $\tau$ has attracted a lot of interest.
 The following result, which provides bounds on the  isolation of the roots as well on the problem of further refining the isolating disks, is an almost straight forward consequence of Proposition~\ref{sagraloff-isolation} (applied to the polynomial $f\cdot\LCF(f)^{-1}$).

\begin{proposition}
 \label{sagraloff-isolation-integer}\cite[Thm.~5]{MSW2} \footnote{See also
 \cite{CIsolate,MSW1,Pa,PT,SM16}} Let 
 $f \in \Z [X]$
  be a
  non-zero
  polynomial of magnitude 
bounded by
$(n,\tau)$.
 Using $\tO(n^2\tau+n^3)$ bit operations, one can compute
 \begin{itemize}
 \item[(a)] well-isolating disks ${\mathcal D}_{r(x)}(m(x))\subset\C$ for all complex roots $x$ of $f$ with dyadic centers $m(x)$ and dyadic radii $r(x)
 $ such that 
 the bitsizes of  all $m(x)$ and $r(x)$ sum up to $\tO (n \tau)$, and
 \item[(b)] the multiplicities $\mult(x,f)$ of each of the roots $x$.
  \item[(c)] For an arbitrary positive integer $L$, one can further refine all isolating disks to a size less than $2^{-L}$ using $\tO(n^2\tau+n^3+nL)$ bit operations.\footnote{Notice that, in contrast to the general case, where the coefficients of $f$ are not necessarily integers, the additional factor $\max_{x\in \mathrm{V}_{\C}(f)}\mult(x,f)$ is missing. This is due to the fact that, within the given complexity, we can first compute the square-free part $f^\star$ of $f$ and then work with $f^\star$ to refine the isolating disks.}
 \end{itemize}
\end{proposition}

Finally we can also identify common roots  of a polynomials $f$ with polynomials of a family $g_1,\ldots,g_m$.

\begin{proposition}\label{comparingroots}
Let 
$f,g_1,\ldots,g_m\in \Z [X]$
 be non-zero 
 polynomials of magnitudes
bounded by
 $(d,\tau)$, $(d_1,\tau_1),\ldots,(d_m,\tau_m)$, respectively, and let $N,\Lambda$ be positive integers such that $d+d_1+\cdots+d_m<N$ and $\tau+\tau_1+\cdots+\tau_m<\Lambda$. Then, 
we can isolate all roots of $f$ and all polynomials $g_i$, and identify all common roots of each  pair $(f,g_i)$ using no more than 
$\tO(N^2\Lambda+N^3)$ bit operations.
Within the same complexity, we can also determine the signs ($0$, $1$ or $-1$) of the $g_i$ at the real zeroes of $f$.
\end{proposition}

\begin{proof}
We 
first compute the square-free part $f^\star$
using 
$$\tO(d^2 \tau)\in \tO(N^2\Lambda)$$
 bit operations using 
 Corollary \ref{square_free}.
 The magnitude
 of  $f^\star$
 is bounded by $(d,O(d+\tau))$.
  Then, we may compute $h_i:=\gcd(f^\star,g_i)$ for all $i$ using 
 $$\tO(\sum\nolimits_{i=1}^m \max(d,d_i)\cdot(d\tau_i+d d_i+d_i \tau))\in \tO(N(d\Lambda+N\tau+dN)$$
bit operations due to Proposition~\ref{gcd-comp}. In the next step, we compute well-isolating disks for all complex roots of the polynomial $f^\star$ as well as for the complex roots of all polynomials $h_i$. We then refine the isolating disks for the roots of all polynomials $h_i$ to a size less than $\min_{x\in V_\mathbb{C}(f)}\sep(x,f)/4$. Since $\logsep(f)$ is bounded by
 $\tO(N\Lambda+N^2)$, this can be achieved using $\tO(N^2\Lambda+N^3)$ bit operations
 according to Proposition \ref{sagraloff-isolation-integer}.
 
 Notice that, after this refinement, each isolating disk ${\mathcal D}'$ for a root of $h_i$ intersects exactly one isolating disk ${\mathcal D}$ for a root of $f$, and thus ${\mathcal D}$ and ${\mathcal D}'$ isolate one common root. Hence, in order to identify common roots of $f$ and $g_i$, we just have to determine all intersections between the isolating disks for the roots of $h_i$ and those for the roots of $f$. For this, we first compute a lexicographic sorting of all centers of the isolating disks ${\mathcal D}$ for the roots of $f$, which uses $\tO(d^2 \tau)$ bit operations as we need $O(d\log d)$ many comparisons, each of precision $\tO(d\tau)$. Then, for a given isolating disk ${\mathcal D}'$ for a root of $h_i$, we can determine the unique disk ${\mathcal D}$ that intersects ${\mathcal D}'$ using 
 $\tO(N\Lambda+N^2)$ bit operations as the needed precision is bounded by
  $\tO(N\Lambda+N^2)$ 
  and only $O(\log d)$ many comparisons are needed. Hence, the total complexity for this step is also bounded by
   $\sum_{i=1}^m d_i\cdot\tO(N\Lambda+N^2)\in \tO(N^2\Lambda+N^3)$.
 
For the sign determination part, we simply compute sufficiently good $L$-bit approximations $\gamma_{i,x}$ of $g_i(x)$ for all $i$ and all real roots $x$ of $f$ with $g_i(x)\neq 0$. That is, we have to compute $L$-bit approximations $\gamma_{i,x}$ for $L=1,2,4,\ldots$ such that $|\gamma_{i,x}|>2^{-L}$, which then implies that $\operatorname{sign} g_i(x)=\operatorname{sign}\gamma_{i,x}$. Obviously, we succeed in doing so as soon as $L$ is larger than $|\log (|g_i(x)|)|$. Hence, for a specific real root $x$ of $f_i:=f^\star/\gcd(f^\star,g_i)$, the cost for this step is bounded by 
$\tO(d_i(|\log (|g_i(x)|)|+d_i\cdot\log(\max(1, |x|))+\tau_i))$ bit operations using 
 Proposition~\ref{univariate-evaluationbis}. 
The total cost is thus bounded by
\[
\sum_{i=1}^m d_i^2\cdot\sum_{x\in V_{\mathbb{R}}(f)}\log(\max(1,|x|))\quad+\quad N\cdot\sum_{i=1}^m d_i\tau_i \quad+\quad \sum_{i=1}^m d_i\sum_{x\in V_{\mathbb{R}}(f_i)}|\log( |g_i(x)|)|
\]
The first term is upper bounded by 
 $N^2\cdot\logmea(f)\in \tO(N^2\Lambda)$
, and the second term is upper bounded by $N\cdot\sum_{i=1}^m d_i\cdot\sum_{i=1}^m\tau_i\in O(N^2\Lambda)$. For the last term, we use Proposition~\ref{generalunivariate} b) to show that $\sum_{x\in V_{\mathbb{R}}(f_i)}|\log (|g_i(x)|)|\in \tO(N\Lambda+N^2)$ for all $i$, and thus also $\sum_{i=1}^m d_i\sum_{x\in V_{\mathbb{R}}(f_i)}|\log (|g_i(x)|)|\in \tO(N^2\Lambda+N^3)$.
\end{proof}

\section{Bivariate results}\label{bivariateresults}

Similar to our definition of the magnitude of a univariate polynomial, we introduce the following definitions for bivariate polynomials:

\begin{definition}
 A polynomial $F\in \Z[X,Y]$ is of \emph{magnitude}
bounded by
 $(n,\tau)$ if its
 total
  degree is bounded by $n$ and the 
 bitsize of each of its coefficients is bounded by 
 $\tau$.
\end{definition}

\subsection{Amortized bounds for bivariate systems}
\label{subsec:amort}
We consider an arbitrary polynomial $R\in\Z[X]$ of magnitude
bounded by
 $(N,\Lambda)$. Given
  $x\in \C$,
 we denote by $\mu(x)=\mult(x,R)$ its multiplicity as a root of $R$. 
  We further consider a bivariate polynomial
$$F(X,Y)=f_{n_y}(X)\cdot Y^{n_y}+\cdots+f_0(X) \in\Z[X,Y]$$ 
of magnitude
bounded by
 $(n,\tau)$, such that 
$$V_{\mathbb{C}}(R,F)=\{(x,y)\in \mathbb{C}\mid R(x)=F(x,y)=0\}$$
 is finite.

For a 
root $x$ of $R$, we denote by $n(x)$ the degree of $F(x,Y)$, which might be smaller than $n_y$ but at least $0$, since $V_{\mathbb{C}}(R,F)$ is finite. Notice that, for $n(x)\neq n_y$, we have $f_{n(x)}(x)\neq 0$ and $f_{n(x)+1}(x)=\cdots=f_{n_y}(x)=0$.
For a root 
$y$ of $F(x,Y)$, we denote in all Section \ref{subsec:amort} , with a slight abuse of notation,
$$\nu(y)=\mult(y,F(x,Y))$$
the multiplicity of $y$ as a root of $F(x,Y)$.

The following proposition provides amortized complexity bounds on the sum of lengths and  Mahler measures of the polynomials $F(x,Y)$.

\begin{proposition}\label{pro:bounds}
Let $R\in\Z[X]$ and $F\in\Z[X,Y]$ be 
non-zero
polynomials of magnitude
bounded by
 $(N,\Lambda)$ and  $(n,\tau)$ respectively such that $V_{\mathbb{C}}(R,F)$ is finite. Then, it holds that
\begin{equation}\label{pro:bounds1}
\sum_{x\in  \mathrm{V}_{\mathbb{C}}(R)}\mu(x)\cdot \loglen(F(x,Y)) \in \tO(N\tau+n\Lambda+n N)
\end{equation}
\begin{equation}\label{pro:bounds2}
\sum_{x\in  \mathrm{V}_{\mathbb{C}}(R)}\mu(x)\cdot\logmea(F(x,Y)) \in \tO(N\tau+n\Lambda+n N)
\end{equation}
\end{proposition}
\begin{proof}
For each root $x$ of $R$ in $\mathbb{C}$, we have 
$$2^{-n(x)}\cdot\Len(F(x,Y))\leq \Mea(F(x,Y))\leq \| F(x,Y)\|
$$ using Definition \ref{def} and according to (\ref{comparebis}). 
\\
Since $n(x)\leq n$ and  $\sum_{x\in  \mathrm{V}_{\mathbb{C}}(R)}\mu(x) = N$, it holds that $$\sum_{x\in  \mathrm{V}_{\mathbb{C}}(R)}n(x)\mu(x) \leq n N$$ and thus $$2^{-n N}\leq 2^{-\sum_{x\in  \mathrm{V}_{\mathbb{C}}(R)}n(x)\mu(x)}.$$
Let $\ell(x)$ be such that 
$$| f_{\ell(x)}(x)|=\max_{j=0,\ldots,n_y} | f_{j}(x)|.$$
We have $$ 
|f_{n(x)}(x)| \leq \Len(F(x,Y))\leq (n_y+1)\cdot | f_{\ell(x)}(x)|$$
$$ \| F(x,Y)\|\leq \sqrt{n_y+1} \cdot | f_{\ell(x)}(x)|,$$
hence 
denoting by 
\begin{eqnarray*}
A&=&\prod_{x\in  \mathrm{V}_{\mathbb{C}}(R)}
|f_{n(x)}(x)|^{\mu(x)}\\
B&=&\prod_{x\in  \mathrm{V}_{\mathbb{C}}(R)} | f_{\ell(x)}(x)|^{\mu(x)},
\end{eqnarray*}
$$A
\leq \prod_{x\in  \mathrm{V}_{\mathbb{C}}(R)} \Len(F(x,Y))^{\mu(x)}\leq (n_y+1)^N \cdot B
$$
and
$$2^{-nN}A
\leq \prod_{x\in  \mathrm{V}_{\mathbb{C}}(R)} \Mea(F(x,Y))^{\mu(x)}\leq \sqrt{n_y+1}^N \cdot B.
$$

The claims (\ref{pro:bounds1}) and (\ref{pro:bounds2}) follow by Proposition  \ref{generalunivariate}  applied to $R$ and the family $f_j$.
\end{proof}

Our next aim is to prove the following Proposition.

\begin{proposition}
\label{generalbivariate}
Let 
$R(X)\in  \Z[X]$ 
be a non-zero polynomial
of magnitude
bounded by
 $(N,\Lambda)$ and $F(X,Y)\in \Z[X,Y]$
be of magnitude 
bounded by
$(n_1,\tau_1)$, and suppose that  
$$V_{\mathbb{C}}(R,F)=\{(x,y)\in \mathbb{C}^2\mid R(x)=F(x,y)=0\}$$
is finite.
Let $G_1,\ldots,G_m$  in $\mathbb{C}[X,Y]$ be polynomials of 
magnitude bounded by  $(n_2,\tau_2)$.

\begin{itemize}
\item[(a)] Suppose that $A\subset V_{\mathbb{C}}(R,F)$ and that, for every $(x,y)\in A$, 
 $i(x,y)\in\{1,\ldots,m\}$ is such that $G_{i(x,y)}(x,y)\not=0$. Then, it holds that:
 \begin{equation}\label{boundgeneralmultivar-a}
 \sum_{(x,y) \in A}
  \mu(x)  
  \nu(y)
     \log (|G_{i(x,y)} (x, y)|) \in \tilde O((\tau_2n_1+\tau_1n_2)N+(\Lambda+N) n_1n_2).
    \end{equation}
\item[(b)]  Suppose that  $G_1,\ldots,G_m\in\Z[X,Y]$ and that, for every  $(x,y)\in V_{\mathbb{C}}(R,F)$,
there exists $i$ such that $G_i(x,y)\not=0$.
Denoting by $i(x,y)$ the smallest value of $i$ such 
 that $G_i(x,y)\not=0$, we have
 \begin{equation}\label{boundgeneralmultivar-b}
  \sum_{(x,y) \in V_{\mathbb{C}}(R,F)} \mu(x)  
  \nu(y)
    | \log (|G_{i(x,y)} (x, y)|) |\in \tilde O((\tau_2n_1+\tau_1n_2)N+(\Lambda+N) n_1n_2).
    \end{equation}
  where  $$V_{\mathbb{C}}(R,F)_{x}:=\{y\in \mathbb{C}\mid F(x,y)=0\}.$$
    \end{itemize}
\end{proposition}
\begin{proof}
\noindent a) 
First note that,
$$|G_{i(x,y)}(x,y)|\le
2^{\tau'_2}\max(1,\vert x \vert)^{n_2}
\max(1,\vert y \vert)^{n_2},$$
where $\tau'_2:=\lceil \log( (n_2+1)^2 2^{\tau_2})\rceil\in O(\tau_2+\log (n_2))$.
Further notice that 
\begin{equation}
\label{b-coef-Mea}
\prod_{(x,y)\in A} 
2^{\tau'_2 \mu(x)}\le \prod_{x\in  \mathrm{V}_{\mathbb{C}}(R)} 2^{\tau'_2 n_1 \mu(x)} \in 2^{ O((\tau_2+\log (n_2)) n_1N)}
\end{equation}
and
\begin{eqnarray}
\label{Meal-Sy}
|\LCF(R)^{n_1n_2}| \prod_{(x,y)\in A} \max(1,\vert x\vert)^{\mu(x)
\nu(y) 
n_2}& \le &
\Mea(R)^{n_1n_2}\in 2^{O((\Lambda+\log( N)) n_1 n_2)}.
\end{eqnarray}

Hence,
since
\begin{equation}
 \prod_{x\in  \mathrm{V}_{\mathbb{C}}(R)} 
|\LCF(F(x,Y))^{\mu(x)}|
\prod_{y\mid (x,y)\in A}
 \max(1,\vert y \vert)^{\nu(y)\mu(x)}\leq 
 \prod_{x\in  \mathrm{V}_{\mathbb{C}}(R)} \Mea(F(x,Y))^{\mu(x)}
\end{equation}
it follows that
\begin{equation}
\label{Meal-Sx}
\prod_{(x,y)\in A}
 \max(1,\vert y \vert)^{\mu(x)\nu(y) n_2}\in 2^{\tilde O((\tau_1 N+\Lambda n_1+n_1 N)n_2)}
\end{equation}
by 
Proposition  \ref{pro:bounds} and Proposition \ref{generalunivariate}.
We thus conclude that
\begin{equation}
\label{b-p}
 \prod_{(x,y)\in A} |G_{i(x,y)}(x,y)|^{\mu(x)\nu(y)}\in 2^{\tilde O((\tau_2 n_1+\tau_1 n_2)N+(\Lambda+N) n_1n_2)}
\end{equation}
and
\begin{equation}
\label{valeur-N}
\sum_{(x,y)\in A} \mu(x) \nu(y)  \log (|G_{i(x,y)} (x,y)|) \in \tilde O((\tau_2 n_1+\tau_1 n_2)N+(\Lambda+N) n_1n_2).
\end{equation}
  
\noindent 
b) In the first step, we aim to prove that 
$$\prod_{(x,y)\in V_{\mathbb{C}}(R,F)} |G_{i(x,y)}(x,y)|^{\mu(x)\nu(y)}\ge  \frac{1}{E''}$$
where $E''$ is a natural number of
bitsize $O(\Lambda n_1n_2+\tau_1 n_2 N)$.

Let as before
$$F(X,Y):=f_{n_y}(X)Y^{n_y}+\ldots+f_0(X),$$ with $n_y=\deg_Y(F)\le n_1$, and $R^\star$ the square-free part of
$R$ which is content-free (cf Definition \ref{def:contgcd}).

We define the polynomial sequence $(R^\star_{\ell} (X))_{\ell\in[1,n_y]}\in \Z[X],$
such that
\begin{equation}
\label{fixdegree}
 \deg (F (x, Y)) =\ell
 \Longleftrightarrow R^\star_{\ell} (x) = 0, 
 \end{equation}
 as follows.
 We first compute a family of polynomials $R^\star_{\ell}$ whose zeroes characterize
$\deg_Y(F(x,Y))=\ell$.
 \begin{equation}
 R^\star_{\le n_y} (X) :=R^\star(X),\label{factorisation-degree1}
  \end{equation}
 and
  \begin{equation}
 R^\star_{\le \ell-1} (X) := \gcd ( R^\star_{\le \ell}
 (X), f_{\ell}(X)),\quad  R^\star_{\ell-1} (X) := \frac{ R^\star_{\le \ell} (X)}{ R^\star_{\le \ell-1}
 (X)}.\label{factorisation-degree2}
  \end{equation}
for all $\ell \in \{n_y, \ldots, 1\}.$ 
Notice that
\begin{eqnarray}
R^\star&=&\prod_{\ell \in \{0,\ldots,n_y\}} R^\star_{\ell} (X)
\end{eqnarray}

 We further define the following content-free polynomials in $\Z[X]$.
 \[ R^\star_{\ell,>1} (X) :={\gcd(R^\star_\ell,R')}, R^\star_{\ell,1} (X):=\frac{R^\star_\ell}{R^\star_{\ell,>1}(X)} \]
 and, 
 for all $i \in \{1, \ldots, N \},$ $$R^\star_{\ell,>i} (X) := \gcd (R^\star_{\ell,>i-1}
 (X), R^{[i]}(X)), \quad R^\star_{\ell,i} (X) := \frac{R^\star_{\ell,>i-1} (X)}{R^\star_{\ell,>i}
 (X)}.$$
It is clear using Equation (\ref{fixdegree}) that
 $\deg (F (x, Y)) =\ell$ and $x$ is a root of multiplicity $\mu$ of $R$ if and only if
$R^\star_{\ell,\mu} (x) = 0.$

Notice that
\begin{eqnarray}
R^\star&=&\prod_{\ell,\mu} R^\star_{\ell,\mu}\\
R&=&\cont(R) \prod_{\ell,\mu} (R^\star_{\ell,\mu})^\mu
\end{eqnarray}

Let 
\begin{equation} \label{ref:trunc}
F_\ell(X,Y):=f_\ell(X)Y^\ell+\ldots+f_0(X),
\end{equation}
and
$$G=G_{1}+U G_{2}+\ldots+U^{m-1}G_m.$$ 
We further define
$$Z_{\ell,\mu}:=V_{\mathbb{C}}(R^\star_{\ell,\mu},F) \subset V_{\mathbb{C}}(R,F)$$
and 
$$A_{\ell,\mu}(U):=\res_X(\res_Y(F_\ell,G),R^\star_{\ell,\mu})\in \Z[U].$$
Let $$Q_\ell(U,X):=\res_Y(F_\ell,G)$$ and 
notice that
$$Q_\ell(U,x)=f_\ell(x)^{O(n_2)} \prod_{y\mid (x,y)\in Z_{\ell,\mu}} G(x,y,U)^{\nu(y)}.$$
Denoting by $\delta_\ell \le n_1n_2$ the degree of $Q_\ell (U,X)$ with respect to $X$ and $n_\ell$ the degree of $f_\ell$ with respect to $X$,
$$A_{\ell,\mu}(U)=\LCF(R^\star_{\ell,\mu})^{\delta_\ell
}\prod_{x\mid R^\star_{\ell,\mu}(x)=0}  f_\ell(x)^{O(n_2)}  \prod_{(x,y)\in Z_{\ell,\mu}} G(x,y,U)^{\nu(y)}
$$
The coefficient of the term of smallest degree in $U$ of $A_{\ell,\mu}(U)$ is a non-zero integer and equal to
$$\LCF(R^\star_{\ell,\mu})^{\delta_\ell-n_2 n_\ell}\res_X(f_\ell,R^\star_{\ell,\mu})^{n_2}\prod_{(x,y)\in Z_{\ell,\mu}} G_{i(x,y)}(x,y)^{\nu(y)}.$$
Since ${\delta_\ell-n_2 n_\ell}\in O(n_1n_2)$, 
$\cont(R) \prod_{\ell,\mu} \LCF(R^\star_{\ell,\mu})^\mu= \LCF(R)\in 2^{O(\Lambda)}$, and $$\prod_{\ell,\mu}\res_X(f_\ell,(R^\star_{\ell,\mu})^\mu)=\res_X(f_\ell,R)\in 2^{O(\tau_1 N+\Lambda n_1)},$$ 
we conclude that
\begin{equation}
\label{alphagamma}
\prod_{(x,y)\in V_{\mathbb{C}}(R,F)} |G_{i(x,y)}(x,y)|^{\mu(x)\nu(y)}\ge \frac{1}{E''}
\end{equation}
with $E'':=|\LCF(R)^{n_1n_2} \res_X(f_\ell,R)^{n_2}|\in 2^{O(\Lambda n_1n_2+\tau_1 n_2 N)}$.

Finally, 
we have
 \begin{equation}
 \label{b-n}
\sum_{(x,y)\in  \mathrm{V}_{\mathbb{C}}(R,F)} -\mu(x)\nu(y)\log( |G_{i(x,y)}(x,y)| ) 
\in O(\Lambda n_1n_2+\tau_1 n_2 N).
 \end{equation}

Defining  $A:=\{(x,y)\in V_{\mathbb{C}}(R,F)\mid |G_i(x,y)|\ge 1\},$
and using (\ref{b-p}), (\ref{b-n}) and  (\ref{valeur-N}),
we obtain
finally
\begin{equation}
\label{bit-Abar2}
\sum_{x\in  \mathrm{V}_{\mathbb{C}}(R,F)} \mu(x)\nu(y)| \log (|G_{i(x,y)}(x,x)| )|  \in \tilde O((\tau_2n_1+\tau_1n_2)N+(\Lambda+N) n_1n_2)
\end{equation} 
\end{proof}

 The following proposition provides amortized complexity bounds on the sum of the Mahler measures of the polynomials $F(x_i,Y)$, the separators of the roots $x_{i,j}$ as well as the absolute values of the first non-vanishing derivatives of $F(x_i,Y)$ at the roots $x_{i,j}$:

We denote 
\begin{equation}\label{Fk}
F^{[k]}=\frac{\partial_Y^{k}F}{k!}.
\end{equation}

\begin{proposition}\label{thm:bounds}
Let 
$R(X)\in  \Z[X]$
be a non-zero polynomial
of magnitude 
bounded by
$(N,\Lambda)$ and $F(X,Y)\in \Z[X,Y]$
be of magnitude 
bounded by
$(n,\tau)$, and suppose that  
$$V_{\mathbb{C}}(R,F)=\{(x,y)\in \mathbb{C}^2\mid R(x)=F(x,y)=0\}$$
is finite.
Then, it holds that
\begin{itemize}
\item[(a)] $\sum_{x\in  \mathrm{V}_{\mathbb{C}}(R)} \mu(x)\cdot \logGdisc(F(x,Y))\in
\tO(n(N\tau+n\Lambda+Nn)).$
\item[(b)] $\sum_{x\in  \mathrm{V}_{\mathbb{C}}(R)}\mu(x)\cdot\logsep(F(x,Y)) \in \tO(n(N\tau+n\Lambda+Nn)),$
\end{itemize}
\end{proposition}

\begin{proof}
(a) is an immediate consequence of Proposition \ref{generalbivariate}
applied to the family $F^{[k]}$.
(b) follows directly from (a), Proposition~\ref{GammaDeltaSigma} and Proposition \ref{pro:bounds}.
\end{proof}

\subsection{Fixing the degree and the 
number of distinct roots
in the fibers}
We first  give results on the multiplicities of the roots of $F(x,Y)$ at the $X$-critical points of $F$. 

Consider a bivariate polynomial
\begin{equation}
F(X,Y)=f_{n_y}(X)\cdot Y^{n_y}+\cdots+f_0(X) \in\mathbb{C}[X,Y].
\end{equation} 
Notice that
$$V_{\mathbb{C}}(F)=\{(x,y)\in \mathbb{C}^2\mid F(x,y)=0 \}$$
 contains no vertical line if and only if the polynomials $f_i(X)$, $i=0,\ldots,n_y$, do not share a common non-trivial factor.
 In this case, let $$\Crit_X(V_{\mathbb{C}} (F)) := \{(x, y) \in \mathbb{C}^2 \mid F (x, y) =\partial_Y F(x,y)= 0\}$$ 
 be the set of $X$-critical points of $F$.

    \begin{proposition}\label{multmult0}
 Let
  $F\in\mathbb{C}[X,Y]$
   be   such that $V_{\mathbb{C}}(F)$ contains no vertical line.
  Given $x\in \mathbb{C}$,
      \begin{equation}\label{mult5}
      \sum_{x\mid (x,y)
      \in \Crit_X (V_{\mathbb{C}} (F))}  (\mult(y, F_{n(x)}(x,Y))-1)
  \le \mult(x,\Disc_Y(F_{n(x)})).
  \end{equation} 
  \end{proposition}
    The proof of  Proposition \ref{multmult0}  is based on the following lemma.

  \begin{lemma}
  \label{mult1bis}Let $F$ and $G$ be two 
  non-zero
   bivariate polynomials. 
    Given $x\in \mathbb{C}$ such that  $F(x,Y)$ and $G(x,Y)$ are non identically zero,
      \begin{equation}\label{mult5bis}
      \deg(\gcd(F(x,Y),G(x,Y)))
  \le \mult(x,\res_Y(F,G)).
  \end{equation} 
  \end{lemma}

 \begin{proof} 
The claim follows clearly from the two following 
statements.

 a) Let $f,g$ be two univariate polynomials of respective degrees 
  $p,q$.
Let $\phi$ be the mapping from $K_{< q} [X] \times K_{< p} [X]$ to $K_{<
  p+q} [X]$ sending $(U, V)$ to $Uf + Vg$. Then $\Ima(\phi)$ is the set of multiples of the greatest common divisor $h=\operatorname{gcd}(f,g)$ of $f$ and $g$, and the rank of $\phi$ is  $p+q-\deg(h)$. It is clear that $(m\cdot g/h,-m\cdot f/h)$ is in the kernel of $\phi$ for every polynomial $m$ of degree $<\deg(h)$. In the other direction, every $(U,V)$ in the kernel of $\phi$ is such that there exists $m$ of degree $<\deg(h)$ with $U=m\cdot g/h$ and $V=-m\cdot f/h$. This implies that the dimension of $\Ker(\phi)$ is equal to $\deg(h)$ and  the dimension of $\Ima(\phi)$ is thus equal to $p+q-\deg(h)$. But every element of $\Ima(\phi)$ is a multiple of $h$, and the vector space of multiples of $h$ of degree $<p+q$ is also of dimension $p+q-\deg(h)$. It follows that  $\Ima(\phi)$  coincides with the set of multiple of $h$, and the rank of $\phi$ is  $p+q-\deg(h)$.

b) Let $M(X)$ be an $n\times n$ matrix with coefficients in $K[X]$. If the rank of $M(x_0)$ is equal to $n-k$, then $x_0$ is a root of $\det(M(X))$ of multiplicity at least $k$. The proof   is by induction on $k$.
     
     If $k=0$ the statement is true.
     
     If $k>0$, then $M(x_0)$ is not invertible, and
     $\det(M(x_0))=0$.
  Denote by   $m_{i,j}(X)$ the $(i,j)$-th entry of $M(X)$,  and by $M_{i,j}(X)$ the $(n-1)\times (n-1)$-matrix obtained by removing thee $i$-th row and $j$-th column from $M(X)$.
     The rank of $M_{i,j}(x_0)$ is
     $r(x_0) \le n-k$, hence by induction hypothesis $x_0$ is a root of $\det(M_{i,j}(X))$ of multiplicity at least equal to $(n-1)-r(x_0) \ge k-1$.
According to Jacobi's formula, we have 
     $$\frac{d}{dX}\det(M(X))=\sum_{i,j} (-1)^{i+j} m'_{i,j}(X) \det(M_{i,j}(X)),$$
and thus the claim follows by induction since $x_0$ is a root of $\det(M(X))$ and a root of
     multiplicity at least $k-1$ of its derivative.
  \end{proof}

\begin{proof}[Proof of Proposition \ref{multmult0}]
 Use Lemma  \ref{mult1bis} with $F=F_{n(x)}$, $G=\partial_YF_{n(x)}(X,Y)$ noting that 
   \begin{equation}\label{mult6}
      \sum_{y \mid (x,y)
      \in \Crit_X(V_{\mathbb{C}} (F))}  (\mult(y, F_{n(x)}(x,Y))-1)
 =\deg(\gcd(F_{n(x)}(x,Y),\partial_Y F_{n(x)}(x,Y))).
  \end{equation}
  \end{proof}
  
      \begin{proposition}\label{multmult1}
      Let $F\in \Z[X,Y]$
      be a non-zero
       square-free polynomial of magnitude
bounded by
$(n,\tau)$ such that $V_{\mathbb{C}}(F)$ has no vertical line.
    There exists a polynomial $R_Y\in \Z[Y]$ of magnitude
bounded by
 $(O(n^2),O(n\tau+n^2))$ such that
   given $y\in \mathbb{C}$
  \begin{equation}\label{mult1}
\sum_{x \mid (x,y)
\in \Crit_X (V_{\mathbb{C}} (F))}  (\mult(y
, F_{n(x)}(x,Y))-1)\le \mult(y,
R_Y).
  \end{equation}
In particular, the zeroes of $R_Y$ contain the projection of $\Crit_X(V_{\mathbb{C}} (F))$ on the $Y$-axis.
  \end{proposition}
  
\begin{proof}
a)  We suppose first that there is only one $x$ such that $(x,y)\in \Crit (V_{\mathbb{C}} (F))$, $\mult(y
, F_{n(x)}(x,Y))> 1$ and denote $\mu=\mult(y
, F_{n(x)}(x,Y))$.
Defining
\begin{equation}\label{UV}
R_Y(Y):=\res_X(F,\partial_YF)(Y),
\end{equation}
there exist $U(X,Y), V(X,Y)$ such that
\begin{equation}\label{UV1}
R_Y(Y)=U(X,Y)F(X,Y)+V(X,Y)\partial_YF(X,Y).
\end{equation}
Derivating (\ref{UV}) $1,\ldots,\mu-2$ times with respect to $Y$ and using that
$$F(x,y)=\partial_YF(x,y)=\ldots=\partial_Y^{(\mu-1)}F(x,y)=0$$
we obtain
\begin{equation}\label{multimult}
R_Y(y)= \partial_Y R_Y(y)=\ldots= \partial_Y^{(\mu-2)}R_Y(y)=0.
\end{equation}

b) We reduce the general case to the preceding situation by the change of variable
$Y\leftarrow Y+\eps X$ where $\eps$ is a new variable. We use for the proof the field $\mathbb{R}\langle \eps \rangle$ of algebraic Puiseux series \cite{BPRbook2}, which is a real closed field containing $\mathbb{R}(\eps)$, totally ordered with the order $0_+$ making $\eps$ positive and smaller than any positive element of $\mathbb{R}$. We denote
$\mathbb{C}\langle \eps \rangle ^2=\mathbb{R}\langle \eps \rangle ^2[i]$.

Notice that each $X$-critical point
$(x,y)$ of $F$ in $\mathbb{C}^2$ yields a $X$-critical point $(x,y-\eps x)$ of $\tilde{F}(X,Y):=F(X,Y+\eps X)$ in $\mathbb{C}\langle \eps \rangle ^2$ , and 
that 
the multiplicity of   $y-\eps x$ as a root of $\tilde{F}(x,Y)$ 
coincides with the
the multiplicity $\mu(x,y)$ of  $y$ as a root of $F(x,Y)$ 
as
$\partial_Y^{(i)}\tilde{F}(x,y)=\partial_Y^{(i)}F(x,y-\eps x)$. Moreover there are no two distinct critical points of $\tilde{F}(x,y)$ sharing
the same $y$-coordinate. Hence, it holds by a) that 
$y-\eps x$ is a root of $\res_X(\tilde F,\partial_Y \tilde F)(Y)$ of multiplicity at least 
$\mu(x,y)$.
Then
$$\res_X(\tilde F,\partial_Y \tilde F)(Y)=A(\eps)\tilde R(Y,\eps) $$
with $A(\eps)\in \mathbb{C}(\eps)$,
 $\tilde R(Y,\eps)$ monic in $Y$, 
$$\tilde R(Y,\eps)=\prod_{(x,y)\in \Crit (V_{\mathbb{C}} (F))}(Y-y+\eps x)^{\mu(x,y)}B(Y,\eps),$$
with 
  \begin{equation}\label{mult1fin}
\mult(y-\eps x
, \tilde F_{n(x)}(x,Y))-1\le \mu(x,y),
  \end{equation}
  and 
  $B(Y,\eps)\in \mathbb{C}[X,\eps]$ such that   $B(Y,0)\in \C[X]$ is a non-zero polynomial.
  Hence, denoting by $$\nu(y)=\sum_{x \mid (x,y)\in \Crit (V_{\mathbb{C}} (F))}\mu(x,y),$$
  $$R_Y:=\tilde R(Y,0)=\prod_{
 y\in \pi_Y(\Crit (V_{\mathbb{C}} (F))) }
 (Y-y)^{\nu(y)}B(Y,0),$$
  we have 
    \begin{equation}\label{mult1next}
\sum_{x \mid (x,y)\in \Crit (V_{\mathbb{C}} (F))}  (\mult(y
, F_{n(x)}(x,Y))-1)\le \nu(y)\le \mult(y,
R_Y).
  \end{equation} 
  Finally, notice that $R_Y$ is of magnitude
bounded by
 $(O(n^2),O(n\tau+n^2))$.
\end{proof}

\begin{remark}
When $\deg_X(F)\ge \deg_Y(F)$, it turns out that $$R_Y(Y)=\res_X(F,\partial_YF)(Y)$$ because the Sylvester-matrix of $F,\partial_YF$ and $\tilde F,\partial_Y \tilde F$ have the same dimension, since $\deg_X(F)=\deg_X(\tilde F)$.
However, when $\deg_X(F)<\deg_Y(F)$, it 
can
happen that $R_Y(Y)\not=\res_X(F,\partial_YF)(Y)$.

This is the case for example if
$$F(X,Y)=(X-Y^2+1)(X-Y^2-1)(XY^2-1)$$ since we have $$R_Y(Y)=2048(Y^4 + Y^2 - 1)^2(Y^4 - Y^2 - 1)^2Y^2,$$  
$$\res_X(F,\partial_YF)(Y)=-32 (Y^4 + Y^2 - 1)^2(Y^4 - Y^2 - 1)^2Y^3.$$
\end{remark}.

We now give details on how to determine  the degree of $F(x,Y)$ and its number of distinct complex roots
for a root $x$ of $R$. Our aim is to prove the following proposition:

\begin{proposition}\label{prop:computinggcddeg}
Let $R(X)\in  \Z[X]$
a non-zero polynomial of magnitude
bounded by
 $(N,\Lambda)$ and $F(X,Y)\in \Z[X,Y]$
be of magnitude
bounded by
 $(n,\tau)$, and suppose that  
$$V_{\mathbb{C}}(R,F)=\{(x,y)\in \mathbb{C}^2\mid R(x)=F(x,y)=0\}$$
is finite.
  Computing $$n(x):=\deg(F(x,Y)),k(x):=\deg
  (\gcd (F (x, Y), \partial_Y (F(x, Y))$$ for every root $x \in  \mathbb{C}$ of $R$   
  has bit complexity
   $$\tO (n\max(N,n^2)(N\tau+n\Lambda+N n)+n^5 \tau + n^6).$$
\end{proposition}

Note that $n(x)-k(x)$ is the number of distinct complex roots of $F(x,Y)$.

We 
start by estimating
 the cost of computing the polynomials $R^\star_{\ell}$ 
defined in Equation (\ref{factorisation-degree1}) and Equation (\ref{factorisation-degree2}).

  \begin{lemma}\label{lemmacompdeg}
  The computation of all the polynomials $$( R^\star_{\ell} (X))_{\ell\in[1,n_y]}$$ uses 
$\tilde O(\max(n,N)( N\tau+n \Lambda+Nn)+n^3\tau+n^4)$ bit operations.
  \end{lemma}
  \begin{proof} 
Since $R^\star$ is of degree at most $N$ and bitsize bounded by $\Lambda+N$, and $f_{n_y}(X)$ is of degree
at most $n$ and bitsize  $O(\tau)$, the computation of their gcd needs $\tilde O(\max(n,N)(N\tau+n\Lambda+Nn))$ bit operations according to Proposition \ref{gcd-comp}.
Since, for all $\ell\in [0,n_y-1]$, $\deg( R^\star_{\le \ell})\le n$ and the bitsize of the coefficients of the $ R^\star_{\le \ell}$ is at most $n+\tau$,
the complexity of computing all $( R^\star_{\le \ell})_{\ell\in[0,n_y-1]}$ is in $O(n^3\tau+n^4)$.

It remains to compute the  $R^\star_{\ell}$ themselves by performing the exact division of $ R^\star_{\le \ell}$ by $R^\star_{\le \ell-1}$. This takes $O(N\Lambda+N^2)$ binary operations for $\ell=n_y$ and 
$O(n\tau+n^2)$ binary operations for each $\ell<n_y$.
\end{proof}

The proof of Proposition \ref{prop:computinggcddeg} uses a family of polynomials $\Delta^\star_{\ell,k}$
defined as follows.
We further denote 
$\sDisc_{\ell,k} (X)$ the  $k$-th subdiscriminant of $F_\ell$ considered as a polynomial in $Y$.
 We also define the content-free elements of $\Z[X]$ by:
 \[  \Delta^\star_{\ell,\geq 0} (X) :=  R^\star_{\ell} ,\]
 and  for all
 $k \in \{ 0, \ldots, \ell - 1 \},$ 
  \begin{equation}
\Delta^\star_{\ell,\geq k+1} (X) := \gcd ( \Delta^\star_{\ell,\geq k}
 (X), \sDisc_{\ell,k} (X)),\quad
  \Delta^\star_{\ell,k} (X) := \frac{ \Delta^\star_{\ell,\geq k} (X)}{ \Delta^\star_{\ell,\geq k+1}
 (X)}.\label{factorisation-degree-gcd1}
  \end{equation}

Given $\ell \in[1,n_y]$ and $x \in \mathbb{C}$ such that $\deg_Y(F (x, Y)) = \ell$ (i.e. $R^\star_{\ell}(x)=0$), it is clear that, using Notation \ref{ref:trunc},
  \label{compmult}
 \[ \deg_Y(\gcd (F_{\ell} (x, Y), \partial_Y F_{\ell} (x, Y))) \geq  k
 \Longleftrightarrow  \Delta^\star_{\ell,\geq k} (x) = 0. \]
 \begin{equation}
 \deg_Y(\gcd (F_{\ell}(x, Y), \partial_Y F_{\ell} (x, Y))) = k
 \Longleftrightarrow \Delta^\star_{\ell,k} (x) = 0. \label{compmult2}
 \end{equation}

 \begin{lemma}\label{computesubres}
  The computation of all the polynomials $$(\sDisc_{\ell,k} (X))_{\ell\in[1,n_y],k \in [0, \ell - 1
  ]}$$ uses $\tO (n^5 \tau)$ bit operations. 
  \end{lemma}
  
  \begin{proof} The claim follows clearly from
 \cite[Prop.~8.46]{BPRbook2},\cite[\S 11.2]{GG} 
  since there are $n_y\le n$ lists of sudiscriminants (i.e. subresultant coefficients) to compute.
\end{proof}

   \begin{lemma}\label{pro:degrees}
The computation of all the polynomials $(\Delta^\star_{\ell,k} (X))_{\ell\in[1,n_y],k \in [0, \ell - 1]}$ uses a number of bit operations bounded by 
      $$\tO (n\max(N,n^2)(N\tau+n\Lambda+N n)+n^5 \tau + n^6).$$
  \end{lemma}

   \begin{proof} 
   We compute first  all the polynomials $$(\sDisc_{\ell,k} (X))_{\ell\in[1,n_y],k \in [0, \ell - 1
  ]}$$ using Lemma \ref{computesubres}.
   
   Let $\delta_{\ell,k}$ be the degree of $\Delta^\star_{\ell,\geq k}$, and let $\tau_{\ell,k}$ be the maximal bitsize of the coefficients of $\Delta^\star_{\ell,\geq k}$.
   
   The roots of $\Delta^\star_{n_y,0}$  are exactly the roots $x$ of $R$ with
   $$\deg_Y(F(x,Y))=n_y\text{ and } \deg_Y(\gcd (F(x, Y), \partial_Y F (x, Y))) = 0.$$
  The computation of $\Delta^\star_{n_y,0}$ needs $\tO (\max(N,n^2)(Nn\tau+n^2\Lambda+N n^2))$ bit operations according to Proposition \ref{gcd-comp} and  Proposition \ref{exact_division_comp}.
   
  From Proposition~\ref{multmult0}, we conclude that each root of $\Delta^\star_{n_y,\geq k}$ for $k>0$ is also a root of $\Disc_Y(F)$ of multiplicity at
  least $k$. Since $\Delta^{\star}_{n_y,\geq k}$ divides $R^\star$, it is square-free and $(\Delta^{\star}_{n_y,\geq k})^{k}$ divides $\Disc_Y(F)$, and thus
  $$\delta_{n_y,k} \leqslant \frac{n (2 n - 1)}{k},\quad\text{and } \sum_{k = 1}^{n_y}
  \delta_{n_y,k} \in \tO (n^2) .$$
  Moreover,
(\ref{comparelog30}) yields that $\tau_{n_y,k}
  \leqslant \logmea (\Delta^\star_{n_y,\geq k}) + \delta_{n_y,k}$. Since 
   $\Delta^{\star}_{n_y,\geq k}$
   divides
  $\Disc_Y(F)$ and since the Mahler measure is multiplicative (see \ref{multiplicative}), we have $$\sum_{k = 1}^{n_y} \logmea (\Delta^\star_{n_y,\geq k}) \leqslant \logmea (\Disc_Y(F)) \in \tO (n
  \tau),$$ and thus 
  $$\sum_{k = 1}^{n_y} \tau_{n_y, k} \in \tO ( n \tau+n^2).$$
  
   On the other hand, for $\ell \in[1,n_y-1]$,  $\Delta^\star_{\ell,\geq k}$ is a divisor of $R^\star_\ell$, so that
   $$\sum_{k = 0}^\ell \delta_{\ell,k} \le (\ell+1) \cdot \deg(R^\star_{\ell}),\sum_{\ell=1}^{n_y-1}\deg(R^\star_{\ell})\le n_y\le n,$$ 
 and
    $$\sum_{\ell=1}^{n_y-1} \sum_{k = 0}^\ell
  \delta_{\ell,k} \in \tO (n^2).$$
  As above, we use 
  (\ref{comparelog30}) to show that
  $\tau_{\ell,k}
  \leqslant \logmea (\Delta^\star_{\ell,\geq k}) + \delta_{\ell,k}$. Since $\Delta^\star_{\ell,\geq k}$ is a divisor of $R^\star_\ell$, which is a divisor of
  $f_{\ell}$, and $$\sum_{k =0}^\ell \logmea (\Delta^\star_{\ell,\geq k}) \leqslant \logmea (f_\ell) \in \tO (\tau),$$ we conclude that
  $$\sum_{\ell=1}^{n_y-1}\sum_{k = 0}^\ell \tau_{\ell,k}\in \tO (n^2 + n \tau).$$
  
    Finally
\begin{equation}\label{phi}
\sum_{k = 1}^{n_y}
  \delta_{n_y,k}+\sum_{\ell=1}^{n_y-1} \sum_{k = 0}^\ell
  \delta_{\ell,k} \in \tO (n^2)
  \end{equation}
  \begin{equation}\label{tau}
 \sum_{k = 1}^{n_y-1}
  \tau_{n_y,k}+\sum_{\ell=1}^{n_y-1} \sum_{k= 0}^\ell
\tau_{\ell,k} \in \tO ( n \tau+n^2) .
  \end{equation}
  
  The computation of $\Delta^\star_{\ell,\geq k+1} (X)=\gcd (\Delta^\star_{\ell,\geq k} (X), \sDisc_{\ell, k} (X))$ for $(\ell,k)\not=(n_y,0)$ uses
  $$ \tO (\max (\delta_{\ell,k},n^2)(n^2 \tau_{\ell,k }+\delta_{\ell,k} n \tau))\in \tO (n^2(n^2 \tau_{\ell,k }+\delta_{\ell,k} n \tau))$$ bit operations
according to Proposition \ref{gcd-comp}.
  Finally, the computation of all
  the $\Delta^\star_{\ell,\geq k}$ needs 
  $$\tO (\max(N,n^2)(Nn\tau+n^2\Lambda+N n^2)+n^5 \tau + n^6)$$ bit operations using (\ref{phi}) and (\ref{tau}).
  
  It remains to compute  the  $\Delta^\star_{\ell,k}$ themselves by performing the exact division of $ \Delta^\star_{\ell,\geq k}$ by $\Delta^\star_{\ell,\geq k+1}$. This takes $O(N\Lambda+N^2)$ binary operations for $(\ell,k)=(n_y,0)$,
$O(n^2\tau+n^3)$ binary operations for each $(n_y,k),$ with $k \not=0$,
and $O(n\tau+n^2)$ binary operations for each $(\ell,k),$ with $\ell <n_y$.
 
\end{proof}

\begin{proof}[Proof of Proposition \ref{prop:computinggcddeg}] 
In order to determine $n(x)$ for every $x\in V_{\mathbb{C}}(f)$, we use
Equation (\ref{fixdegree}),
Lemma \ref{lemmacompdeg} and Proposition  \ref{comparingroots}.

Similarly in order to determine $k(x)$ for every $x\in V_{\mathbb{C}}(f)$, we use 
Equation (\ref{compmult2}), Lemma \ref{pro:degrees} and   Proposition  \ref{comparingroots}.
\end{proof}

\subsection{Bivariate root isolation}
 
We now bound the complexity of computing the roots of all polynomials $F(x,Y)$, where $x$ runs over all roots of $R$, using
Proposition  \ref{sagraloff-isolation}.

\begin{proposition}\label{thm:costisolation}
Let 
$R\in\Z[X]$ be
a non-zero polynomial
 of magnitude 
bounded by
$(N,\Lambda)$, let $F\in\Z[X,Y]$ be of magnitude 
bounded by
$(n,\tau)$, and suppose that  
$$V_{\mathbb{C}}(R,F)=\{(x,y)\in \mathbb{C}^2\mid R(x)=F(x,y)=0\}$$
is finite. 
Then, it holds:
\begin{itemize}
\item[(a)] Using $$ \tO(N^2\Lambda+N^3+n\cdot\max(n^2,N)\cdot(N\tau+n\Lambda+Nn)+n^5 \tau +n^6)$$
bit operations, we can compute 
\begin{itemize}
\item[(a.1)] well-isolating disks ${\mathcal D}_{x,y}\subset \C$ for all complex roots $y$ of all polynomials $F(x,Y)$, where the sum of the bitsizes of the radii and centers of all disks ${\mathcal D}_{x,y}$ is bounded by $\tO(n(N\tau+n\Lambda+Nn))$,
\item[(a.2)] the corresponding multiplicities 
$\mu(x,y)$
 for each of the complex roots $y$ of all polynomials $F(x,Y)$, and
\item[(a.3)] dyadic approximation $\tilde{\sigma}_{x,y}$ of the separations $\sep(y,F(x,Y))$ such that $$\frac{1}{2}\cdot\sep(y,F(x,Y))<\tilde{\sigma}_{x,y}<2\cdot \sep(y,F(x,Y))$$ for all roots $x$ (resp. $y$) of $R$ (resp. $F(x,Y)$).
\end{itemize}
\item[(b)] Let $V\subset\{(x,y)\in\mathbb{C}^2:R(x)=0\text{ and }F(x,y)=0\}$, and $L$ be a positive integer. Then, we can further refine all isolating disks ${\mathcal D}_{x,y}$, with $(x,y)\in V$, to a size less than $2^{-L}$, in a number of bit operations bounded by 
$$ \tO(N^2\Lambda+N^3+n\cdot\max(n^2,N)\cdot(N\tau+n\Lambda+Nn)+L\cdot(N\cdot\mu+n^2\cdot\sum_{x \in \pi_X(V)}\mu_{x})),$$
where $\mu_{x}:=\max_{(x,y) \in V}
\mu(x,y)$, and $\mu=\max_{x \in \pi_X(V)}\mu_{x}$.
\end{itemize}
\end{proposition} 

\begin{proof}
First compute  $$n(x):=\deg F(x,Y),k(x):=\deg\gcd(F(x,Y),\partial_Y F(x,Y))$$  
for all roots $x$ of $R$.
According to Proposition~\ref{prop:computinggcddeg}, this computation needs  $$\tO (n\max(N,n^2)(N\tau+n\Lambda+N n)+n^5 \tau + n^6)$$ bit operations. Further notice that $F(x,Y)$ has $m(x)=n(x)-k(x)$ distinct complex roots.

Let $x$ be a fixed complex root of $R$, $\tau_{x}\in\Z$ such that $2^{-\tau_x}<\LCF(F(x,Y))\le 2^{-\tau_x+2}$, and $F_x:=2^{\tau_x}\cdot F(x,Y)$. Notice that $F_x$ has the same roots as $F(x,Y)$, and the leading coefficient of $F_x$ has absolute value in between $1$ and $4$. Then according to Proposition~\ref{sagraloff-isolation}, we can compute well-isolating disks ${\mathcal D}_{x,y}$ for the roots of $F_x$ (and thus also for the roots of $F(x,Y)$) as well as the multiplicities 
$\mu(x,y)$
 in a number of bit operations that is bounded by 
\begin{align}\label{cost:isolation}
\tO(n(n^2+n\cdot\logmea(F_x)
+\logGdisc(F_x))).
\end{align}
using Proposition~\ref{GammaDeltaSigma};
For this, we need an approximation of $F_x$ to an absolute precision that is bounded by
\begin{align}\label{error:input}
\rho_x
\in\tO(n\logmea(F_x)+\logGdisc(F_x)).
\end{align}
Using Proposition \ref{pro:bounds}, Proposition~\ref{thm:bounds} now yields 
the bound 
\[
\sum_{x \in V_{\mathbb{C}}(R)} \rho_x\in\tO(n\cdot (N\tau+n\Lambda+Nn))
\]
for the sum of the bound (\ref{error:input}) for the needed input precision over all $x\in V_{\mathbb{C}}(R)$.

It remains to show that we can compute sufficiently good approximations of the polynomials $F_x$ in a number of bit operations that is bounded by $\tO(n^2\cdot(N\tau+n\Lambda+Nn))$. In order to compute a $\rho_x$-bit approximation of $F _x$, we first compute a $\tau_x$ with $2^{-\tau_x-2}<\LCF(F(x,Y))\le 2^{-\tau_x}$ as well as a $(\rho_x+\tau_x)$-bit approximation of $F(x,Y)$ and then shift the coefficients of the latter approximation of $F(x,Y)$ by $\tau_x$ bits. We first estimate the cost for the computation of the $\tau_x$'s. According to Proposition~\ref{univariate-evaluationbis}, we can compute an approximation $\tilde{c}_x$ of $c_x:=|\LCF(F(x,Y))|$ with $|c_x-\tilde{c}_x|<2^{-L}$ using $\tO(n(L+n\log(\max(1,|x|)))+\tau)$ bit operations, and as input we need an $\tO(L+n\log\max(1,|x|)+\tau)$-bit approximation  of $x$. 
Hence, when choosing $L=2,4,8,\ldots$, we succeed in computing $\tau_x$ for an $L_x$ of size $$L_x\in O(|\log (|\LCF(F(x,Y))|)|+\tau+n+n\log(\max(1,|x|))), $$ and the cost for the evaluation is bounded by 
$\tO(nL_x)$
bit operations. When summing up the latter bound over all $x$ and using 
Proposition~\ref{generalunivariate} (b) (with $g_i=f_{n_y-i}$, the sequence of coefficients of $F\in\Z[X][Y]$), we obtain
\begin{align*}
\sum_{x \in V_{\mathbb{C}}(R)}\tO(nL_x)&\in n\cdot \tO(\sum_{x \in V_{\mathbb{C}}(R)} (|\log( |\LCF(F(x,Y))|)|+\tau+n+n\log(\max(1,|x|))))\\
&\in \tO(n^2(N\tau+n\Lambda+Nn))
\end{align*}
Notice that the above computation further implies that $\sum_{x\in V_{\mathbb{C}(R)}}\tau_x\in\tO(n(N\tau+n\Lambda+Nn))$.

For estimating the cost of computing sufficiently good approximations of $F(x,Y)$, we can again use Proposition~\ref{univariate-evaluationbis}. We conclude that the cost for computing a $(\rho_x+\tau_x)$-bit approximation of $F(x,Y)$ is bounded by $\tO(n^2(\tau+n+n\log(\max(1,|x|))+\rho_x+\tau_x))$ bit operations and that we need a $\tO(\tau+n+n\log(\max(1,|x|))+\rho_x+\tau_x)$-bit approximation of $x$. Summing up the cost over all $x$ then yields
\begin{align*}
&\tO(n^2N(\tau+n)+n^3\sum_{x \in V_{\mathbb{C}}(R)} \log(\max(1,|x|))+n^2\sum_{x \in V_{\mathbb{C}}(R)} (\rho_x+\tau_x))
\end{align*}
which is in
\begin{align*}\tO(n^3\cdot(N\tau+n\Lambda+Nn)).
\end{align*}
Finally, we have to bound the cost for computing sufficiently good approximations of the roots $x$ of $R$. Notice that each term $\tau+n+n\log(\max(1,|x|))+\rho_x+\tau_x$ is bounded by $\tO(n\cdot(N\tau+n\Lambda+Nn))$ for each root $x$ of $R$; in fact, the latter bound even applies to the sum of all these terms. Hence, it suffices to compute approximations of all roots of $R$ to an absolute precision of size $\tO(n\cdot (N\tau+n\Lambda+Nn))$. According to Proposition~\ref{sagraloff-isolation-integer}, the cost for the computation of well-isolating disks of corresponding size is bounded by $\tO(N^2\Lambda+N^3+Nn(N\tau+n\Lambda+Nn))$. The bound on the sum of the bitsizes of the radii and the centers of the disks ${\mathcal D}_{x,y}$ follows directly from Proposition~\ref{pro:bounds} and Proposition~\ref{thm:bounds}.
This concludes the proof of Parts a) and b).
For Part (c), notice that $$\frac{\min_{y'\neq y:F(x,y')=0}|m_{x,y}-m_{x,y'}|}{\sep(y,F(x,Y))}\in (1-1/32,1+1/32),$$ where $m_{x,y}$ denotes the center of $\mathcal{D}_{x,y}$. Hence, approximations $\tilde{\sigma}_{x,y}$ with the required properties can directly be obtained from the distances between the centers $m_{x,y}$. Since the sum of the bitsizes of all centers $m_{x,y}$ is bounded by 
$\tO(\sum_{x \in V_{\mathbb{C}}(R)} \logmea(F_x)+\logsep(F_x^\star))\in\tO(n(N\tau+n\logmea+Nn))$, the claim follows.

It remains to prove the last claim on the cost for refining the disks $\mathcal{D}_{x,y}$, with $(x,y)\in V$, to a size less than $2^{-L}$. For this, we use Proposition~\ref{sagraloff-isolation-integer} (c), which shows that, for a fixed $x$, we can refine all disks $\mathcal{D}_{x,y}$ using 
$$\tO(n(L\cdot \mu_{x}+n^2\cdot\logmea(F_x)+n\logGdisc(F_x)+n^3))$$
bit operations. Now, summing the latter bound over all $x$ yields $$\tO(n^2(N\tau+n\Lambda+Nn)+nL\cdot\sum_{x \in V_{\mathbb{C}}(R)}\mu_{x}).$$ For the input precision $\rho_x$ to which we need to approximate the polynomial $F_x$, we obtain the bound  
$$\rho_x
\in\tO(L\cdot \mu+n\cdot\logmea(F_x)+\logGdisc(F_x)+n^2).$$
 Again, using the same argument as above, it follows that sufficiently good approximations of the polynomials $F_x$ can be computed using 
 $$
 \tO(n^3\cdot(N\tau+n\Lambda+Nn)+L(N\mu+ n^2\cdot\sum_{x \in V_{\mathbb{C}}(R)} \mu_{x})) 
 $$
 bit operations.
\end{proof}

Theorem \ref{firstmain} follows from  
 Proposition \ref{thm:costisolation}.

\begin{proposition}\label{thm:comparinginfibers}
Let 
$R\in\Z[X]$
a non-zero polynomial
 of magnitude 
bounded by
$(N,\Lambda)$, $F,G\in\Z[X,Y]$ be of magnitude
bounded by
 $(n,\tau)$, and $H:=F\cdot G$. Suppose moreover that
$$V_{\mathbb{C}}(R,H)=\{(x,y)\in \mathbb{C}^2\mid R(x)=H(x,y)=0\}$$
is finite.
Using a number of bit operations bounded by 
\[
 \tO(N^2\Lambda+N^3+n^5\tau+n^6+n\cdot\max(n^2,N)\cdot (N\tau+n\Lambda+Nn))
\] 
we can can carry out the following computations for all complex roots $x$ of $R$:
\begin{itemize}
\item[(a)] computing well-isolating disks $\mathcal{D}_{x,y}$ for all complex roots $y$ of the polynomial $H(x,Y)$ together with the corresponding multiplicities $\nu(x,y)=\mult(y,H(x,Y))$.
\item[(b)] determining for each root $y$ of $H(x,Y)$ whether $y$ is a root of $F(x,Y)$, $G(x,Y)$, or both. If $x$ as well as $y$ are real, we can further determine the sign of $F(x,y)$ and $G(x,y)$.
\end{itemize}
\end{proposition} 

\begin{proof}
Part (a) already follows from Lemma~\ref{pro:degrees} and Proposition~\ref{thm:costisolation} as $H$ has magnitude
bounded by
 $(O(n),O(\log (n)+\tau))$. We 
may further assume that, for all complex roots $x$ of $R$, we have already computed
\begin{itemize}
\item well-isolating disks $\mathcal{D}^F_{x,y'}$ 
and $\mathcal{D}^G_{x,y''}$ for all complex roots $y'$ and $y''$ of the polynomials $F(x,Y)$ and $G(x,Y)$, respectively, 
\item the corresponding multiplicities $\mult(y',F(x,Y))$ and $\mult(y'',G(x,Y))$,
\item the degrees of the polynomials $F(x,Y)$ and $G(x,Y)$, and
\item the signs of the leading coefficients of the polynomials $F(x,Y)$ and $G(x,Y)$ in case that $x$ is a real root of $R$.
\end{itemize}

For (b), we now refine each disk $\mathcal{D}^F_{x,y'}$ such that it intersects with exactly one of the disks $\mathcal{D}_{x,y}$. If this is the case, then it holds that $y=y'$. In addition, $y$ is also a root of $G(x,Y)$ if and only if $\nu(x,y)>\mult(y',F(x,Y))$ as $\mult(y,F(x,Y))+\mult(y,G(x,Y))=\nu(x,y)$. Hence, for each root $y$ of $H(x,Y)$, we also know its multiplicity as a root of $F(x,Y)$ and $G(x,Y)$. When restricting to the real roots of $R$, this implies that we can directly deduce the sign of $F(x,y)$ (resp. $G(x,y)$) at each real root $y$ of $H(x,Y)$ that is not a root of $F(x,Y)$ (resp. $G(x,Y)$). Namely, from the sign of the leading coefficient of $F(x,Y)$ (resp. $G(x,Y)$) and its degree, we know its sign at $\pm\infty$, and the polynomial changes signs exactly at those roots of $H$ that are roots of $F(x,Y)$ (resp. $G(x,Y)$) of odd multiplicity.  

It remains to bound the cost for the refinement of the disks $\mathcal{D}^F_{x,y'}$. We proceed in rounds enumerated by $\ell=1,2,3,\ldots$: Initially, we set $Z^1:=V_{\mathbb{C}}(R,F)$. In the $\ell$-th round, we refine each of the isolating disks $\mathcal{D}^F_{x,y'}$ for all $(x,y')\in Z^\ell$ to a size less than $2^{-2^{\ell}}$ and check whether it intersects exactly one of the isolating disks $\mathcal{D}_{x,y}$ for the roots of $H(x,Y)$. If this is the case, we know that $y=y'$. 

 After having treated all points in $Z^{\ell}$, we set $Z^{\ell+1}$ to be the set of all $(x,y')$ in $Z^\ell_x$ for which the isolating disk $\mathcal{D}^F_{x,y'}$ intersects more than one of isolating disks for $H(x,Y)$. That is, $Z^{\ell}$ is the set of all $(x,y')\in Z^1$ for which we have not determined a corresponding root $y$ of $H(x,Y)$ with $y=y'$ after the $\ell$-th round. We then proceed with the $(\ell+1)$-st round. We stop as soon as  $Z^{\ell}$ becomes empty, in which case, each root of each $F(x,Y)=0$ is matched to a corresponding root of $H(x,Y)$.
Notice that, for each $(x,y')$, we must succeed in round $\ell_{x,y'}$, for some $\ell_{x,y'}$ with
$$2^{\ell_{x,y'}}\le O(|\log(\sep(y',H(x,Y)))|).$$ From the amortized bounds on the separation of the roots  (Proposition~\ref{thm:bounds}), we have that $|\log(\sep(y',H(x,Y)))|\in \tO(n(N\tau+n\Lambda+Nn))$, thus
we are done after $\ell_{\max}$ rounds for  
\begin{equation}\label{kappa0}
\ell_{\max}=\max_{x,y'} \ell_{x,y"}+1\in O(\log (n(N\tau+n\Lambda+Nn))).
\end{equation}
According to Proposition~\ref{thm:costisolation}, the cost for refining the disks $\mathcal{D}^F_{x,y'}$ for all $(x,y')\in Z^\ell$ to a size less than $2^{-2^{\ell}}$ is bounded by 
\begin{align*}
\sum_{\ell=1}^{\ell_{\max}} \tO(N^2\Lambda+N^3+n\max(n^2,N)(N\tau+n\Lambda+Nn)+2^{\ell}(N\mu^{[\ell]}+n^2\sum_{x\in\C:R(x)=0}\mu_x^{[\ell]}))
\end{align*}
bit operations,
where 
\[
\mu_{x}^{[\ell]}:=
\begin{cases}
\max_{(x,y')\in Z^{\ell}} \mult(y',F(x,Y))&\text{if  there exists }(x,y')\in Z^{\ell}\\
0 &\text{otherwise.}
\end{cases}
\]
 and $\mu^{[\ell]}:=\max_{x:R(x)=0}\mu^{[\ell]}_x.$
 Since $\ell_{\max}$ is bounded by $O(\log (n(N\tau+n\Lambda+Nn)))$, we are left to bound the sum
 \begin{align}\label{sumalongfiber}
 \sum_{\ell=1}^{\ell_{\max}} 2^{\ell}\cdot (N\cdot\mu^{[\ell]}+ n^2\cdot \sum_{x\in\C:R(x)=0}\mu_x^{[\ell]}).
 \end{align}
If, for a fixed root $x$ of $R$, there exists  a $(x,y')\in Z^{\ell}$, then let $(x,y'_{\ell,x})\in Z^\ell$ be a point in $Z^\ell$ with $\mult(y'_{\ell},F(x_\ell,Y))=\mu^{[\ell]}$. In other words, $y'_{\ell,x}$  maximizes the multiplicity within the fiber for all roots $y'$ of $F(x,Y)$ with $(x,y')\in Z^{\ell}$. In addition, let $(x_\ell, y'_\ell)$ be a point in $Z^\ell$ that maximizes the multiplicity over all fibers. Thus, the sum in (\ref{sumalongfiber}) can be rewritten as
\[
N\cdot\sum_{\ell=1}^{\ell_{\max}} \mult(y'_\ell,F(x_\ell,Y))\cdot 2^{\ell}+n^2\cdot\sum_{\ell=1}^{\ell_{\max}} \sum_{x\in\C:R(x)=0}\mult(y'_{\ell,x},F(x,Y))\cdot 2^\ell.
\]
Notice that a pair $(x,y')\in Z^{\ell_x}$ can appear at most $\ell_{\max}$ many times in each of the above sums. In addition, it holds that $(x,y')\notin Z^{\ell}$ if $\ell>\ell_{x,y'}$ for some $2^{\ell_{x,y'}}\in O(|\log(\sep(y',H(x,Y)))|)$. Hence, the above sum is upper bounded by
\begin{align}\label{sumalongfiber2}
O(\ell_{\max}\cdot (n^2+N)\cdot \sum_{(x,y')\in Z^1}\mult(y',F(x,Y))\cdot |\log(\sep(y',H(x,Y)))|)
\end{align}
Since $\mult(y',F(x,Y))\le \mult(y',H(x,Y))$ and $\ell_{\max}\in O(\log (n(N\tau+n\Lambda+Nn)))$, we thus conclude from Proposition~\ref{thm:bounds} that (\ref{sumalongfiber2}) is upper bounded by
\[
\tO(n\cdot(n^2+N)\cdot (N\tau+n\Lambda+Nn)).
\]
Finally, the cost of checking whether an isolating disk $\mathcal{D}^F_{x,y'}$ intersects exactly one of the isolating disks $\mathcal{D}_{x,y}$ for the roots of $H(x,Y)$ is bounded by 
$O(n\cdot(\logsep(H(x,Y)^\star)+\logmea(H(x,Y)))$ as there are $n$ comparisons between disks with radii and centers of bitsize 
$O(\logsep(H(x,Y)^\star)+\logmea(H(x,Y))$. The total cost for all comparisons is thus bounded by 
\[
O(\kappa\cdot n^2\cdot \sum_{x:R(x)=0} [\logsep(H(x,Y)^\star)+\logmea(H(x,Y))]=\tilde{O}(n^3(N\tau+n\Gamma+Nn)),
\]
where the latter inequality follows from Proposition~\ref{pro:bounds} and Proposition~\ref{thm:bounds}.
Hence, our claim follows.
\end{proof}

We remark that the above considerations immediately yield a
new
 algorithm for solving a bivariate polynomial system that achieves the current record complexity bound for this problem.

\begin{corollary}\label{cor:bivariatesystems}
Let $F,G\in\Z[X,Y]$ be coprime polynomials of magnitude 
bounded by
$(n,\tau)$. Then, we can compute  isolating regions for all complex solutions of the system $F=G=0$ using $\tO(n^5\tau+n^6)$ bit operations.
\end{corollary}
\begin{proof}
Let $R=\res_X(F,G)$ be the resultant polynomial of $F$ and $G$, which can be computed using $\tO(n^4\tau+n^5)$ bit operations. Any common solution $(x_0,y_0)\in\mathbb{C}$ of $F=G=0$ yields a root $x_0$ of $R$, and then $y_0$ is a common root of $F(x_0,Y)$ and $G(x_0,Y)$. Vice versa, any common root $y_0$ of $F(x_0,Y)$ and $G(x_0,Y)$ yields a solution of $F=G=0$. According to Proposition~\ref{thm:comparinginfibers}, we can compute all common roots of $F(x,Y)$ and $G(x,Y)$ for all complex roots $x$ of $R$ using $\tO(n^5\tau+n^6)$ bit operations. Hence, the claim follows. 
\end{proof}

\begin{remark}\label{thankstotest}

In the proof of Proposition~\ref{thm:comparinginfibers}, we needed to recursively refine the isolating disks for the roots of $F(x,Y)$ until a certain test applies. That is, we needed to check whether a disk intersects exactly one of the isolating  disks of $H(x,Y)$. While this test itself is rather simple (and cheap), its success is directly related to a hidden parameter, which is in this case the separation of some specific (but unknown) root of the polynomial $H(x,Y)$. 

In the worst case, the cost for refining isolating disks for the polynomial $F(x,Y)$ until the test applies is rather large (i.e. it is comparable to our bound for the overall computation), and thus it would be very costly to refine all isolating disks to such a small size. Fortunately, this is not necessary as, for most roots of $F(x,Y)$, the success of the test is related to a root of $H(x,Y)$ with larger separation. In order to exploit this fact, we need to design adaptive algorithms in each of these cases and to use our amortized bounds on the separation of the roots (Proposition \ref{corodisc} and Proposition~\ref{thm:bounds}).

Three instances of such a situation appear in this paper, the first in Proposition \ref{thm:comparinginfibers} and the two others in the proof of Theorem \ref{sing-fibers} parts c.1 and c.2.
\end{remark}

\section{Computation of the topology}\label{sec:top}
\subsection{Preliminary definitions and results
}\label{sec:basic}

Let $P \in \Z [X, Y]$ be a square-free polynomial of magnitude
bounded by
 $(d,\tau)$. In addition, let
$$V_{\mathbb{R}} (P) = \{(x, y) \in \mathbb{R}^2 \mid P (x, y) = 0\}$$ be the real
algebraic curve defined by $P$, and 
$$V_{\mathbb{C}}(P) = \{(x, y) \in \mathbb{C}^2 \mid P (x, y) = 0\}$$
be the corresponding complex algebraic curve.

We first decompose
\begin{equation}\label{tildeP}
P(X,Y)=p(X)\cdot Q(X,Y)
\end{equation}
 with $p(X)\in \Z [X]$ and  $Q(X,Y)\in  \Z [X, Y]$ such that $Q(x,Y)$
never identically vanishes for  any $x \in \mathbb{C}.$ In more geometric terms, we separate the vertical lines contained in $V_{\mathbb{C}} (P)$ from the remaining part $V_{\mathbb{C}} (Q)$ of the curve, which does not contain any vertical lines.

\begin{proposition}\label{removelines}
We can compute $p(X)$ and $Q(X,Y)$ using $\tO(d^4+d^3\tau)$ bit operations. 
The polynomial $p(X)$ and $Q(X,Y)$
 have magnitude bounded by $(d,\tau+d+\log (d+1))$.
\end{proposition}

\begin{proof}
Let  $P(X,Y):=p_{\deg_Y(P)}(X)Y^{\deg_Y(P)}+\ldots+p_0(X),$ with $\deg_Y(P)\le d$. Let $p(X)$ the gcd of all the coefficients $p_i(X)$. We first compute $p(X)$ then write $P(X,Y)= p(X)Q(X,Y)$.
The claimed bounds on the cost for computing $p(X)$ and $Q(X,Y)$ 
 as well as on the magnitude of these polynomials follow immediately from Propositions~\ref{Mignotte} and~\ref{gcd-comp}.
\end{proof}

From now on, we study the zero set of $Q(X,Y)$,  which contains no vertical lines.
We suppose moreover that $\deg_X(Q(X,Y))>0$ because otherwise $V_{\mathbb{R}} (Q)$ is a finite number of horizontal lines and its topology, very easy to describe, is treated as a special case in the first part of sub-section \ref{topology}.
We also describe in sub-section \ref{topology} how to add back vertical lines contained in  $V_{\mathbb{R}} (P)$ to obtain the topology of $V_{\mathbb{R}} (P)$ from the topology of $V_{\mathbb{R}} (Q)$.

A point $(\alpha, \beta) \in V_{\mathbb{R}} (Q)$ 
 is called:
\begin{itemize}
 \item an \emph{$X$-critical} point if $\partial_Y  Q(\alpha, \beta) = 0$,
 \item a \emph{$Y$-critical} point if $\partial_X Q (\alpha, \beta) = 0$,
 \item a \emph{singular} point if $\partial_X Q (\alpha, \beta) = \partial_Y Q(\alpha, \beta) = 0$,
 \item a \emph{regular} point if $\partial_Y Q (\alpha, \beta)\neq 0$ and
 $\partial_X Q (\alpha, \beta) \neq 0$.
 \end{itemize}
 Notice that a singular point of $V_{\mathbb{R}} (Q)$ is also an $X$-critical and a $Y$-critical point.

We define
 \begin{eqnarray}
 D_X(X)& := 
 & \res_Y(Q,\partial Y Q)=q_{d'}(X)\Disc_Y (Q) (X)
 , \label{eqdisc}
 \end{eqnarray}
 where $q_{d'}(X)$ is the leading coefficient of $Q$ with respect to $Y$. 
 
  In geometric terms, the zeroes of $D_X$ are the projections on the $X$-axis of the 
 $X$-critical points
 or the values $\alpha$ such that $\deg(Q(\alpha,Y))<\deg_Y(Q)=d'$.
 
  The special case where $D_X(X)$ has no real root is considered separately in the first part of sub-section \ref{topology}.

Let
 \begin{equation}
\alpha_1<\ldots<\alpha_{N}
\end{equation}
be  the real roots of $
D_X(X)$.

  For every $i=1,\ldots,N$, we denote by  $m_i$ the number of roots of $Q(\alpha_i,Y)$ and by 
  $\beta_{i,1},\ldots,\beta_{i,m_i}$, with
\begin{equation}
\beta_{i,1}<\ldots<\beta_{i,m_i},
\end{equation}
 the real roots of $Q(\alpha_i,Y)$.
 
 We denote by
  $m'_0$ (resp.  $m'_i$,$i=1,\ldots,N$ ) the number of roots of 
 $Q(x,Y)$ when $x\in(-\infty,\alpha_1)$ (resp $x\in(\alpha_i,\alpha_{i+1})$, $x\in(x_N,+\infty)$).
 It is clear that the numbers  $m'_i$,$i=1,\ldots,N-1$ coincide with the number of roots of $Q(\alpha'_i,Y)$ where $\alpha'_i$ is the smallest root of $D'_X$ bigger than $\alpha_i$.

 In order to evaluate $m'_0$ and $m'_N$, we use the following classical Cauchy bound.
 Let $f(X)=a_nX^n+\cdots + a_qX^q$, be a univariate polynomial of degree $n$  with coefficients in $\mathbb{C}$, such that $a_q \neq 0$, $a_{q-1}=\ldots,a_0=0$.
Denoting
\begin{equation}
\label{cauchybounds}
\cau(f):=\sum_{q \leq i\leq n} \left| \frac{a_i}{a_n} \right|,
\end{equation}
the absolute value of any 
root of $f$ in $\mathbb{C}$ is 
smaller than $\cau(f)$ (cf. for example \cite[Lemma 10.2]{BPRbook2}). The numbers  $m'_0$ and  $m'_N$ coincide with the number of roots of $Q(\cau(D_X),Y)$ and $Q(-\cau(D_X),Y)$.

   For each fixed $i$, we define the set of \emph{indices of $X$-critical points above $i$} as
\[
\textsc{CritInd}_i:=\{j\in\{1,\ldots,m_i\}:(\alpha_i,\beta_{i,j})\in\Crit_X(V_{\mathbb{R}}(Q))\}.
\]
with
$\Crit_X(V_{\mathbb{R}}(Q)) := \{(x, y) \in \mathbb{R}^2 \mid Q(x, y) =\partial_Y Q(x,y)= 0\}$ 
and
the set of \emph{indices of singular  points above $i$} as
\[
\textsc{SingInd}_i:=\{j\in\{1,\ldots,m_i\}:(\alpha_i,\beta_{i,j})\in\Sing(V_{\mathbb{R}}(Q))\}.
\]
with
$\Sing(V_{\mathbb{R}}(Q)) := \{(x, y) \in \mathbb{R}^2 \mid Q(x, y) =\partial_Y Q(x,y)=\partial_X Q(x,y)= 0\}.$

Our aim is to determine the number  $\textsc{Left}_{i,j}$ of the 
curve segments
ending at 
$(\alpha_i,\beta_{i,j})$ to the left of $\alpha_i$ as well as the number
 $\textsc{Right}_{i,j}$ of the 
curve segments
 ending at 
$(\alpha_i,\beta_{i,j})$ to the right of $\alpha_i$.

\begin{figure}
\centering
 \includegraphics[width=9cm,height=7cm]{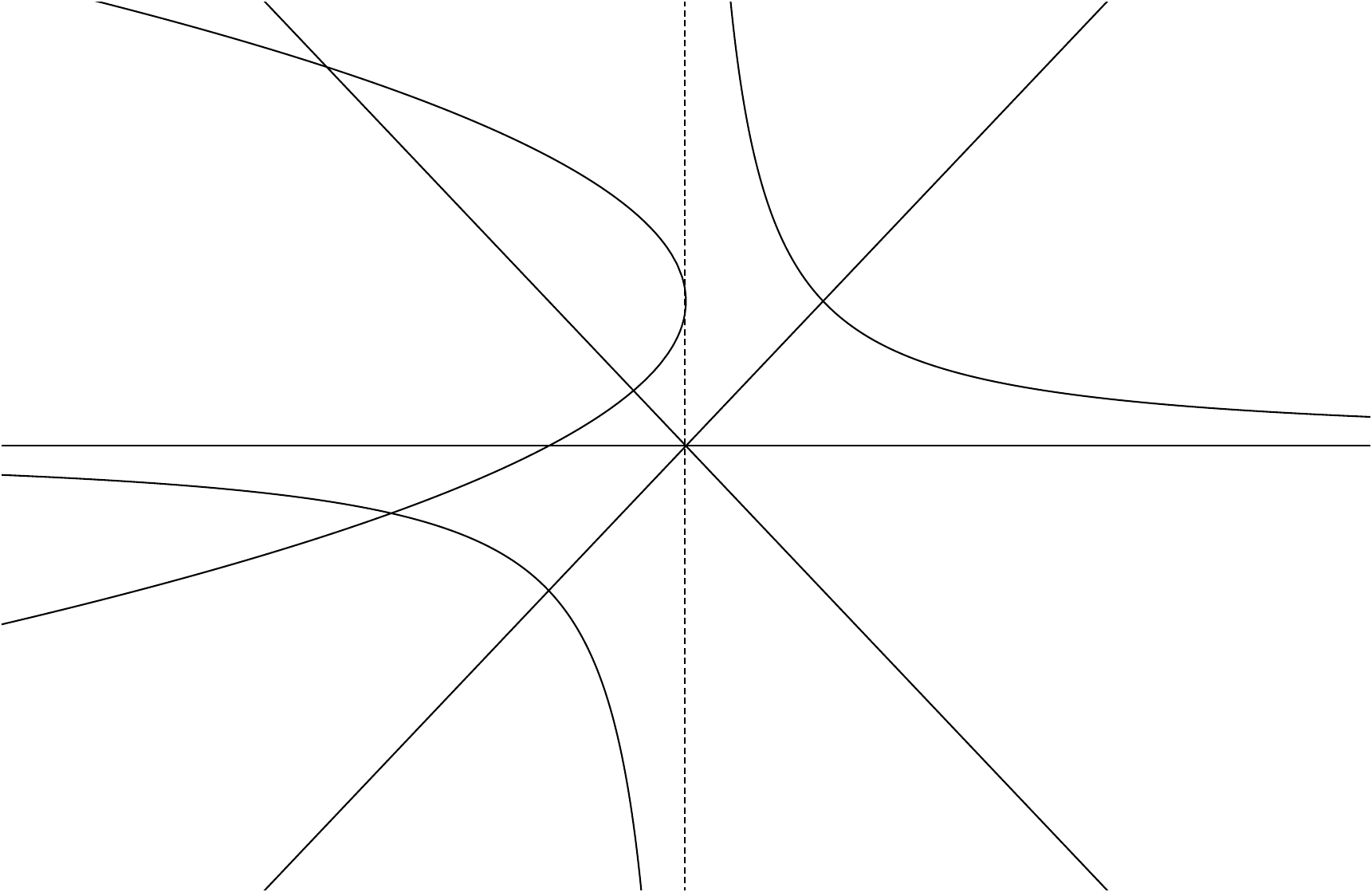} 
\caption{\label{various cases} Simple asymptotes, $X$-critical points and singular points}
\end{figure}

These numbers   are easy to determine if $j\notin \textsc{SingInd}_i$.
If $(\alpha_i,\beta_{i,j})$ is regular, $\textsc{Left}_{i,j}=\textsc{Right}_{i,j}=1$.
If $(\alpha_i,\beta_{i,j})$ is $X$-critical but not singular, denoting by $\mu_{i,j}$  the multiplicity of $\beta_{i,j}$ as a root of $Q(\alpha_i,Y)$ 
\begin{itemize}
\item if $\mu_{i,j}$  is odd, 
 \begin{equation}
  \label{equa:odd}\textsc{Left}_{i,j}=\textsc{Right}_{i,j}=1 
 \end{equation}
 \item if $\mu_{i,j}$  is even and $\partial_Y Q^{[\mu_{i,j}]}(\alpha_i,\beta_{i,j})>0$,  
\begin{equation}
\label{equa:crit+}\textsc{Left}_{i,j}=0,\textsc{Right}_{i,j}=2 
\end{equation}
\item if $\mu_{i,j}$  is even and $\partial_Y  Q^{[\mu_{i,j}]}(\alpha_i,\beta_{i,j})<0$, 
 \begin{equation}
 \label{equa:crit-}\textsc{Left}_{i,j}=2,\textsc{Right}_{i,j}=0 
 \end{equation}
\end{itemize}

We now explain how to deal with the vertical asymptotes in the simple case.  We denote by $d'=\deg_Y(Q)\le d$ and by $q_{0}(X),\ldots, q_{d'}(X)$ the coefficients of $Q$ considered as a polynomials in $Y$. The vertical asymptotes occur at values of $\alpha$ where  $\deg(Q(\alpha,Y))<\deg_Y(Q)$, i.e. at zeroes of $q_{d'}$. Our aim is to count the number of asymptotes $\textsc{Left}_{i,0}$ and  $\textsc{Right}_{i,0}$ tending to $-\infty$ just before $\alpha_i$ and  just after $\alpha_i$ (resp.  $\textsc{Left}_{i,m_i+1}$  and $\textsc{Right}_{i,m_i+1}$ tending to $+\infty$ just before $\alpha_i$ and  just after $\alpha_i$).
If $\deg(Q(\alpha,Y))=d'$ there are no asymptotes and
$$ \textsc{Left}_{i,0}=\textsc{Right}_{i,0}=\textsc{Left}_{i,m_i+1}=\textsc{Right}_{i,m_i+1}=0.$$
The asymptote is \emph{simple} when $\deg(Q(\alpha,Y))=d'-1$, i.e. when $q_{d'}(\alpha)=0$,  $q_{d'-1}(\alpha)\not=0$, and we have
\begin{itemize}
\item if $q_{d'-1}(\alpha)<0$, 
 \begin{equation}   \label{equa:asympt-}
 \textsc{Left}_{i,0}=1,\textsc{Right}_{i,0}=0,\textsc{Left}_{i,m_i+1}=0,\textsc{Right}_{i,m_i+1}=1;
 \end{equation}
\item if $q_{d'-1}(\alpha)>0$
 \begin{equation} \label{equa:asympt+}
 \textsc{Left}_{i,0}=0,\textsc{Right}_{i,0}=1,\textsc{Left}_{i,m_i+1}=1,\textsc{Right}_{i,m_i+1}=0.
  \end{equation}
\end{itemize}

Using the notation above, we can prove the following Theorem.
\begin{theorem}[
Cylindrical Algebraic Decomposition]\label{sing-fibers0}
 Using $\tO (d^5 \tau + d^6)$ bit operations, we can compute:
 \begin{itemize}
 \item the number $N$,
 \item for each $i=1,\ldots, N$, the numbers $m_i$,
 \item for each $i=0,\ldots, N$, the numbers $m'_i$,
 \item for each $i=1,\ldots, N$, the set $\textsc{CritInd}_i$,
   \item for each $i=1,\ldots, N$, the set $\textsc{SingInd}_i$,
 \item for each $j\in \textsc{CritInd}_i\setminus \textsc{SingInd}_i$,  the  multiplicity $\mu_{i,j}$ of $\beta_{i,j}$ as a root of $Q(\alpha_i,Y)$
  and the sign $\sign(\partial_Y^{\mu_{i,j}} Q(\alpha_i,\beta_{i,j}))$,
  \item for each $i=1,\ldots, N$, the degree $\deg(Q(\alpha_i,Y))$, 
  \item for each $i=1,\ldots, N$ such that $\deg(Q(\alpha_i,Y))=d'-1$, $\sign(q_{d'-1}(\alpha_i))$.
  \end{itemize}
\end{theorem}
\begin{proof}
First note, using Lemma \ref{magni}, that $D_X$ is a polynomial of magnitude
bounded by
 $(N,\Lambda)\in (O(d^2),\tO(d\tau+d^2))$.  Hence, by Proposition \ref{sagraloff-isolation-integer} and Proposition \ref{comparingroots}, 
we can compute well-isolating intervals for the real roots of $D_X$ in a number of bit operations bounded by $\tO(d^5\tau+d^6)$.
 According to 
Proposition~\ref{prop:computinggcddeg} (with $R:=D_X$ and $F:=Q(X,Y)$), we may further compute $\deg(Q(\alpha_i,Y))$ and the number of distinct roots of $Q(\alpha_i,Y)$ for all $i$ from $1$ to $N$  using $\tO(d^5\tau+d^6)$ bit operations. From Proposition~\ref{thm:costisolation}, we conclude that  using also $\tO(d^5\tau+d^6)$ bit operations we can further compute 
well-isolating intervals for all real roots of the 
polynomials $Q(\alpha_i,Y)$ as well as the corresponding multiplicities $\mu_{i,j}$ of $\beta_{i,j}$ for all $i$ from $1$ to $N$. Each root $\beta_{i,j}$ of multiplicity larger than one then 
corresponds to an $X$-critical point $(\alpha_i,\beta_{i,j})$, which defines $\textsc{CritInd}_i$ for $i=1,\ldots,N$.
In order to define $\textsc{SingInd}_i$ for $i=1,\ldots,N$ we compute the sign of $\partial_X Q(\alpha_i,\beta_{i,j})$ for all $j\in \textsc{CritInd}_i$ using Proposition \ref{thm:comparinginfibers}.
We compute for each $j\in \textsc{CritInd}_i\setminus \textsc{SingInd}_i$  the sign $\sign(\partial_Y^{\mu_{i,j}} Q(\alpha_i,\beta_{i,j}))$ using again Proposition \ref{thm:comparinginfibers}.
We finally compute for each $i=1,\ldots, N$ such that $\deg(Q(\alpha_i,Y))=d'-1$, $\sign(q_{d'-1}(\alpha_i))$ using Proposition  \ref{comparingroots}.
 \end{proof}

As a consequence if it is the case that
  \begin{itemize}
   \item for each $i=1,\ldots, N$, the set  $\textsc{SingInd}_i$ has $0$ or $1$ element
  \item for each $i=1,\ldots, N$, $\deg(Q(\alpha_i,Y))=d'$  or $\deg(Q(\alpha_i,Y))=d'-1$
  \end{itemize}
  we have all the information needed to compute $\textsc{Left}_{i,j}$ and $\textsc{Right}_{i,j}$ for $j\notin \textsc{SingInd}_i$ as well as $\textsc{Left}_{i,0}$ and $\textsc{Right}_{i,m_i+1}$ 
 using  Equations ( \ref{equa:odd}), (\ref{equa:crit+}),(\ref{equa:crit-}),  (\ref{equa:asympt-}),(\ref{equa:asympt+}).
 Finally, for each $i=1,\ldots, N$, $j\in \textsc{SingInd}_i$
\begin{itemize}
\item $\textsc{Left}_{i,j}=m'_{i-1}-\sum_{k=0,\ldots,m_i+1,k\not=j} \textsc{Left}_{i,k}$,
\item $\textsc{Right}_{i,j}=m'_i-\sum_{k=0,\ldots,m_i+1,k\not=j} \textsc{Right}_{i,k}.$
\end{itemize}

The more complicated remaining cases when there are several singular points above one $\alpha_i$ or several vertical asymptotes at $\alpha_i$ are discussed in the following sections.

\subsection{Adjacency boxes and number of left and right curve segments at a singular point}\label{connectingalgo}

We now deal with the determination of the number  $\textsc{Left}$  of the  curve segments
ending at  $(\alpha,\beta)$ to the left of $\alpha$ as well as the number $\textsc{Right}$ of the 
curve segments ending at $(\alpha,\beta)$ to the right of $\alpha$ at a singular point $(\alpha,\beta)$  of $V_{\mathbb{R}} (Q)$.
 
\begin{definition}\label{def:adjac}
A rectangle
$\bbox(\alpha,\beta):=[\alpha^-,\alpha^+]\times [\gamma^-,\gamma^+]$
is an adjacency box associated to a 
singular 
point $(\alpha,\beta)$  of $V_{\mathbb{R}} (Q)$
if the following holds
\begin{itemize}
\item $\alpha^-< \alpha < \alpha^+$, $\gamma^-< \beta < \gamma^+$,
\item there is no $X$-critical point of $V_{\mathbb{R}} (Q)$ inside $([\alpha^-,\alpha)\cup (\alpha,\alpha^+])\times \mathbb{R}$,
\item except points of the horizontal line $Y=\beta$ when it is included  $V_{\mathbb{R}} (Q)$, there is  no $Y$-critical point of $V_{\mathbb{R}} (Q)$ inside $([\alpha^-,\alpha)\cup (\alpha,\alpha^+])\times \mathbb{R}$,
\item there is no $Y$-critical point or $X$-critical point  of $V_{\mathbb{R}} (Q)$ inside $\mathbb{R} \times([\gamma^-,\beta)\cup (\beta,\gamma^+])$.
\end{itemize}
It follows clearly from the definition that $(\alpha,\beta)$ is the only singular point of 
$\bbox(\alpha,\beta)$.
\end{definition}

The aim of this subsection is to explain how counting the intersection points of $V_{\mathbb{R}} (Q)$ with some specific parts of the boundary of $\bbox(\alpha,\beta)$, 
and taking into account some sign conditions on the slope of the curve, 
 makes it possible to compute the number  $\textsc{Left}$ and
 $\textsc{Right}$.

We denote by 
\begin{itemize}
\item $\LBox(\alpha,\beta):=
[\alpha^-,\alpha] \times [\gamma^-,  \gamma^+]\setminus \{(\alpha,\beta)\}$  the left-half  of $\bbox(\alpha,\beta)$
\item $\RBox(\alpha,\beta):=
[\alpha,\alpha^+] \times [\gamma^-,  \gamma^+]\setminus \{(\alpha,\beta)\}$
 the right-half of $\bbox(\alpha,\beta)$.
 \end{itemize}

 We denote by $\mathrm{Slope}(X,Y)$ the rational fraction
 $$\mathrm{Slope}(X,Y)=-\frac{\partial_X Q(X,Y)}{\partial_Y Q(X,Y)}.$$ 
 The slope of the
tangent to a level line of $Q$  at a regular point $(x,y)$ is
given by 
 $$\mathrm{Slope}(x,y)=-\frac{\partial_X Q(x,y)}{\partial_Y Q(x,y)}.$$ 
 Note that at any point $(x,y)$ of $\bbox(\alpha,\beta) \setminus \{(\alpha,\beta)\}$ intersected with $V_{\mathbb{R}} (Q)$, $\mathrm{Slope}(x,y)$ is well defined and not $0$, except if $y=\beta$ and $Y=\beta$ is a horizontal line contained in  $V_{\mathbb{R}} (Q)$.

A boundary point $(x, y)$ of $\LBox(\alpha,\beta)$ (resp.  $\RBox(\alpha,\beta)$) in $V_{\mathbb{R}} (Q)$
is exactly of one of the following types

- type 0: there is no curve segment of $V_{\mathbb{R}} (Q)$
containing $(x,y)$ inside $\LBox(\alpha,\beta)$  (resp.  $\RBox(\alpha,\beta)$);  this is the case only at some corners, i.e. if
$(x,y)=(\alpha_-,\gamma_+)$ with $\mathrm{Slope}(\alpha_-,\gamma_+)>0$ or $(x,y)=(\alpha_-,\gamma_-)$  with $\mathrm{Slope}(\alpha_-,\gamma_-)<0$ (resp. $(x,y)=(\alpha_+,\gamma_+)$  with $\mathrm{Slope}(\alpha_+,\gamma_+)<0$ or $(x,y)=(\alpha_+,\gamma_-)$) with $\mathrm{Slope}(\alpha_+,\gamma_-)>0$).

 - type 1: there is a 
 curve segment of $V_{\mathbb{R}} (Q)$ containing $(x,y)$ which
  ends at another boundary point denoted by $\Match(x,y)$ of $\LBox(\alpha,\beta)\cap V_{\mathbb{R}} (Q)$  (resp.  $\RBox(\alpha,\beta)\cap V_{\mathbb{R}} (Q)$);

 - type 2: there is a 
 curve segment of $V_{\mathbb{R}} (Q)$ containing $(x,y)$ which
 ends at $\Match(x,y)=(\alpha, \beta)$.

\begin{remark}\label{propcurvesegment}
The following properties are clear from Definition \ref{def:adjac}.
The matching point $\Match(x,y)$  has the same slope sign as $(x,y)$ for
type 1. If $(x,y)\not= (x',y')$, $\Match(x,y)\not=\Match(x',y')$ except if
$\Match(x,y)=\Match(x',y')=(\alpha,\beta)$.
A curve segment 
associated to a point of type 1 
does not meet any other 
curve segment, 
and 
a curve segment of type 2
 meets only other 
 curve segments 
 of type 2, at
$(\alpha, \beta)$.
\end{remark}

We denote by 
\begin{itemize}
 \item $\LBox^+(\alpha,\beta):=
[\alpha^-,\alpha] \times (\beta, \gamma^+]$  the upper left-half  of $\bbox(\alpha,\beta)$
\item $\RBox^+(\alpha,\beta):=
[\alpha,\alpha^+] \times (\beta,  \gamma^+]$
the upper right-half  of $\bbox(\alpha,\beta)$,
 \item $\LBox^-(\alpha,\beta):=
[\alpha^-,\alpha] \times [\gamma^-,\beta)$  the lower left-half  of $\bbox(\alpha,\beta)$
\item $\RBox^-(\alpha,\beta):=
[\alpha,\alpha^+] \times [\gamma^-,  \beta)$
 the lower right-half of $\bbox(\alpha,\beta)$.
 \end{itemize}

 We introduce some definitions to describe the intersection of 
 $V_{\mathbb{R}} (Q)$ with the boundaries of $\LBox(\alpha,\beta)$, 
 $\RBox(\alpha,\beta)$, $\LBox^+(\alpha,\beta)$, $\RBox^+(\alpha,\beta)$, 
 $\LBox^-(\alpha,\beta)$ and $\LBox^-(\alpha,\beta)$.

\begin{notation}
 Denote by 
 \begin{itemize}
 \item $V_{\alpha^-}=V_{\mathbb{R}} (Q)\cap (\{\alpha^-\} \times (\gamma^-,
 \gamma^+))$  and $V_{\alpha^+}=V_{\mathbb{R}} (Q)\cap (\{\alpha^+\} \times (\gamma^-, \gamma^+))$, ordered by increasing value of $y$,
 \item  $V_{\alpha^-}^{=\beta}=V_{\mathbb{R}} (Q)\cap \{(\alpha^-,\beta)\}$, $V_{\alpha^-}^{<\beta}= V_{\mathbb{R}} (Q)\cap ( \{\alpha^-\} \times (\gamma^-,
 \beta))$,  $V_{\alpha^-}^{>\beta}= V_{\mathbb{R}} (Q)\cap (\{\alpha^-\} \times (\beta,\gamma^+))$, ordered by increasing value of $y$,
  \item  $V_{\alpha^+}^{=\beta}=V_{\mathbb{R}} (Q)\cap  \{(\alpha^+,\beta)\}$, $V_{\alpha^+}^{<\beta}=V_{\mathbb{R}} (Q)\cap (\{\alpha^+\} \times (\gamma^-,
 \beta))$,  $V_{\alpha^+}^{>\beta}= V_{\mathbb{R}} (Q)\cap (\{\alpha^+\} \times (\beta,\gamma^+))$, ordered by increasing value of $y$,
\item $V_{\alpha}^{<\beta}=V_{\mathbb{R}} (Q)\cap (\{\alpha\} \times (\gamma^-,
 \beta))$, $V_{\alpha}^{>\beta}=V_{\mathbb{R}} (Q)\cap (\{\alpha\} \times (\gamma^-,
 \beta))$, ordered by increasing value of $y$,
  \end{itemize}
and 
 \begin{itemize}
 \item  $H_{\gamma^-}^{= \alpha^-}=V_{\mathbb{R}} (Q)\cap \{(\alpha^-,\gamma^-)\}$, $H_{\gamma^-}^{ < \alpha}=V_{\mathbb{R}} (Q)\cap ((\alpha^-,\alpha) \times \{\gamma^-\})$,   $H_{\gamma^-}^{= \alpha}=V_{\mathbb{R}} (Q)\cap \{(\alpha,\gamma^-)\}$, $H_{\gamma^-}^{ >\alpha}=V_{\mathbb{R}} (Q)\cap ((\alpha,\alpha^+) \times \{\gamma^-\})$, $H_{\gamma^-}^{= \alpha^+}=V_{\mathbb{R}} (Q)\cap \{(\alpha^+,\gamma^-)\}$,
 \item $H_{\gamma^+}^{= \alpha^-}=V_{\mathbb{R}} (Q)\cap \{(\alpha^-,\gamma^+)\}$, $H_{\gamma^+}^{ < \alpha}=V_{\mathbb{R}} (Q)\cap ((\alpha^-,\alpha) \times \{\gamma^+\})$,   $H_{\gamma^+}^{= \alpha}=V_{\mathbb{R}} (Q)\cap \{(\alpha,\gamma^+)\}$, $H_{\gamma^+}^{ >\alpha}=V_{\mathbb{R}} (Q)\cap ((\alpha,\alpha^+) \times \{\gamma^+\})$, $H_{\gamma^+}^{= \alpha^+}=V_{\mathbb{R}} (Q)\cap \{(\alpha^+,\gamma^+)\}$.
  \end{itemize}
\end{notation}

In summary:
 \begin{itemize}
 \item the boundary points of $\LBox^+(\alpha,\beta)$ in  $V_{\mathbb{R}} (Q)$
are $V_{\alpha^-}^{>\beta} \cup H_{\gamma^+}^{=\alpha^-}\cup H_{\gamma^+}^{<\alpha} \cup H_{\gamma^+}^{=\alpha} \cup V_{\alpha}^{>\beta}$,
\item  the boundary points of $\RBox^+(\alpha,\beta)$ in  $V_{\mathbb{R}} (Q)$
are   $V_{\alpha^+}^{>\beta} \cup H_{\gamma^+}^{=\alpha^+} \cup H_{\gamma^+}^{>\alpha} \cup H_{\gamma^+}^{=\alpha} \cup V_{\alpha}^{>\beta}$,
 \item the boundary points of $\LBox^-(\alpha,\beta)$ in  $V_{\mathbb{R}} (Q)$
are  $V_{\alpha^-}^{<\beta} \cup H_{\gamma^-}^{=\alpha^-} \cup H_{\gamma^-}^{<\alpha}+H_{\gamma^-}^{=\alpha} \cup V_{\alpha}^{<\beta}$,
\item the boundary points of $\RBox^-(\alpha,\beta)$ in  $V_{\mathbb{R}} (Q)$
are  $V_{\alpha^+}^{<\beta} \cup H_{\gamma^+}^{=\alpha^+}\cup H_{\gamma^+}^{>\alpha} \cup H_{\gamma^-}^{=\alpha}\cup V_{\alpha}^{<\beta}$.
 \end{itemize}

\begin{figure}
\centering
 \includegraphics[width=12cm,height=9cm]{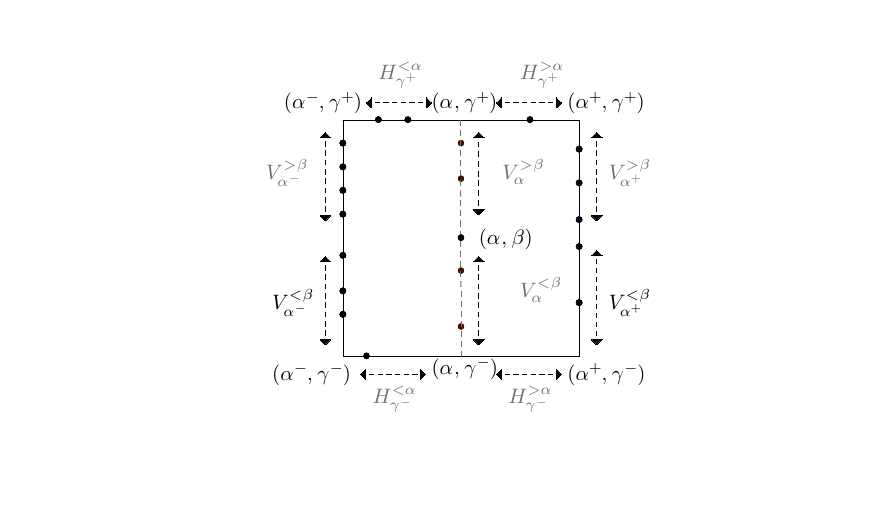} 
\caption{\label{description_lists}Illustration of the notations}
\end{figure}

\begin{exemple}
In the case of Figure \ref{description_lists},  if  the signs of the slopes are all positive on the 
 boundary points,
the matching process is intuitive and illustrated by Figure \ref{matching}.
\begin{figure}
\centering
\includegraphics[width=9cm,height=7cm]{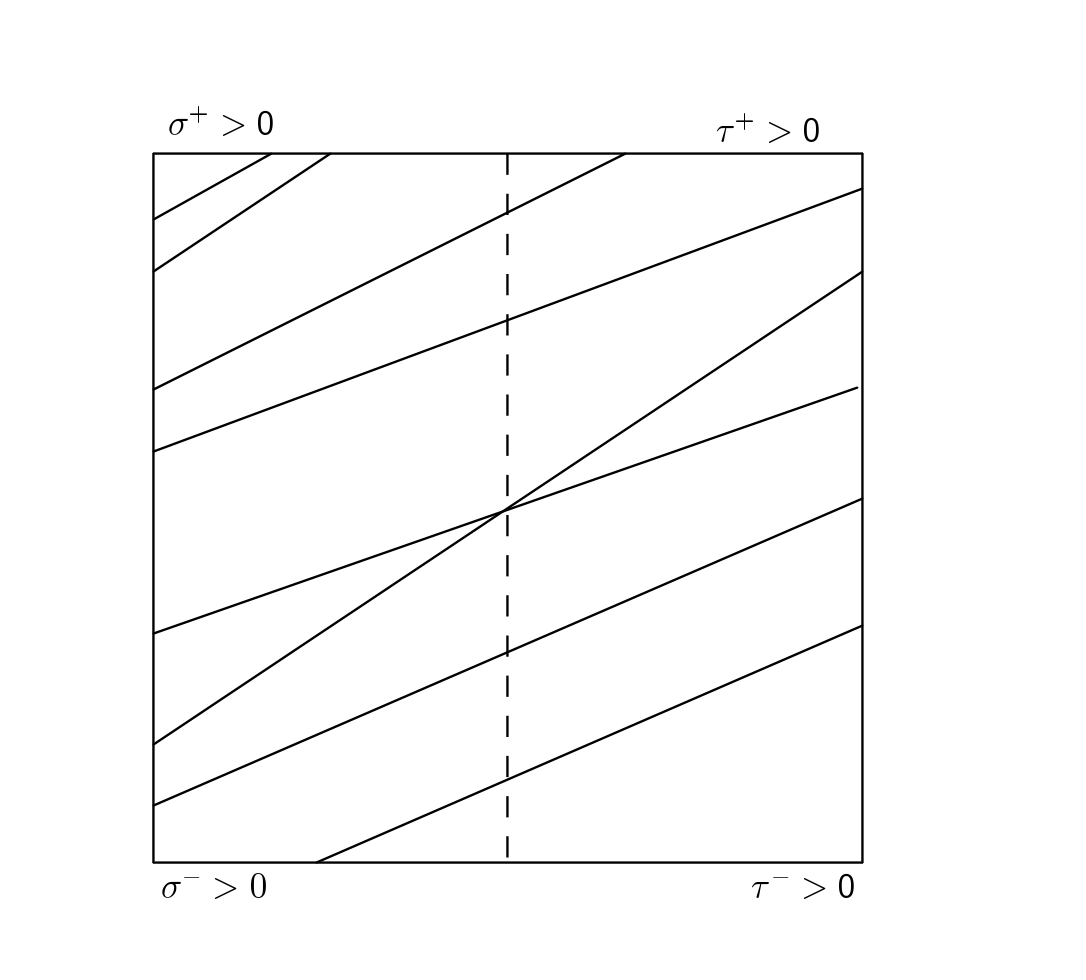}
\caption{
\label{matching}
Example of matching}
\end{figure}
   \end{exemple}

The aim of the next pages is to transform the intuition used in the example into explicit definitions and an algorithm covering all cases.

\begin{proposition}
 \label{forcorrectness}
It holds
 \begin{itemize}
 \item [a)]  
 All the boundary points of  $\LBox^+(\alpha,\beta)$ in  $V_{\mathbb{R}} (Q)$  have the same slope sign.
 \item [b)]
All the boundary points of  $\RBox^+(\alpha,\beta)$ in  $V_{\mathbb{R}} (Q)$  have the same slope sign.
  \item [c)] 
 All the boundary points of  $\LBox^-(\alpha,\beta)$ in  $V_{\mathbb{R}} (Q)$ have the same slope sign.
 \item [d)] 
 All the boundary points of  $\RBox^-(\alpha,\beta)$ in  $V_{\mathbb{R}} (Q)$  have the same slope sign.
  \end{itemize}
\end{proposition}
 
\begin{proof} 
We prove only a), the proofs for b), c) and d) being similar.

Suppose that $H_{\gamma^+}^{=\alpha^-}+H_{\gamma^+}^{<\alpha}+H_{\gamma^+}^{=\alpha}$ has at least two elements, and that $\left( x_1, \gamma^+ \right)$ and $\left( x_2, \gamma^+ \right)$ $x_1<x_2$ are two 
 points of $H_{\gamma^+}^{=\alpha^-}+H_{\gamma^+}^{<\alpha}+H_{\gamma^+}^{=\alpha}$ with different slope signs and denote by $C_1$ and $C_2$ the connected components of $V_{\mathbb{R}}(Q)$ inside $[\alpha^-,\alpha)\times \R$ such that $(x_1,\gamma^+)\in \bar C_1$ and 
 $(x_2,\gamma^+)\in \bar C_2$.
 Since
 $\bbox(\alpha,\beta)$ is an adjacency box (cf Definition \ref{def:adjac}), $C_1$ and $C_2$ are the graphs of two monotonous semi-algebraic continuous functions $\varphi_1$ and $\varphi_2$ defined on $[\alpha^-,\alpha)$. At $\alpha$, $\varphi_1$ and $\varphi_2$ have limits of opposite signs, so that the sign of the limit of $\varphi_1-\varphi_2$ is well defined at $\alpha$.
 The signs of $\varphi_1(x_1)-\varphi_2(x_1)$ at $x_1$ and the sign of the limit of $\varphi_1-\varphi_2$ at $x_2$ are opposite, so $\varphi_1(x)=\varphi_2(x)$ for a value $x\in (x_1,\alpha)$, which is impossible because such a point would be a singular point of $V_{\mathbb{R}}(Q)$  inside $[\alpha^-,\alpha)\times \R$.
 
 Suppose that $V_{\alpha}^{>\beta}$ (resp. $V_{\alpha^-}^{>\beta}$) has at least two elements, and that $\left(\alpha, y_1\right)$ and $\left(\alpha,y_2\right)$ 
 (resp.   $\left(\alpha^-, y_1\right)$ and $\left(\alpha^-,y_2\right)$ ), $y_1<y_2$, are two
 points of $V_{\alpha}^{>\beta}$ (resp. $V_{\alpha}^{>\beta}$)  with different slope signs and denote by $C'_1$ and $C'_2$ the connected components of $V_{\mathbb{R}}(Q)$ inside $\R\times (\beta,\gamma^+]$ containing them.
 Since 
  $\bbox(\alpha,\beta)$ is an adjacency box (cf Definition \ref{def:adjac}),
   $C'_1$ and $C'_2$ are the graphs of two monotonous semi-algebraic continuous functions $\psi_1$ and $\psi_2$ defined on 
 $(\beta,\gamma^+]$.
 The signs of $\psi_1-\psi_2$ at $y_1$ and $y_2$ are opposite, so $\psi_1(y)=\psi_2(y)$ for a value $y\in (y_1,y_2)$, which is impossible because such a point would be a singular point of $V_{\mathbb{R}}(Q)$  inside $\R\times (\beta,\gamma^+]$.

Suppose  that $H_{\gamma^+}^{=\alpha^-}+H_{\gamma^+}^{<\alpha}+H_{\gamma^+}^{=\alpha}\not=\emptyset$ and $V_{\alpha}^{>\beta}\not=\emptyset$, and let  $(\alpha,y)$ be the element of $\bar V_{\alpha}^{>\beta}$ with highest value of $y$.

(i) Suppose  that the slope sign of $(\alpha,y)$ is negative 
and the slope sign of the elements of $H_{\gamma^+}^{=\alpha^-}+H_{\gamma^+}^{<\alpha}+H_{\gamma^+}^{=\alpha}$ is positive. Let $(x,\gamma^+)$ be the element of $H_{\gamma^+}^{<\alpha}+H_{\gamma^+}^{=\alpha}$ with highest value of $x$. Denote by $C'_1$ and $C'_2$ the connected components of $V_{\mathbb{R}}(Q)$ inside $\R\times (\beta,\gamma^+]$ such that $(\alpha,y)\in \bar C_1$ and $(x,\gamma^+)\in \bar C_2$.  
Since  
 $\bbox(\alpha,\beta)$ is an adjacency box (cf Definition \ref{def:adjac}), $C'_1$ (resp. $C'_2$) is the graph of an increasing (resp. decreasing) semi-algebraic continuous functions $\psi_1$ (resp. $\psi_2$) defined on 
 $(\beta,\gamma^+]$.  The matching point of $(\alpha,y)$ is not $(x,\gamma^+)$ because these two points have opposite slope signs so that the limit of $\varphi_1$ at $\gamma^+$ is strictly less than  $x$.
 The signs of  the limit $\psi_1-\psi_2$ at $\beta$  (resp.  $\gamma^+$) is positive  (resp. negative) and at $y_2$  is negative  (resp. positive) , so $\psi_1(y)=\psi_2(y)$ for a value $y\in (\beta,\gamma^+)$, which is impossible because such a point would be a singular point of $V_{\mathbb{R}}(Q)$  inside $\R\times (\beta,\gamma^+)$.

(ii) Suppose  now that the slope sign  of
 $(\alpha,y)$ is positive and the slope sign of the elements of $H_{\gamma^+}^{=\alpha^-}+H_{\gamma^+}^{<\alpha}+H_{\gamma^+}^{=\alpha}$ is negative.
 
If $Q(\alpha,\gamma^+)=0$, this is impossible by (i), using the symmetry with respect to the line $X=\alpha$.

Otherwise there is an element  $(x,\gamma^+)$ in $H_{\gamma^+}^{=\alpha^-}+H_{\gamma^+}^{<\alpha}$. Denote by $C_1$ (resp. $C_2$) the connected components of $V_{\mathbb{R}}(Q)$ inside $[\alpha^-,\alpha)\times \R$ such that $(x,\gamma^+)\in\bar C_1$ (resp.
 $(\alpha,y)\in \bar C_2$). 
 Since
  $\bbox(\alpha,\beta)$ is an adjacency box (cf Definition \ref{def:adjac}),
   $C_1$ (resp. $C_2$) is the graph of an increasing (resp. decreasing) semi-algebraic continuous functions $\varphi_1$ (resp. $\varphi_2$) defined on $[\alpha^-,\alpha)$. The matching point of $(x,\gamma^+)$ is not $(\alpha,y)$ because these two points have opposite slope signs so that the limit of $\varphi_2$ at $\alpha$ is strictly less than  $y$. So the limit of $\varphi_1-\varphi_2$ at $\alpha^-$ (resp. $\alpha$) is negative (resp. positive). It follows that $\varphi_1(x)=\varphi_2(x)$ for a value $x\in (\alpha^-,\alpha)$, which is impossible because such a point would be a singular point of $V_{\mathbb{R}}(Q)$  inside $[\alpha^-,\alpha)\times \R$.
  
   Suppose finally that $H_{\gamma^+}^{=\alpha^-}+H_{\gamma^+}^{<\alpha}+H_{\gamma^+}^{=\alpha}\not=\emptyset$ and  $V_{\alpha^-}^{>\beta}\not=\emptyset$, and let  $(\alpha,y)$ be $V_{\alpha^-}^{>\beta}\not=\emptyset$ with highest value of $y$.

(i) Suppose  that the slope sign of $(\alpha,y)$ is negative 
and the slope sign of the elements of $H_{\gamma^+}^{=\alpha^-}+H_{\gamma^+}^{<\alpha}+H_{\gamma^+}^{=\alpha}$ is positive. Let $(x,\gamma^+)$ be the element of $H_{\gamma^+}^{<\alpha}+H_{\gamma^+}^{=\alpha}$ with smallest value of $x$. Denote by $C'_1$ and $C'_2$ the connected components of $V_{\mathbb{R}}(Q)$ inside $\R\times (\beta,\gamma^+]$ such that $(\alpha,y)\in \bar C_1$ and $(x,\gamma^+)\in \bar C_2$.  
Since  
 $\bbox(\alpha,\beta)$ is an adjacency box (cf Definition \ref{def:adjac}), $C'_1$ (resp.) $C'_2$ is the graph of an increasing (resp. decreasing) semi-algebraic continuous functions $\psi_1$ (resp. $\psi_2$) defined on 
 $(\beta,\gamma^+]$ .  The matching point of $(\alpha,y)$ is not $(x,\gamma^+)$ because these two points have opposite slope signs so that the limit of $\varphi_1$ at $\gamma^+$ is strictly less than  $x$.
 The signs of  the limit $\psi_1-\psi_2$ at $\beta$  (resp.  $\gamma^+$) is positive  (resp. negative) and at $y_2$  is negative  (resp. positive) , so $\psi_1(y)=\psi_2(y)$ for a value $y\in (\beta,\gamma^+)$, which is impossible because such a point would be a singular point of $V_{\mathbb{R}}(Q)$  inside $\R\times (\beta,\gamma^+)$.

(ii) Suppose  now that the slope sign  of
 $(\alpha,y)$ is positive and the slope sign of the elements of $H_{\gamma^+}^{=\alpha^-}+H_{\gamma^+}^{<\alpha}+H_{\gamma^+}^{=\alpha}$ is negative.
 
If $Q(\alpha^-,\gamma^+)=0$ , this is impossible by (i), using the symmetry with respect to the line $X=\alpha$.

Otherwise there is an element  $(x,\gamma^+)$ in $H_{\gamma^+}^{=\alpha^-}+H_{\gamma^+}^{<\alpha}$. Denote by $C_1$ (resp. $C_2$) the connected components of $V_{\mathbb{R}}(Q)$ inside $[\alpha^-,\alpha)\times \R$ such that $(x,\gamma^+)\in\bar C_1$ (resp.
 $(\alpha,y)\in \bar C_2$). 
 Since
  $\bbox(\alpha,\beta)$ is an adjacency box (cf Definition \ref{def:adjac}),
   $C_1$ (resp. $C_2$) is the graph of an increasing (resp. decreasing) semi-algebraic continuous functions $\varphi_1$ (resp. $\varphi_2$) defined on $[\alpha^-,\alpha)$. The matching point of $(x,\gamma^+)$ is not $(\alpha,y)$ because these two points have opposite slope signs so that the limit of $\varphi_2$ at $\alpha$ is strictly less than  $y$. So the limit of $\varphi_1-\varphi_2$ at $\alpha^-$ (resp. $\alpha$) is negative (resp. positive). It follows that $\varphi_1(x)=\varphi_2(x)$ for a value $x\in (\alpha^-,\alpha)$, which is impossible because such a point would be a singular point of $V_{\mathbb{R}}(Q)$  inside $[\alpha^-,\alpha)\times \R$.
\end{proof}

\begin{notation}
We denote by $\sigma^+$ (resp. $\sigma^-$) the slope signs of the boundary points of 
$\LBox^+(\alpha,\beta)$ (resp. $\LBox^-(\alpha,\beta)$) in  $V_{\mathbb{R}} (Q)$ 
and by 
$\tau^+$ (resp. $\tau^-$) 
of the boundary points of 
$\RBox^+(\alpha,\beta)$ (resp. $\RBox^-(\alpha,\beta)$) in  $V_{\mathbb{R}} (Q)$ 

By convention when
there are no boundary points of 
$\LBox^+(\alpha,\beta)$ (resp. $\RBox^-(\alpha,\beta)$)  in  $V_{\mathbb{R}} (Q)$ 
 we define $\sigma^+$ (resp. $\tau^-$) to be  $>0$ 
 and  when
 there are no boundary points of 
$\LBox^-(\alpha,\beta)$, (resp. $\RBox^+(\alpha,\beta)$)) in  $V_{\mathbb{R}} (Q)$ 
 we define $\sigma^-$, (resp. $\tau^+$) to be  $<0$.
\end{notation}

\begin{algorithm}[Number of 
curve segments
 arriving at a singular point]
\label{algoconnect} 
\noindent {\bf 1. Number of 
curve  segments
 arriving to the left} \\
{\bf Input:}  $\#V_{\alpha^-}$, $\#V_{\alpha^-}^{=\beta}$,$\#V_{\alpha^-}^{>\beta}$, $\#V_{\alpha^-}^{<\beta}$,$\#H_{\gamma^-}^{= \alpha^-}$, $\#H_{\gamma^-}^{ < \alpha}$, $\#H_{\gamma^-}^{= \alpha}$, $\#H_{\gamma^+}^{= \alpha^-}$, $\#H_{\gamma^+}^{ < \alpha}$, $\#H_{\gamma^+}^{= \alpha}$, 
 $\#V_{\alpha}^{>\beta}$, $\#V_{\alpha}^{<\beta}$,  $\sigma^-$ and $\sigma^+$ \\
{\bf Output:} 
the number $\textsc{Left}$ of the
curve  segments
 ending at 
$(\alpha,\beta)$ to the left of $\alpha$
\begin{itemize}
\item If $\sigma^+>0$ and $\sigma^-<0$, compute
\begin{equation}\label{totheleft1}
\textsc{Left}=\#V_{\alpha^-}^{=\beta}.
\end{equation}
\item If $\sigma^+<0$ and $\sigma^->0$, compute
\begin{equation}\label{totheleft2}
\textsc{Left}=\#V_{\alpha^-}+\#H_{\gamma^+}^{ = \alpha}+\#H_{\gamma^+}^{< \alpha}+\#H_{\gamma^-}^{ = \alpha^-}+\#H_{\gamma^-}^{< \alpha}-(\#V_{\alpha}^{>\beta }+\#V_{\alpha}^{<\beta}).
\end{equation}
\item If $\sigma^+>0$ and $\sigma^->0$ and $Q(\alpha^-,\beta)=0$, compute
\begin{equation}\label{totheleft3}
\textsc{Left}=1+\#V_{\alpha^-}^{<\beta}-(\#H_{\gamma^-}^{ < \alpha}+\#H_{\gamma^-}^{= \alpha}+\#V_{\alpha}^{<\beta}).
\end{equation}
\item If $\sigma^+>0$ and $\sigma^->0$ and $Q(\alpha^-,\beta)\not=0$, compute
\begin{equation}\label{totheleft4}
\textsc{Left}=\#V_{\alpha^-}+\#H_{\gamma^-}^{ = \alpha^-}+\#H_{\gamma^-}^{<\alpha}-(\#H_{\gamma^+}^{ < \alpha}+\#H_{\gamma^+}^{= \alpha}+\#V_\alpha^{>\beta }+\#V_\alpha^{<\beta}).
\end{equation}
\item If $\sigma^+<0$ and $\sigma^-<0$ and $Q(\alpha^-,\beta)=0$, compute
\begin{equation}\label{totheleft5}
\textsc{Left}=1+\#V_{\alpha^-}^{>\beta}-(\#H_{\gamma^+}^{ < \alpha}+\#H_{\gamma^+}^{= \alpha}+\#V_{\alpha}^{>\beta})
\end{equation}
\item If $\sigma^+<0$ and $\sigma^-<0$ and $Q(\alpha^-,\beta)\not=0$, compute
\begin{equation}\label{totheleft6}
\textsc{Left}=\#V_{\alpha^-}+\#H_{\gamma^+}^{= \alpha^- }+\#H_{\gamma^+}^{< \alpha}-(\#H_{\gamma^-}^{ < \alpha}+\#H_{\gamma^-}^{= \alpha}+\#V_{\alpha}^{>\beta }+\#V_{\alpha}^{<\beta})
\end{equation}
\end{itemize}
\noindent {\bf 2.  Number of 
curve segments
 arriving to the right} \\
{\bf Input:} $\#V_{\alpha^+}$, $\#V_{\alpha^+}^{=\beta}$,$\#V_{\alpha^+}^{>\beta}$, $\#V_{\alpha^+}^{<\beta}$, $\#H_{\gamma^-}^{= \alpha^+}$, $\#H_{\gamma^-}^{ > \alpha}$, $\#H_{\gamma^-}^{= \alpha}$, $\#H_{\gamma^+}^{= \alpha^+}$, $\#H_{\gamma^+}^{ > \alpha}$, $\#H_{\gamma^+}^{= \alpha}$,$\#V_{\alpha}^{>\beta}$, $\#V_{\alpha}^{<\beta}$,  
$\tau^-$ and $\tau^+$ 
\\
{\bf Output:}
 the number $\textsc{Right}$ of the 
curve segments 
 ending at 
$(\alpha,\beta)$  to the right of $\alpha$.
\begin{itemize}
\item If $\tau^+<0$ and $\tau^->0$, compute
\begin{equation}\label{totheright1}
\textsc{Right}=\#V_{\alpha^+}^{=\beta}
\end{equation}
\item If $\tau^+>0$ and $\tau^-<0$, compute
\begin{equation}\label{totheright2}
\textsc{Right}=\#V_{\alpha^+}+\#H_{\gamma^+}^{ = \alpha}+\#H_{\gamma^+}^{> \alpha}+\#H_{\gamma^-}^{ = \alpha^+}+\#H_{\gamma^-}^{> \alpha}-(\#V_{\alpha}^{>\beta }+\#V_{\alpha}^{<\beta})
\end{equation}
\item If $\tau^+<0$ and $\tau^-<0$ and $Q(\alpha^+,\beta)=0$, compute
\begin{equation}\label{totheright3}
\textsc{Right}=1+\#V_{\alpha^+}^{<\beta}-(\#H_{\gamma^-}^{ > \alpha}+\#H_{\gamma^-}^{= \alpha}+\#V_{\alpha}^{<\beta})
\end{equation}
\item If $\tau^+<0$ and $\tau^-<0$ and $Q(\alpha^+,\beta)\not=0$, compute
\begin{equation}\label{totheright4}
\textsc{Right}=\#V_{\alpha^+}+\#H_{\gamma^-}^{ = \alpha^+}+\#H_{\gamma^-}^{>\alpha}-(\#H_{\gamma^+}^{ > \alpha}+\#H_{\gamma^+}^{= \alpha}+\#V_\alpha^{>\beta }+\#V_\alpha^{<\beta})
\end{equation}
\item If $\tau^+>0$ and $\tau^->0$ and $Q(\alpha^+,\beta)=0$, compute
\begin{equation}\label{totheright5}
\textsc{Right}=1+\#V_{\alpha^+}^{>\beta}-(\#H_{\gamma^+}^{ > \alpha}+\#H_{\gamma^+}^{= \alpha}+\#V_{\alpha}^{>\beta})
\end{equation}
\item If $\tau^+>0$ and $\tau^->0$ and $Q(\alpha^+,\beta)\not=0$, compute
\begin{equation}\label{totheright6}
\textsc{Right}=\#V_{\alpha^+}+\#H_{\gamma^+}^{= \alpha^+ }+\#H_{\gamma^+}^{>\alpha}-(\#H_{\gamma^-}^{ > \alpha}+\#H_{\gamma^-}^{= \alpha}+\#V_{\alpha}^{>\beta }+\#V_{\alpha}^{<\beta})
\end{equation}
\end{itemize}
\end{algorithm}

\begin{proof}[Proof of correctness of Algorithm \ref{algoconnect}]
We denote by $\textsc{Left}^+$ 
 the number of curve segments arriving at $(\alpha,\beta)$ inside  $\LBox^+(\alpha,\beta)$ .

If $\sigma^+>0$ and ($\sigma^->0$ or $Q(\alpha^-,\beta)=0$), there is no matching point of a boundary point of  $\LBox^+(\alpha,\beta)$ contained in   $\LBox^-(\alpha,\beta)$ and no curve segment starting from  $V_{\alpha^-}^{>\beta}$ is of type 2, so that $$\textsc{Left}^+=0.$$  This is clear by
Remark \ref{propcurvesegment} and by the fact that there are no singular point in $\LBox(\alpha,\beta)$.

If $\sigma^+<0$  and $\sigma^-<0$ and $Q(\alpha^-,\beta)=0$, there is no matching point of a boundary point of  $\LBox^+(\alpha,\beta)$ contained in  $\LBox^-(\alpha,\beta)$ and 
 $$\textsc{Left}^+=\#V_{\alpha^-}^{>\beta}-(\#H_{\gamma^+}^{ < \alpha}+\#H_{\gamma^+}^{= \alpha}+\#V_{\alpha}^{>\beta}).$$  This is clear by the fact that there are no singular point in $\LBox(\alpha,\beta)$.
 
 If $\sigma^+<0$  and $\sigma^-<0$ with $Q(\alpha^-,\beta)\not=0$, there maybe matching points of a boundary points of  $\LBox^+(\alpha,\beta)$ contained in  $\LBox^-(\alpha,\beta)$ and 
 $$\textsc{Left}=\#V_{\alpha^-}+\#H_{\gamma^+}^{= \alpha^- }+\#H_{\gamma^+}^{< \alpha}-(\#H_{\gamma^-}^{ < \alpha}+\#H_{\gamma^-}^{= \alpha})-(\#V_{\alpha}^{>\beta }+\#V_{\alpha}^{<\beta}).$$
  This is clear by
Remark \ref{propcurvesegment} and by the fact that there are no singular point in $\LBox(\alpha,\beta)$.
 
It is easy to prove the statement, using symmetries.
\end{proof}

\begin{exemple}
In the case of Figure \ref{description_lists},  for the left side we have $Q(\alpha^-,\beta)\not=0$,
 $\#V_{\alpha^-}=7$, $\#H_{\gamma^+}^{= \alpha^-}=0$, $\#H_{\gamma^+}^{ < \alpha}=2$, $\#H_{\gamma^+}^{= \alpha}=0$, $\#H_{\gamma^-}^{= \alpha^-}=0$, $\#H_{\gamma^-}^{ < \alpha}=1$, $\#H_{\gamma^-}^{= \alpha}=0$, $\#V_{\alpha}^{>\beta}=2$, $\#V_{\alpha}^{<\beta}=2$. 
  For the right side we have $Q(\alpha^+,\beta)\not=0$,
$\#V_{\alpha^+}=5$, $\#H_{\gamma^+}^{= \alpha^+}=0 $, $\#H_{\gamma^+}^{ > \alpha}=1$, $\#H_{\gamma^+}^{= \alpha}=0$, $\#H_{\gamma^-}^{= \alpha^+}=0$, $\#H_{\gamma^-}^{ > \alpha}=0$, $\#H_{\gamma^-}^{= \alpha}=0$, $\#V_{\alpha}^{>\beta}=2$, $\#V_{\alpha}^{<\beta}=0$. 
 
 Hence, if  $\sigma^+>0$,  $\sigma^->0$,  from formula (\ref{totheleft4}) $\textsc{Left}= 2$. 
If   
$\tau^+>0$, $\tau^->0$. we have 
 from formula (\ref{totheright6})
 $\textsc{Right}= 2$. 
 This is illustrated by 
Figure \ref{matching}.
  
  Always in the case of Figure \ref{description_lists},
if we have $\sigma^+>0$,  $\sigma^-<0$, $\tau^+>0$, $\tau^-<0$
it follows from formula (\ref{totheleft1}) $\textsc{Left}= 0$ and from formula (\ref{totheright2})
 $\textsc{Right}= 2$. This is illustrated by 
 Figure \ref{matching2}.
 
 \begin{figure}
\centering
\includegraphics[width=9cm,height=7cm]{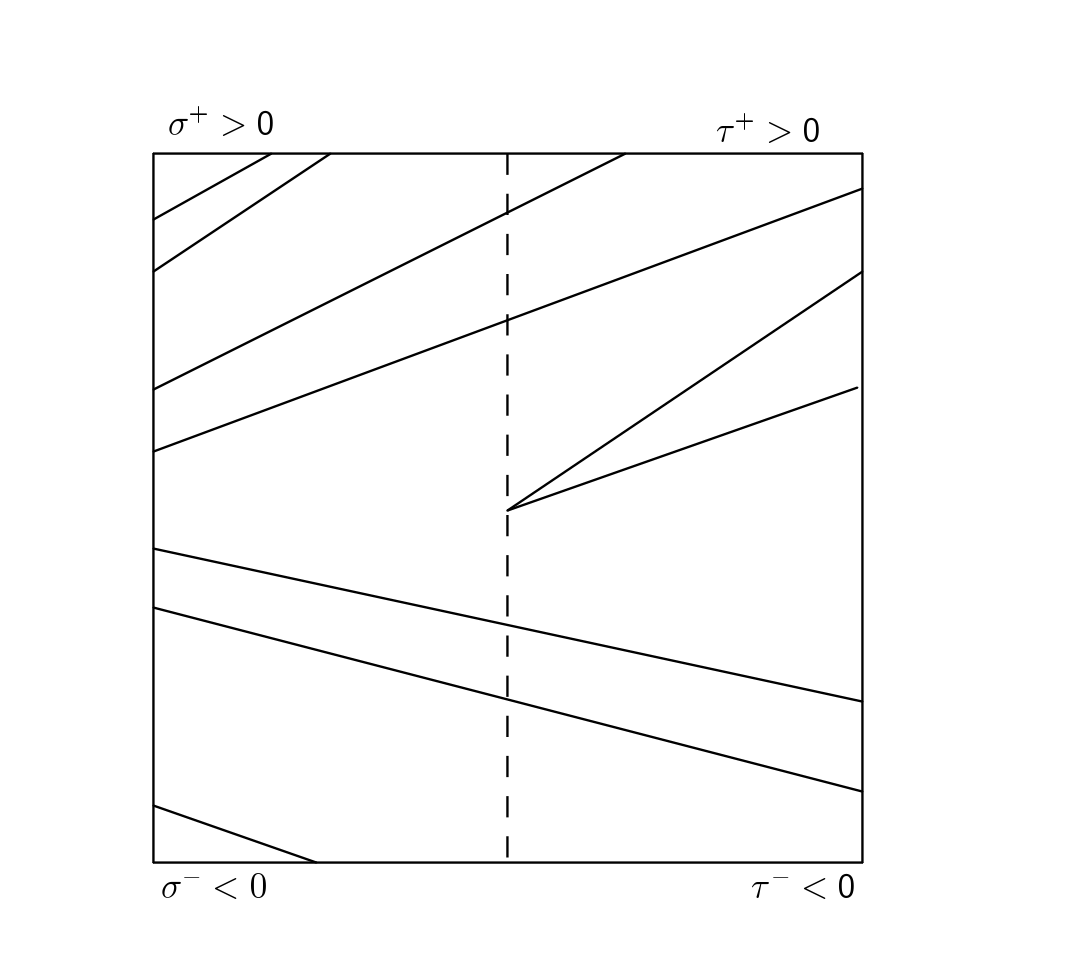}
\caption{\label{matching2}Another example of matching}
  \end{figure} 
 
   \end{exemple}

\subsection{Computing 
adjacency boxes at singular points}

Our aim is now to compute adjacency boxes at singular points.

We need to introduce the following definitions.

In order to distinguish the $Y$-critical points which belong to a horizontal line contained in the  curve and the other $Y$-critical points,
we write
\begin{equation}\label{barP}
Q(X,Y)=c(Y)\cdot \tilde{Q}(X,Y)
\end{equation}
 with $c(Y)\in \Z [Y]$ and  $\tilde{Q}(X,Y)\in  \Z [X, Y]$ such that $\tilde{Q}(X,y)$
never identically vanishes for  any $y \in \mathbb{C}.$ In more geometric terms, we separate the horizontal lines contained in $V_{\mathbb{C}} (Q)$ from the remaining part $V_{\mathbb{C}} (\tilde{Q})$ of the curve, which does not contain any horizontal lines. 

This is done as follows. Let  $Q(X,Y):=c_{\deg_X(Q)}(Y)X^{\deg_X(Q)}+\ldots+c_0(X),$ with $\deg_X(Q)\le d$. Let $c(X)$ the gcd of all the coefficients $c_i(X)$. 
We compute $c(Y)$ and $\tilde{Q}(X,Y)$ using $\tO(d^4+d^3\tau)$ bit operations by Propositions~\ref{Mignotte} and~\ref{gcd-comp}.
The polynomials  
$c(Y)$ and $\tilde{Q}(X,Y)$ have magnitude bounded by $(d,\tau+d+\log (d+1))$.

We introduce the following definitions
 \begin{eqnarray}
  D_Y(Y)& := &\Disc_X (Q) (Y), 
  \label{barD}\\
   S_X(X) &:=& 
   D_X(X) \cdot \res_Y (\tilde Q, \partial_X \tilde Q) (X), \label{eqYcrit}\\
    S_Y(Y) &:=& D_Y(Y)   \cdot 
    \res_X (
    Q, 
    \partial_Y 
    Q) (Y), \label{eqYcritbis}\\
    T_X(X)&:=& S_X(X)\cdot (S_X^{\star})'(X), \label{def:RX}\\
    T_Y(Y)&:=& S_Y(Y)\cdot (S_Y^\star)'(Y).\label{def:RY}
 \end{eqnarray}
 where $(S_X^\star)'$ and  $(S_X^\star)'$ are the derivatives of the square-free parts of $S_X$ and $S_Y$, respectively.
 
 In geometric terms
 \begin{itemize}
 \item the zeroes of $D_Y$ are the projections on the $Y$ axis of the  $Y$-critical points
 or the values $\gamma$ such that $\deg(Q(X,\gamma))<\deg_X(Q)-1$,
 \item the zeroes of $S_X$ are the zeroes of $D_X$ and the projections on the $X$-axis of the $Y$-critical points which do not belong to an horizontal line entirely contained in $V_{\mathbb{R}} (Q)$,
  \item the zeroes of $S_Y$ are the zeroes of $D_Y$ and the projections on the $Y$-axis of the $X$-critical points,
  \item  the zeroes of $T_X(X)$ are the zeroes of $S_X$ and the zeroes of the derivative of its square-free part,
   \item  the zeroes of $T_Y(Y)$ are the zeroes of $S_Y$ and the zeroes of the derivative of its square-free part. 
 \end{itemize}
 Note that between $\alpha_i$ and $\alpha_{i+1}$ there is always one root of $(S_X^\star)'(X)$ and that 
 between two roots of $D_Y$ there is always one root of  $(S_Y^\star)'(Y)$.
 
 \begin{lemma}\label{magni}
 All the polynomials  $D_X,D_Y,S_X,S_Y,T_X,T_Y$ are of magnitude 
bounded by
$(O(d^2), O(d\tau+d^2))$ and can be all computed with a bit complexity $\tO(d^4 \tau+d^5)$.
  \end{lemma}
  \begin{proof}.
Use Proposition  \ref{gcd-comp} and \cite[Prop.~8.46]{BPRbook2},\cite[\S 11.2]{GG}
  \end{proof}

  We 
   denote by $\xi_1,\ldots,\xi_{N'}$, with
  \begin{equation}
\xi_1<\ldots<\xi_{N'},
\end{equation}
 the real roots of $(S_X^\star)'(X)$, and 
 $\xi_0=-\cau(T_X)$, $\xi_{N'+1}=\cau(T_X)$ using Notation \ref{cauchybounds}, such that $\xi_0$ (resp. 
 $\xi_{N'+1}$) is smaller (resp. bigger) than all real roots of $S_X$ and $(S_X^\star)'$.  Remark that $\xi_0$ and $\xi_{N'+1}$ are rational numbers of bitsize $\tO (d\tau+d^2)$.
 
 For every $i=1,\ldots,N$, we denote
 $\alpha^-_i$ and $\alpha^+_i$ the two successive elements of $$\{\xi_0,\xi_1,\ldots,\xi_{N'},\xi_{N'+1}\}$$
defined by
 $\alpha_i\in (\alpha^-_i,\alpha^+_i)$. Note that $\alpha_i$ is the
 only root of $S_X$ (hence of $D_X$)  in $(\alpha^-_i,\alpha^+_i)$. 
 
 In the $Y$-direction, we denote by $\gamma_1,\ldots,\gamma_M$, with
 $$\gamma_1<\ldots<\gamma_M,$$ the real roots of $S_Y(Y)$.
   We also denote by 
   \begin{equation}
   \eta_1<\ldots<\eta_{M'},
     \end{equation}
   the real roots of $(S_Y^\star)'(Y)$, and 
 $\eta_0=-\cau(T_Y)$, $\eta_{M'+1}=\cau(T_Y))$ using Notation \ref{cauchybounds}, such that $\eta_0$ (resp. 
 $\eta_{M'+1}$) is smaller (resp. bigger) than all real roots of $S_Y$ and $(S_Y^\star)'$. Remark that $\eta_0$ and $\eta_{M'+1}$ are rational numbers of bitsize ${O} (d\tau+d^2)$.
 For every $k=1,\ldots,M$,
 we denote by
 $\gamma^-_k$ and $\gamma^+_k$ the two successive elements of $$\{\eta_0,\eta_1,\ldots,\eta_{M'},\eta_{M'+1}\}$$
defined by
  $\gamma_i\in (\gamma^-_i,\gamma^+_i)$. Note that $\gamma_i$ is the
 only root of $S_Y$ (hence of $D_Y$)  in $(\gamma^-_i,\gamma^+_i)$.
 
For every singular point $(\alpha_i,\beta_{i,j})$, notice that $\beta_{i,j}$ is a root of $S_Y$. 
We define
$k(i,j)$ 
as
the index such that 
$\gamma_{k(i,j)}=\beta_{i,j}$.

For each $i=1,\ldots,N$, we denote by
\begin{equation}
\eta^-_{i,1}<\ldots<\eta^-_{i,m^-_i} (\mathrm{resp. }  \eta^+_{i,1}<\ldots<\eta^+_{i,m^+_i})
\end{equation}
 the real roots of $Q(\alpha_i^-,Y)$ (resp. $Q(\alpha_i^+,Y)$).
  
  For each $k=1,\ldots,M$, we denote by
\begin{equation}
\xi^-_{k,1}<\ldots<\xi^-_{k,n^-_k} (\mathrm{resp. }  \xi^+_{k,1}<\ldots<\xi^+_{k,n^+_k})
\end{equation}
 the real roots of $Q(X,\gamma_k^-)$ (resp.  $Q(X,\gamma_k^+)$).

We suppose the computations of Theorem \ref{sing-fibers0} have already been performed and prove the following result. Its 
Corollary \ref{cor:adjacency} provides adjacency boxes at critical points.

\begin{theorem}\label{sing-fibers}

 Using $\tO (d^5 \tau + d^6)$ bit operations, we can carry out the following computations:
 \begin{itemize}
 \item[(a.1)]
  Dyadic intervals $H_i,V_{i,j}$
for $1\le i \le N, 1 \le j \le m_i$
such that
$ H_i$
 is well-isolating for $\alpha_i$  as a root of $T_X$,
 and $ V_{i,j}$
is well-isolating for $\beta_{i,j}$ as a root of $Q(\alpha_i,Y)$.
  \item[(a.2)]
 Dyadic intervals $H^-_i, V^-_{i,j}$, (resp. $H^+_i, V^+_{i,j}$)
  for $1\le i \le N, 1 \le j \le m^-_i$ (resp. $j\le  m^+_i$)
such that
 $ H^-_i$ (resp  $ H^+_i$)
 is well-isolating for $\alpha^-_i$ (resp. $\alpha^+_i$) as a root of $T_X$,
and $V^-_{i,j}$ (resp. $V^+_{i,j}$)
 is well-isolating for $\eta^-_{i,j}$ (resp. $\eta^+_{i,j}$)  as a root of $Q(\alpha^-_i,Y)$ (resp.  $Q(\alpha^+_i,Y)$).
 Moreover the signs of $\partial_XQ(\alpha^-_i,\eta^-_{i,j})$ and $\partial_YQ(\alpha^-_i,\eta^-_{i,j})$ (resp. $\partial_XQ(\alpha^+_i,\eta^+_{i,j})$ and 
 $\partial_YQ(\alpha^+_i,\eta^+_{i,j})$ are also part of the output.
 \item[(b)] Dyadic intervals $V^-_{k}, H^-_{k,\ell}$ (resp. $V^+_{k}, H^+_{k,\ell}$)
   for $1\le k \le M, 1 \le \ell \le n^-_k$ (resp. $\ell \le  n^+_k$) such that
$V^-_{k}$ (resp. $H^+_{k}$) is well-isolating for $\gamma^-_{k}$ (resp. $\gamma^+_{k}$) as a root of 
$T_Y$ 
and $H^-_{k,\ell}$ (resp.  $H^+_{k,\ell}$) is well-isolating for $\xi^-_{k,\ell}$ (resp. $\xi^+_{k,\ell}$) as a root of 
$Q(X,\gamma^-_k)$ (resp.  $Q(X,\gamma^+_k)$).
 Moreover the signs of $\partial_XQ(\xi^-_{k,\ell},\gamma^-_{k})$ and $\partial_YQ(\xi^-_{k,\ell},\gamma^-_{k})$ (resp. $\partial_XQ(\xi^+_{k,\ell},\gamma^+_{k})$ and 
 $\partial_YQ(\partial_XQ(\xi^+_{k,\ell},\gamma^+_{k}))$ are also part of the output.
 \item[(c.1)]  Dyadic intervals $V_k$, for  $k=1,\ldots,M$, that are well-isolating for $\gamma_k$  as roots of $T_Y$ and for each 
 $i=1,\ldots,N$, $j\in 
\textsc{SingInd}_i$,
 the index $k(i,j)$ such that $\beta_{i,j}=\gamma_{k(i,j)}$
 \item[(c.2)] For each $i=1,\ldots,N$, $j\in 
\textsc{SingInd}_i$,
the intervals $H^-_{k(i,j),\ell}$ for $\ell=1,\ldots,n^-_{k(i,j),\ell}$  (resp. $H^+_{k(i,j),\ell}$ for  $\ell=1,\ldots,n^+_{k(i,j),\ell}$)
contain at most one of the three points  $\alpha^-_i$, $\alpha_i$, $\alpha^+_i$.
\item[(c.3)] For each $i=1,\ldots,N$, $j\in 
\textsc{SingInd}_i$,
 the intervals $V^-_{i,j'}$ for $j'=1,\ldots,m^-_{i}$  (resp. $V^+_{i,j'}$ for $j'=1,\ldots,m^+_{i}$) contain at most one of the two points   $\gamma^-_{k(i,j)}$  $\gamma^+_{k(i,j)}$.
  \end{itemize}
Moreover, it holds  that
 \begin{equation}
 \sum^{N}_{i = 1} \lambda (H_i) (\text{resp. }\ \sum^{N}_{i = 1} \lambda (H^-_i),\sum^{N}_{i = 1} \lambda (H^+_i))
 \in \tilde O (d^3 \tau+d^4), \label{eqboxa1}
 \end{equation} 
 \begin{equation}
 \sum^{N}_{i = 1} \sum^{m_i}_{j = 1} \lambda (V_{i, j}) (\mathrm{resp.}\  \sum^{N}_{i = 1} \sum^{m^-_i}_{j = 1} \lambda (V^-_{i, j}),\sum^{N}_{i = 1} \sum^{m^+_i}_{j = 1} \lambda (V^+_{i, j}))
\in \tilde O (d^3
 \tau+d^4) \label{eqboxa2} .
 \end{equation}
  \begin{equation}
 \sum^{M}_{k = 1}   \lambda (V_k)(\mathrm{resp.}\ \sum^{M}_{k = 1}  \lambda (V_k^-),\ \sum^{M}_{k = 1}   \lambda (V_k^+))  \in \tilde O (d^3
 \tau+d^4) \label{eqboxe1} .
 \end{equation}
 \begin{equation}
 \sum^{M}_{k = 1} \sum^{n^-_{k}}_{\ell = 1} \lambda (H^-_{k,\ell})  (\mathrm{resp.}\ \sum^{M}_{k = 1} \sum^{n^+_{k}}_{\ell = 1} \lambda (H^+_{k,\ell})) 
\in \tilde O (d^3
 \tau+d^4) \label{eqboxe2} .
 \end{equation}
\end{theorem}

\begin{remark}\label{amb}
Note that  Theorem \ref{sing-fibers} (c.2) does not decide whether
$\xi^+_{k(i,j),\ell}<\alpha_i$,
$\xi^+_{k(i,j),\ell}=\alpha_i$ or $\xi^+_{k(i,j),\ell}>\alpha_i$
in the case where
$H^+_{k(i,j),\ell}$ intersects $H_i$.
 Indeed, we do not compute the sign of $Q(\alpha_i,\gamma_{k(i,j)}^+)$.
It would be of course possible to obtain this information using exact computations, but not within the complexity bounds we are aiming for in this paper: 
to the best of our knowledge, the computation for this decision would exceed $\tO (d^5 \tau + d^6)$ bit operations, 
because the distance between $\alpha_i$ and $\xi^+_{k(i,j),\ell}$ can be very small and it would be very costly to refine  $H^+_{k(i,j),\ell}$ and $H_i$ enough for this decision.

The same remark holds for similar statements covering  the other cases considered in (c.2) and (c.3).
\end{remark}

\begin{proof}
For (a.1) first note, using Lemma \ref{magni}, that $T_X$ is a polynomial of magnitude
bounded by
 $(N,\Lambda)\in (O(d^2),\tO(d\tau+d^2))$.  Hence, Proposition \ref{sagraloff-isolation-integer} and Proposition \ref{comparingroots}, 
we can compute well-isolating intervals for the real roots of $T_X$ and identify the real roots of $D_X$ in a number of bit operations bounded by $\tO(d^5\tau+d^6)$.
 According to 
Proposition~\ref{prop:computinggcddeg} (with $R:=D_X$ and $F:=Q(X,Y)$), we may further compute 
$\deg(\gcd(Q(\alpha_i,Y),\partial_Y Q(\alpha_i,Y)))$ for all $i$ from $1$ to $N$  using $\tO(d^5\tau+d^6)$ bit operations. Now, from Proposition~\ref{thm:costisolation}, we conclude that  using also $\tO(d^5\tau+d^6)$ bit operations we can further compute 
well-isolating intervals for all real roots of the 
polynomials $Q(\alpha_i,Y)$.

For (a.2), using Proposition \ref{sagraloff-isolation-integer} and Proposition \ref{comparingroots}, 
we can compute well-isolating intervals for the real roots of $T_X$ and identify the real roots of $(D_X^\star)'$ in a number of bit operations bounded by $\tO(d^5\tau+d^6)$.
 According to 
Proposition~\ref{prop:computinggcddeg} (with $R:=(T_X^\star)'$ and $F:=Q(X,Y)$), we may further compute $\deg Q(\xi_i,Y)$.
Now, from Proposition~\ref{thm:costisolation}, we conclude that within a number of bit operations bounded by $\tO(d^5\tau+d^6)$ we can further compute well-isolating intervals for all real roots of the 
polynomials  
$$Q(\xi_i,Y) \cdot \partial_XQ(\xi_i,Y)\cdot \partial_YQ(\xi_i,Y),$$
and identify the roots of $Q(\xi_i,Y)$ for each $i=1,\ldots,N'$ by  Proposition \ref{thm:comparinginfibers}.
It remains to compute isolating intervals for all the 
 roots of the 
polynomials  
$$Q(\xi_0,Y) \cdot \partial_XQ(\xi_0,Y)\cdot \partial_YQ(\xi_0,Y),$$
and
$$Q(\xi_N',Y)\cdot  \partial_XQ(\xi_N',Y)\cdot \partial_YQ(\xi_N',Y).$$
It is then easy to identify $\alpha_i^-$ and $\alpha_i^+$ as well as $H^-_i,V^-_{i,j'},V^+_i,V^+_{i,j'}$ as part of the results of the preceding computations.
 The bound on the sum of the bitsizes of the 
intervals $H^-_i,V^-_{i,j'},H^+_i,V^+_{i,j'}$
follows from Part (a.1) of Proposition~\ref{thm:costisolation} and Proposition \ref{thm:comparinginfibers}.

Part (b) is entirely similar to Part (a.2), exchanging the role of $X$ and $Y$.

In Part (c.1), the computation of $T_Y$ takes $\tO(d^4 \tau+d^5)$ according to Lemma \ref{magni}, and the computation of the $V_k$, $k=1,\ldots,M$ uses a number of bit operations bounded by $\tO(d^5\tau+d^6)$ from Proposition \ref{sagraloff-isolation-integer} and Proposition \ref{comparingroots} since $T_Y$ is of magnitude
bounded by
 $(O(d^2),O(d\tau+d^2))$.
 
 We define $I^1_i:=\textsc{SingInd}_i$. 
In order to compute the numbers $k(i,j)$ for each 
$\beta_{i,j}, j \in \textsc{SingInd}_i$, we proceed in rounds enumerated by $\ell=1,2,3,\ldots$. In the $\ell$-th round, we refine the isolating intervals for all $i$ and all roots $\beta_{i,j}$ with $j \in I^\ell_i$ to a size less than $2^{-2^{\ell}}$. If the corresponding isolating interval 
$V_{i,j}$
 intersects with at most one isolating interval $V_k$ for the roots of 
$T_Y$, we know that $k=k(i,j)$. After having treated all elements in $I^{\ell}_i$, we set $I^{\ell+1}_i$ to be the set of all critical indices in $I^\ell_i$ for which the isolating interval
 $V_{i,j}$ for $\beta_{i,j}$ intersects more than one of the intervals $V_k$. That is, $I^{\ell}_i$ is the set of all critical indices  for which $k(i,j)$ is not known after the $\ell$-th round. We then proceed with the $(\ell+1)$-st round. We stop as soon as $I^{\ell}_i$ becomes empty for every $i=1,\ldots,N$, in which case, $k(i,j)$ is determined for all $X$-critical points.

We use the polynomial $R_Y$ defined in Proposition \ref{multmult1}, with $F=Q$, remembering that the roots of $R_Y$ contain the projections of the $Y$-critical points of $Q$.
Notice that, for each critical point $(\alpha_i,\beta_{i,j})$, we succeed in round $\ell_{i,j}$, where
$2^{\ell_{i,j}}$ is bounded by $O(|\log(\sep(\beta_{i,j},T_Y R_Y))|)$. That is, $j \notin I^{\ell}_i$ for any $\ell>\ell_{i,j}$.

In addition, the cost of the test for checking whether the interval $V_{i,j}$ intersects with exactly one isolating interval $V_k$ is bounded by $\tO(d^3\tau+d^4)$ bit operations in each round. Indeed, 
 we need to consider only $O(\log (d))$ comparisons between corresponding
endpoints of the occurring intervals and each comparison is carried out with a precision bounded by $\tO(d^3\tau+d^4)$, using that, from the amortized bounds on the separation of the roots  (Proposition \ref{corodisc}) $$O(|\log(\sep(\beta_{i,j},T_Y R_Y))|)\in O(d^3\tau+d^4).$$ Since there are $O(d^2)$ many critical points, the 
total cost for the comparisons is thus bounded by $\tO(d^5\tau+d^6)$.

 It remains to estimate the cost for refining the intervals $V_{i,j}$ to a width less than $2^{-2^{\ell_{i,j}}}$ for all $i,j$.
Using again $O(|\log(\sep(\beta_{i,j},T_Y R_Y))|)\in O(d^3\tau+d^4)$, 
we are done after $\kappa$ rounds for  
\begin{equation}\label{kappa1}
\kappa=\max_{i,j} \ell_{i,j}+1 \in O(\log(d^4+d^3 \tau)).
\end{equation}
 According to Proposition~\ref{thm:costisolation}, this cost is bounded by
\[
\sum_{\ell=1}^\kappa \tO(d^5\tau+d^6+ 2^{\ell}d^2\cdot \sum_{i=1}^{N}\mu_{i}^{[\ell]})\in \tO(d^5\tau+d^6)+\tO(d^2\cdot\sum_{\ell} 2^{\ell}\cdot \sum_{i=1}^{N}\mu_{i}^{[\ell]}),
\]
where, denoting $\mu_{i,j}=\mult(\beta_{i,j},Q(\alpha_i,Y))$
\[
\mu_{i}^{[\ell]}:=
\begin{cases}
\max_{j \in I^{\ell}_i} 
\mu_{i,j}&\text{if  } I^\ell_x\not= \emptyset\\
0 &\text{otherwise.}
\end{cases}
\]

Hence, it suffices to show that 
$$\sum_{\ell=1}^\kappa 2^{\ell}\cdot \sum_{i=1}^{N}\mu_{i}^{[\ell]}\in\tO(d^3\tau+d^4).$$
 If  $I^\ell_i\not=\emptyset$, let $j^{[\ell]}_i\in I^\ell_i$ be such that $
 \mu_{i,j^{[\ell]}_i}=\mu_{i}^{[\ell]}$. In other words, $j^{[\ell]}_i$ is the critical index $I^{\ell}_i$ over $\alpha_i$ which maximizes the multiplicity within the fiber. We may thus write
\[
\sum_{\ell=1}^\kappa 2^{\ell}\cdot \sum_{i=1}^{N}\mu_{i}^{[\ell]}=\sum_{\ell=1}^\kappa 2^{\ell}\sum_{i,I^\ell_i\not=\emptyset}
 \mu_{i,j^{[\ell]}_i}
\]
Obviously, each critical index $(i,j)$
 appears in the latter sum a number of times that is bounded by $\kappa\in O(\log(d^3\tau+d^4))$. In addition, since $ (\alpha_i,\beta_{i,j})\notin I_{\kappa}$ 
and $2^{\ell_{i,j}}\in O(|\log(\sep(\beta_{i,j},T_Y R_Y))|)$, it follows that,
\[
\sum_{\ell=1}^\kappa 2^{\ell}\sum_{i=1}^N \sum_{j\in I^\ell_i}
\mu_{i,j}
\in \tO(\sum_{i=1}^N \sum_{j\in \textsc{SingInd}_i}
\mu_{i,j}
\cdot|\log(\sep(\beta_{i,j},T_Y R_Y))|).
\]
 Since 
\[
\sum_{\ell=1}^\kappa 2^{\ell}\sum_{i,I^\ell_i\not=\emptyset}
\mu_{i}^{[\ell]}
\le \sum_{\ell=1}^\kappa 2^{\ell}\sum_{i=1}^N \sum_{j\in I^\ell_i}
\mu_{i,j},\]
we have
\[
\sum_{\ell=1}^\kappa 2^{\ell}\sum_{i,I^\ell_i\not=\emptyset}
\mu_{i}^{[\ell]}
\in \tO(\sum_{i=1}^N \sum_{j\in \textsc{SingInd}_i}
\mu_{i,j}
\cdot|\log(\sep(\beta_{i,j},T_Y R_Y))|).
\]

According to Proposition~\ref{multmult1}, it holds that, for any fixed value $\beta_{i,j}$,
\[
\sum_{(\alpha_i,\beta_{i,j}):\beta_{i,j}=\gamma}(
\mu_{i,j}-1)\le\mult(\gamma,R_Y)).
\] 
Hence, since 
$\mu_{i,j}\le 2\cdot(\mu_{i,j}-1)$
we conclude that
$$\sum_{(i,j):\beta_{i,j}=\gamma}
\mu_{i,j} \le 2\cdot \mult(\gamma,T_Y R_Y).$$
This shows that
\[
\sum_{\ell=1}^\kappa 2^{\ell} \sum_{i=1}^{N}\mu_{i}^{[\ell]}\in \tO(
\logsep(T_Y R_Y))
\in\tO(d^3\tau+d^4).
\]
using
Proposition \ref{corodisc}.

Part (c.2,c.3) can be proved as  follows.
We treat in details the case of intervals  $V^-_{i,j}$ which is the first half of c3).

For each fixed $k$,  we define the set of \emph{indices of singular places at level $k$} as
\[
\textsc{SingPl}_k:=\{i\in\{1,\ldots,N\}:(\alpha_i,\gamma_k)\in \Sing(V_{\mathbb{R}}(Q))\}.
\]
with, as before,
with
$\Sing(V_{\mathbb{R}}(Q)) := \{(x, y) \in \mathbb{R}^2 \mid Q(x, y) =\partial_Y Q(x,y)=\partial_X Q(x,y)= 0\}.$ 
Let $V^-_{i,j'}$ be the intervals as computed according to
Part (a.2). In what follows, we restrict to the set $I^1$ of all so-called \emph{bad pairs} $(i,j')$ of indices such that the  
interval $V^-_{i,j'}$ intersects two intervals $V_{k(i,j)}^-$ and $V_{k(i,j)}^+$, and define $I^1_i$ as the set of corresponding indices above $i$.
 We fix a mapping $\phi$ that maps 
each such bad pair $(i,j')$ to an arbitrary index $k(i,j)$ (there might exist more than one such index)
such that $V^-_{i,j'}$ intersects the intervals $V_{k(i,j)}^-$ and $V_{k(i,j)}^+$ and moreover $\gamma_{k(i,j)}^+-\gamma_{k(i,j)}^-$ is minimal with this property. Then, the size of the preimage of each $k(i,j)$ is upper bounded by $\textsc{SingPl}_{k(i,j)}$. Namely, for a fixed $i$, there can be at most one interval $V_{i,j'}^-$ intersecting the two intervals $V_{k(i,j)}^-$ and $V_{k(i,j)}^+$. We thus conclude, using Proposition \ref{multmult1} and its notation, with $F=Q$, that 
$$|\Phi^{-1}(k(i,j))|\le  \#\textsc{SingPl}_{k(i,j)}\le 
\mult(\gamma_{k(i,j)},R_Y) \le 
\mult(\gamma_{k(i,j)},R_YT_Y),
$$
which further implies that $I^1$ contains at most $O(d^2)$ many elements. Further notice that a pair $(i,j)$ cannot be bad if the width of the corresponding interval $V^-_{i,j'}$ is smaller than $\frac{1}{2}\cdot \sep(\gamma,T_Y)$ as the distance between  the intervals $V_k^-$ and $V_k^+$ is at least $\frac{1}{2}\cdot\min(|\gamma_{k(i,j)}^-\gamma_{k(i,j)}^-|)\le \frac{1}{2}\cdot \sep(\gamma^-_{k(i,j)},T_Y)$. Hence, in order to guarantee that the  $V^-_{i,j'}$ does not intersects two intervals $V_{k(i,j)}^-$ and $V_{k(i,j)}^+$, it is enough to refine the intervals $V^-_{i,j'}$ with $(i,j')\in I_1$ to a width smaller than $\frac{1}{2}\cdot \sep(\gamma^-_{k(i,j)},T_Y)$. 

For this,  we proceed in rounds enumerated by $\ell=1,2,3,\ldots$, where, in the $\ell$-th round, we refine all intervals $V^-_{i,j'}$ with $(i,j')\in I^{\ell}_i$ to a width less than $2^{-2^{\ell}}$. Then, we remove all pairs $j'$ from $I^\ell_i$ that are not bad anymore to obtain $I^{\ell+1}_i$. In other words, $I^{\ell+1}_i$ is the set of all pairs for which $V^-_{i,j}$  intersects  two intervals $V_{k(i,j)}^-$ and $V_{k(i,j)}^+$ after the $\ell$-th round. We then proceed with the $(\ell+1)$-st round. We stop as soon as all  $I^{\ell}_i$ becomes empty, in which case, none of the intervals $V^-_{i,j'}$ violates the condition.

Notice that an interval $V^-_{i,j'}$ is removed after $\ell_{i,j'}$ rounds, where $2^{\ell_{i,j'}}$ is bounded by $O(|\log(\sep(\gamma_{\phi((i,j'))},T_Y))|)$. That is, $j\notin I^{\ell}_i$ for any $\ell>\ell_{i,j'}$. This further implies that each root contained in $\phi(V_\ell)$ has separation smaller than $2^{1+2^{-\ell}}$. Further notice that,  from the amortized bounds on the separation of the roots  (Proposition \ref{corodisc})
$$O(|\log(\sep(\gamma_{\phi((i,j'))},T_Y))|)\in O(d^3\tau+d^4),$$ so that the test  checking whether $j'\in I^\ell_i$ is bounded by $\tO(d^3\tau+d^4)$ bit operations in each round as we need to consider only $O(\log(d))$ comparisons between corresponding
endpoints of the occurring intervals and each comparison is carried out with a precision bounded by $\tO(d^3\tau+d^4)$. Since there are $O(d^2)$ many element in each $I$, the 
total cost for the comparisons is thus bounded by $\tO(d^5\tau+d^6)$.

 It remains to estimate the cost for refining the intervals $V^-_{i,j'}$ to a width less than $2^{-2^{\ell_{i,j}}}$ for all $(i,j)\in I$.
Using again $O(|\log(\sep(\gamma_{\phi((i,j'))},T_Y))|)\in O(d^3\tau+d^4),$
$$\kappa=\max_{i,j'} \ell_{i,j'}+1\in O(\log(d^3 \tau+d^4)).$$
 According to Proposition~\ref{thm:costisolation}, this cost is bounded by
\[
\sum_{\ell=1}^\kappa \tO(d^5\tau+d^6+ 2^{\ell}d^2\cdot \lambda_{\ell})\in \tO(d^5\tau+d^6)+\tO(d^2\cdot\sum_{\ell} 2^{\ell}\cdot|\{i\mid I^\ell_i\not= \emptyset\}|,
\]
 Hence, it suffices to show that 
$$\sum_{\ell=1}^\kappa 2^{\ell}\cdot |\{i\mid I^\ell_i\not= \emptyset\}|\in\tO(d^4+d^3\tau),$$
or alternatively that
\begin{align}\label{sumoverellandgammas}
\sum_{\ell=1}^\kappa \sum_{k(i,j):\log(\sep(\gamma_{k(i,j)},R_YT_Y))<2^{1-2^\ell}}\mult(\gamma_{k(i,j)},R_YT_Y)\cdot|\log(\sep(\gamma_{k(i,j)},R_YT_Y))| \in\tO(d^3\tau+d^4).
\end{align}
as each root $\gamma_{k(i,j)}$ has at most $\mult(\gamma_{k(i,j)},R_YT_Y)$ preimages under the mapping $\phi$ and $\phi(I^\ell_i)$ contains only roots $\gamma_{k(i,j)}$ of separation smaller than $2\cdot 2^{1-2^\ell}$.
Since $\kappa\in O(\log(d^3\tau+d^4))$ and since $\sum_{k=1}^M \mult(\gamma_{k(i,j)},R_YT_Y)\cdot|\log(\sep(\gamma_k,R_YT_Y))| \in\tO(d^3\tau+d^4)
$, we conclude that the inequality in (\ref{sumoverellandgammas}) holds.

The case of intervals  $V^+_{i,j'}$ which is the second half of (c.1) is entirely similar.
We also omit the proof of (c.2) which is quite similar, exchanging the role of $X$ and $Y$ and treating first the case $\alpha_i^-,\alpha_i$ and second the case $\alpha_i,\alpha_i^+$.
\end{proof}

 Using the notation of Theorem \ref{sing-fibers}, we have
\begin{corollary}
 \label{cor:adjacency}
For every $i$ and $j\in \textsc{SingInd}_i$, $$\bbox(\alpha_i,\beta_{i,j})=[\alpha^-_i,\alpha^+_i]\times[\gamma^-_{k(i,j)},\gamma^+_{k(i,j)}]$$
is an adjacency bor for $(\alpha_i,\beta_{i,j})$.

Moreover if 
\begin{itemize}
\item $H_i^-$ does not  intersect any $H_{k(i,j),\ell}^-$ or $V_{k(i,j)}^-$  does not  intersect any $V_{i,\ell}^-$ 
\item $H_i$ does not  intersect any $H_{k(i,j),\ell}^-$ or $V_{k(i,j)}^-$  does not  intersect any $V_{i,\ell}$ 
\item $H_i^+$ does not  intersect any $H_{k(i,j),\ell}^-$ or $V_{k(i,j)}^-$  does not  intersect any $V_{i,\ell}^-$ 
\item $H_i^-$ does not  intersect any $H_{k(i,j),\ell}^+$ or $V_{k(i,j)}^+$  does not  intersect any $V_{i,\ell}^-$ 
\item $H_i$ does not  intersect any $H_{k(i,j),\ell}^+$ or $V_{k(i,j)}^+$  does not  intersect any $V_{i,\ell}$ 
\item $H_i^+$ does not  intersect any $H_{k(i,j),\ell}^+$ or $V_{k(i,j)}^+$  does not  intersect any $V_{i,\ell}^+$ 
\end{itemize}
the numbers 
$\textsc{Left}_{i,j}$ and $\textsc{Right}_{i,j}$ are known.
\end{corollary}
\begin{proof}
The fact that $\bbox(\alpha_i,\beta_{i,j})$ satisfies the conditions of Definition \ref{def:adjac} follows clearly from Theorem \ref{sing-fibers} (a) and (b) given the definition of $T_X$ and $T_Y$. 

It is clear that if  $H_i^-$ does not  intersect any $H_{k(i,j),\ell}^-$, $H_i^-$ does not  intersect any $H_{k(i,j),\ell}^-$ and $H_i^+$ does not  intersect any $H_{k(i,j),\ell}^-$, then $\#H_{\gamma_{k(i,j)}^-}^{= \alpha_i^-}=\#H_{\gamma_{k(i,j)}^-}^{= \alpha_i}=\#H_{\gamma_{k(i,j)}^-}^{= \alpha_i^+}=0$ and the quantities $\#H_{\gamma_{k(i,j)}^-}^{ < \alpha_i}$ and   $\#H_{\gamma_{k(i,j)}^-}^{ > \alpha_i}$ are determined from the output of Theorem \ref{sing-fibers}.
Moreover even if $V_{k(i,j)}^-$  intersect an interval $V_{i,\ell}^-$  we can compare $\eta_{i,\ell}^-$ and $\gamma_{k(i,h)}^-$, because we know the sign of $Q(\alpha_i^-,\gamma_{k(i,h)}^-)$.

Taking into account all the other items, it is easy to see that there is enough information to determine  from the output of Theorem \ref{sing-fibers} all the numbers necessary to compute $\textsc{Left}_{i,j}$ and $\textsc{Right}_{i,j}$ according to Algorithm \ref{algoconnect}.
\end{proof}

 \subsection{Computing the 
 number of curve segments arriving to the left and to the right of a singular point}\label{adjacencyboxes}

Our aim is to 
compute
for all adjacency boxes  $\bbox(\alpha_i,\beta_{i,j})$ associated to 
singular points $(\alpha_i,\beta_{i,j})$, the numbers $\textsc{Left}_{i,j}$ and $\textsc{Right}_{i,j}$ of 
curve segments
 arriving at
$(\alpha_i,\beta_{i,j})$ to the left and to the right using the formulas in Algorithm \ref{algoconnect}.

Since we concentrate on a fixed singular point, we can forget about the indices in the subsection; Consider 
a 
singular 
point $(\alpha,\beta)=(\alpha_i,\beta_{i,j})$, and its adjacency box  $\bbox(\alpha,\beta)$ with
$\alpha^-=\alpha^-_i,\alpha^+=\alpha^+_i$
and
$\gamma^-=\gamma^-_{k(i,j)},\gamma^+=\gamma^+_{k(i,j)}$, where
$\alpha^-,\alpha,\alpha^+$ and $\gamma^-,\gamma^+$ are given by dyadic intervals $(a^-,a'^-),(a,a'),(a^+,a'^+)$ and $(c^-,c'^-),(c^+,c'^+)$.
Denote $\textsc{Left}=\textsc{Left}_{i,j}$ and $\textsc{Right}=\textsc{Right}_{i,j}.$

The information computed in Theorem \ref{sing-fibers} is not always sufficient to determine the quantities involved in  formulas (\ref{totheleft1}, \ref{totheleft2}, \ref{totheleft3}, \ref{totheleft4}, \ref{totheleft5}, \ref{totheleft6}) and (\ref{totheright1}, \ref{totheright2}, \ref{totheright3}, \ref{totheright4}, \ref{totheright5},  \ref{totheright6}) of Algorithm \ref{algoconnect}. For example,
if at the corner $(\alpha^-,\gamma^+)$, the interval $(x,x')$
isolating a root $\xi$ of $Q(X,\gamma^+)$  contains 
$\alpha^-$ and   the interval $(y,y')$
isolating a root $\eta$ of $Q(\alpha^-,Y)$  contains 
$\gamma^+$,  we do not know the sign $Q(\alpha^-,\gamma^+)$ so we do not know whether $\xi<\alpha^-$, $\xi=\alpha^-$ or $\xi>\alpha^-$ (resp. 
$\eta<\gamma^+$, $\eta=\gamma^+$ or $\eta>\gamma^+$)
(see Remark \ref{amb})
and we cannot determine
 $\#V_{\alpha^-}$, $\#H_{\gamma^+}^{= \alpha^-}$, $\#H_{\gamma^+}^{ < \alpha}$.

So we introduce the following definition to deal with such situations.

\begin{definition}
The corner $(\alpha^-,\gamma^+)$ is
{\bf ambiguous}
  if there are intervals 
$(x,x')$ and  $(y,y')$ output by Theorem \ref{sing-fibers}
such that 
$(x,x')$, 
isolating a root $\xi$
 of $Q(X,\gamma^+)$,
intersects  $(a^-,a'^-)$
and 
$(y,y')$,
isolating a root $\eta$
 of $Q(\alpha^-,Y)$, 
 intersects  $(c^+,c'^+)$.
 We omit similar definitions for the
three
corners  $(\alpha^-,\gamma^-)$,
 $(\alpha^+,\gamma^-)$,  $(\alpha^+,\gamma^+) $.
 
 Similarly the midpoint $(\alpha,\gamma^+)$ is
 {\bf ambiguous}
   if there are intervals 
$(x,x')$ and  $(y,y')$ output by Theorem \ref{sing-fibers}
such that 
$(x,x')$,
isolating a root $\xi$,
 of $Q(X,\gamma^+)$
 contains 
$\alpha$ and 
$(y,y')$,
isolating a root $\eta$
 of $Q(\alpha,Y)$, contains 
$\gamma^+$.
We omit a
similar definition for the
other midpoint  $(\alpha,\gamma^-)$.
\end{definition}

At the ambiguous corner $(\alpha^-,\gamma^+)$, it is not possible to know the cardinals of 
$H_{\gamma^+}^{<\alpha}\cap [x,x']$, $H_{\gamma^+}^{=\alpha}\cap [x,x']$  and $V_{\alpha^-}\cap [y,y']$ since we do not know whether $\xi \in H_{\gamma^+}^{<\alpha}$ (resp.  $\eta \in V_{\alpha^-}$).

However we note that
\begin{itemize} 
\item If $\sigma^+>0$ 
(with $\partial_XQ(\alpha^-,\gamma^+)>0$ and $\partial_YQ(\alpha^-,\gamma^+)<0$) then $Q(x,\gamma^+)<0$ and $Q(x',\gamma^+)>0$ while 
$Q(\alpha^-,y)>0$ and $Q(\alpha^-,y')<0$.
\begin{itemize}
\item If $Q(\alpha^-,\gamma^+)>0$, then  $\xi<\alpha^-$ and $\eta>\gamma^+$. 
 So that $\#H_{\gamma^+}^{<\alpha}\cap [x,x']=\#V_{\alpha^-}\cap [y,y']=0
$
\item If $Q(\alpha^-,\gamma^+)<0$, then   $\xi>\alpha^-$ and $\eta<\gamma^+$. 
 So that $\#H_{\gamma^+}^{<\alpha}\cap [x,x']=\#V_{\alpha^-}\cap [y,y']=1,
$
\item If $Q(\alpha^-,\gamma^+)=0$,  then  $\xi=\alpha^-$ and $\eta=\gamma^+$. 
 So that $\#H_{\gamma^+}^{<\alpha}\cap [x,x']=\#V_{\alpha^-}\cap [y,y']=0
$
\end{itemize}

\begin{figure}[!h]
    \centering
    \begin{subfigure}[b]{0.3\textwidth}
        \includegraphics[width=\textwidth]{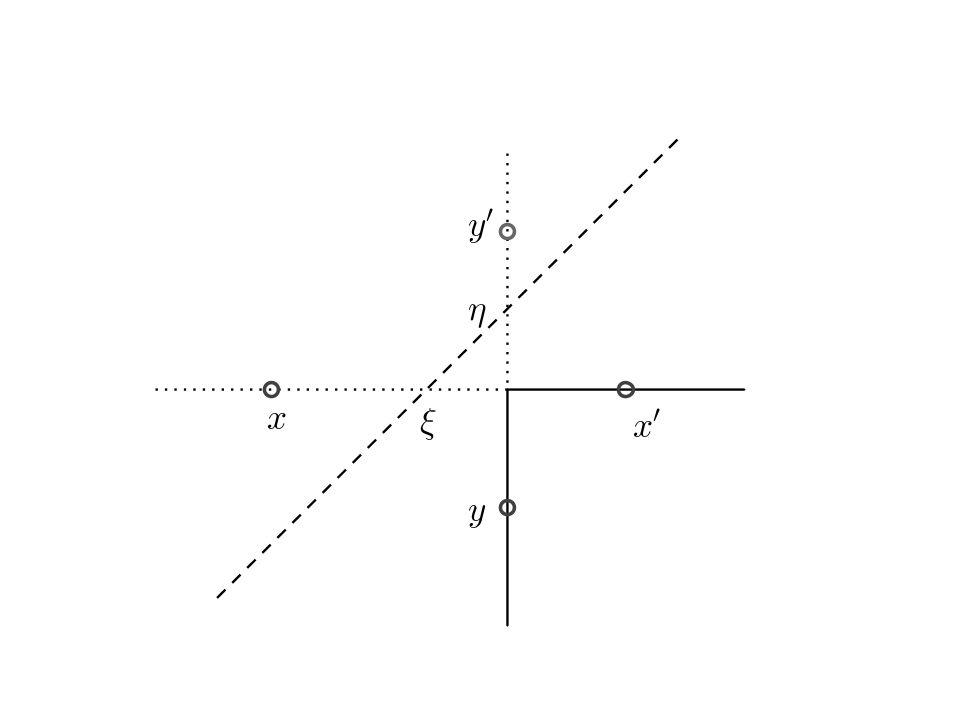}
    \end{subfigure}
    \begin{subfigure}[b]{0.3\textwidth}
        \includegraphics[width=\textwidth]{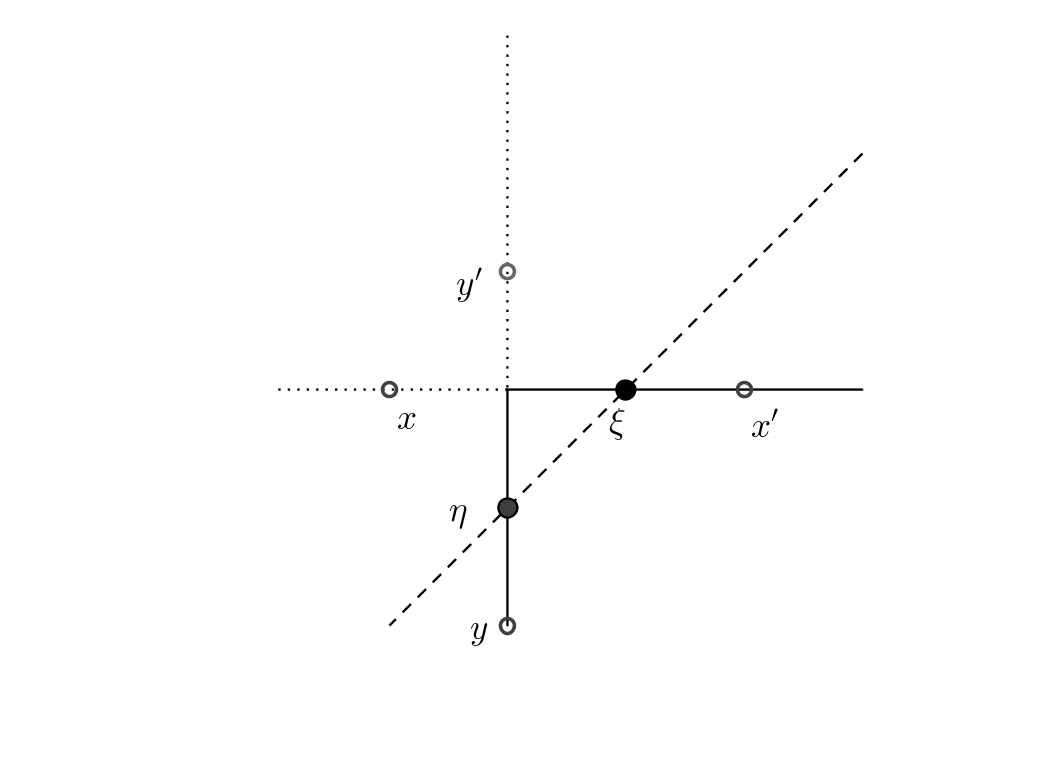}
    \end{subfigure}
    \begin{subfigure}[b]{0.3\textwidth}
        \includegraphics[width=\textwidth]{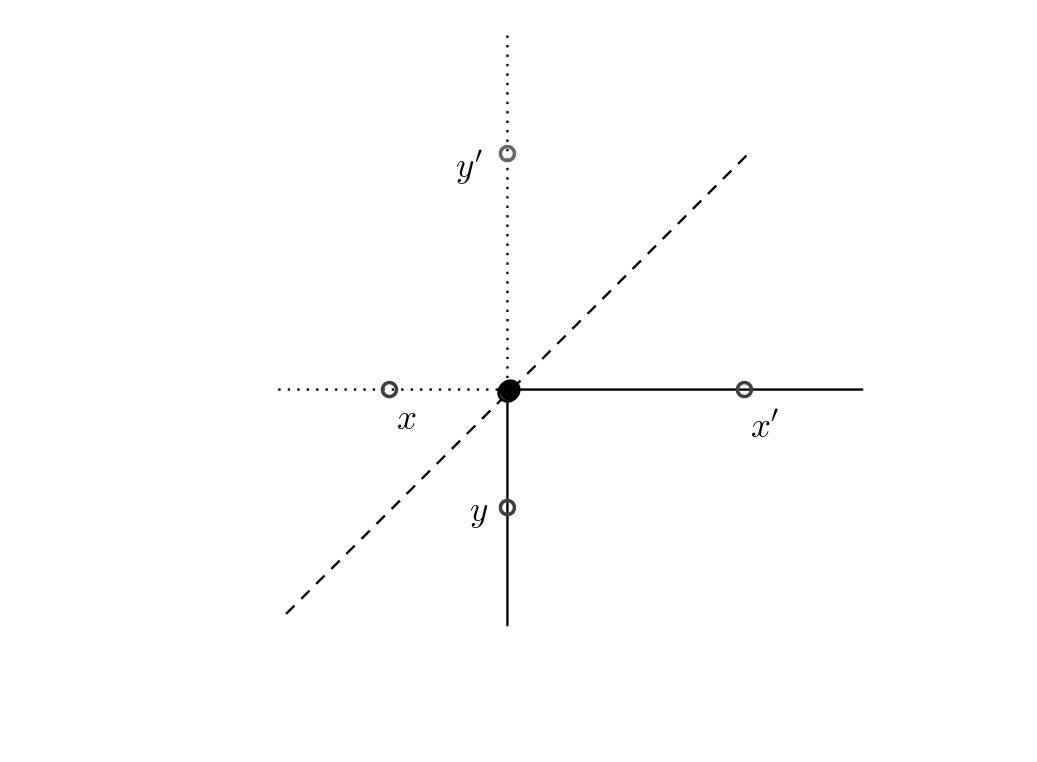}
    \end{subfigure}
    \caption{The three possibilities for an ambiguous corner with $\sigma^+ >0$}
\end{figure}

 In all cases, when  $\sigma^+>0$ (with $\partial_XQ(\alpha^-,\gamma^+)>0$ and $\partial_YQ(\alpha^-,\gamma^+)<0$)  
 $$\#H_{\gamma^+}^{<\alpha}\cap [x,x']-\#V_{\alpha^-}\cap [y,y']=0.$$
\item If $\sigma^+<0$ (with $\partial_XQ(\alpha^-,\gamma^+)>0$ and $\partial_YQ(\alpha^-,\gamma^+)>0$), then $Q(x,\gamma^+)<0$ and $Q(x',\gamma^+)>0$ while 
$Q(\alpha^-,y)<0$ and $Q(\alpha^-,y')>0$.
\begin{itemize}
\item If $Q(\alpha^-,\gamma^+)>0$, then $\xi<\gamma^-$ and $\eta<\alpha^+$.
   So that 
$\#H_{\gamma^+}^{<\alpha}\cap [x,x']=\#H_{\gamma^+}^{=\alpha^-}\cap [x,x']=0,\#V_{\alpha^-}\cap [y,y']=1
$
\item If $Q(\alpha^-,\gamma^+)<0$,
then   $\xi>\gamma^-$ and  $\eta>\alpha^+$. 
 So that 
$\#H_{\gamma^+}^{<\alpha}\cap [x,x']=1,\#V_{\alpha^-}\cap [y,y']=\#H_{\gamma^+}^{=\alpha^-}\cap [x,x']=0
$
\item If $Q(\alpha^-,\gamma^+)=0$ then  $\xi=\alpha^-$ and $\eta=\gamma^+$.
   So that
$\#H_{\gamma^+}^{<\alpha}\cap [x,x']=\#V_{\alpha^-}\cap [y,y']=0,\#H_{\gamma^+}^{=\alpha^-}\cap [x,x']=1
$
\end{itemize}

\begin{figure}[!h]
    \centering
    \begin{subfigure}[b]{0.3\textwidth}
        \includegraphics[width=\textwidth]{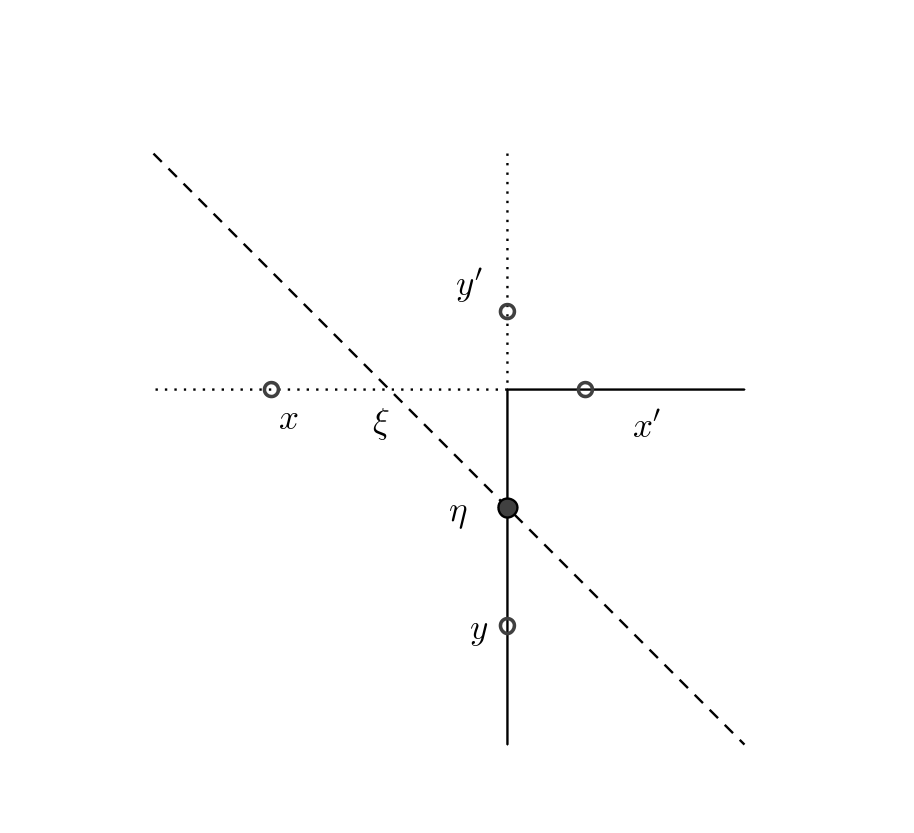}
    \end{subfigure}
    \begin{subfigure}[b]{0.3\textwidth}
        \includegraphics[width=\textwidth]{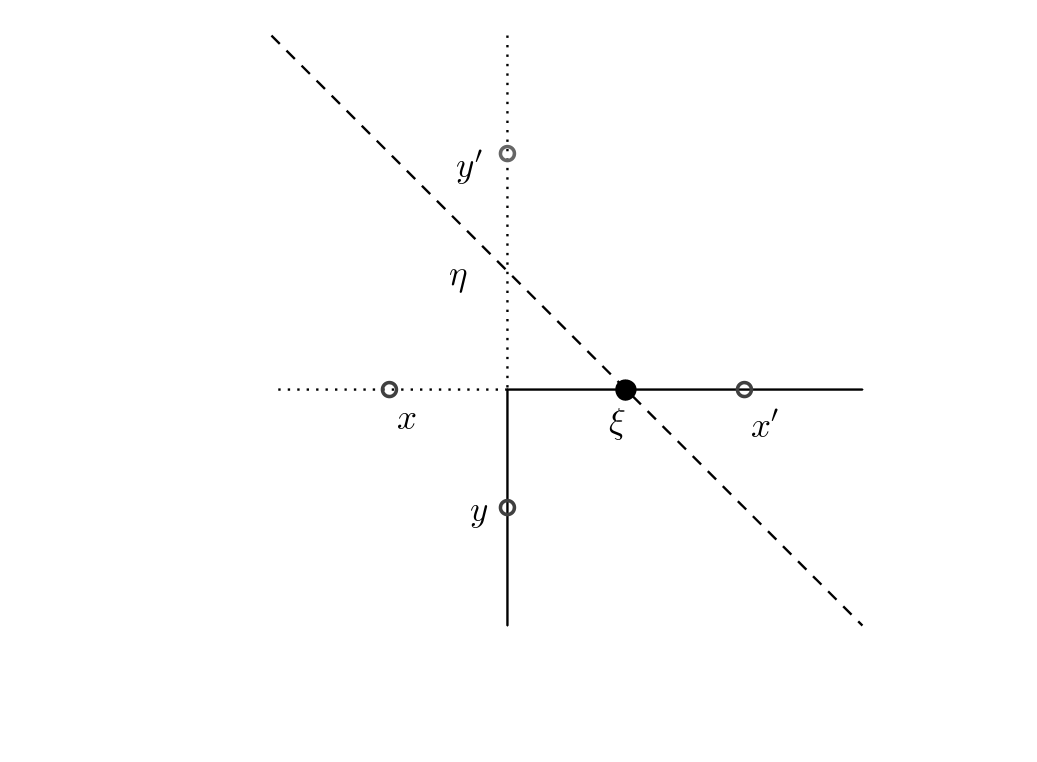}
    \end{subfigure}
    \begin{subfigure}[b]{0.3\textwidth}
        \includegraphics[width=\textwidth]{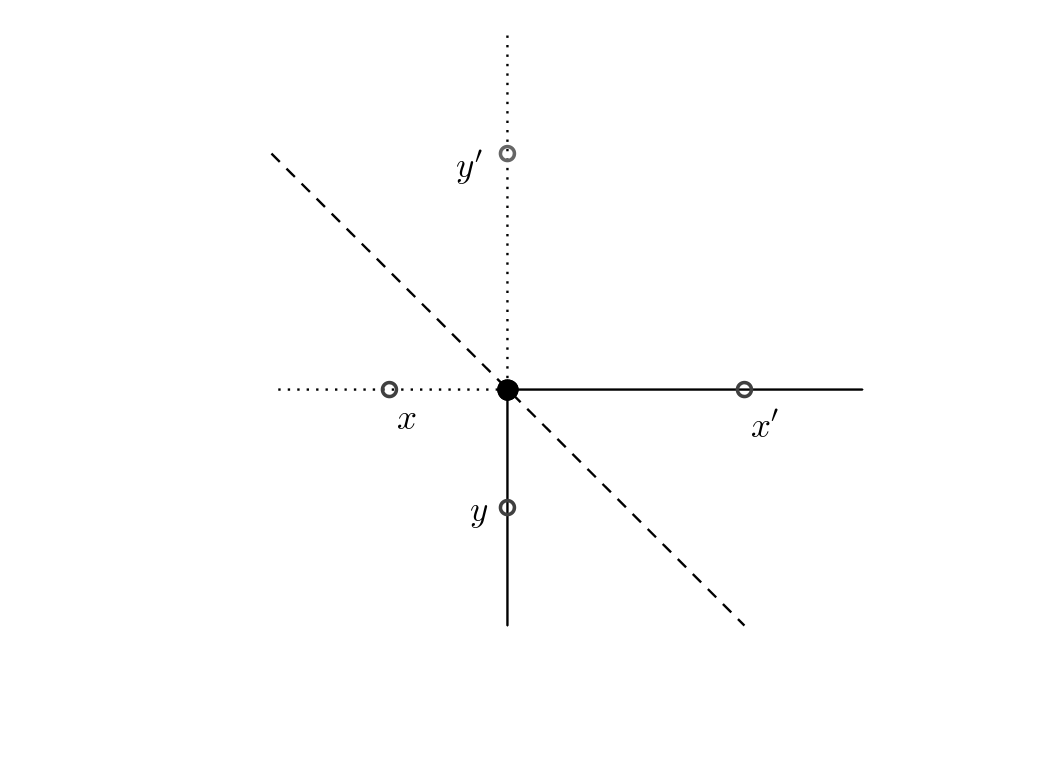}
    \end{subfigure}
    \caption{The three possibilities for an ambiguous corner with $\sigma^+<0$}
\end{figure}

In all cases, when  $\sigma^+<0$ (with $\partial_XQ(\alpha^-,\gamma^+)>0$ and $\partial_YQ(\alpha^-,\gamma^+)>0$)   $$\#H_{\gamma^+}^{<\alpha}\cap [x,x']+\#V_{\alpha^-}\cap [y,y']+\#V_{\alpha^-}\cap [y,y']=1.$$

\item Details for the remaining cases $\sigma^+>0$ with  $\partial_XQ(\alpha^-,\gamma^+)<0$ and $\partial_YQ(\alpha^-,\gamma^+)>0$ (resp. $\sigma^+<0$ with $\partial_XQ(\alpha^-,\gamma^+)<0$ and $\partial_YQ(\alpha^-,\gamma^+)<0$) are entirely similar, and omitted.
\end{itemize}

Summarizing the situation in  all the cases, the conclusion is as follows
\begin{itemize}
\item If $\sigma^+>0$ then
$$\#V_{\alpha^-}\cap [y,y']-\#H_{\gamma^+}^{<\alpha}\cap [x,x']=0
$$
\item If $\sigma^+<0$ then
$$\#V_{\alpha^-}\cap [y,y']+\#H_{\gamma^+}^{<\alpha}\cap [x,x']+\#H_{\gamma^+}^{=\alpha}\cap [x,x']=1$$

\end{itemize}
so that the sign of $Q(\alpha^-,\gamma^+)$ has no influence on $\textsc{Left}$
given the formulas of Algorithm \ref{algoconnect}.

\medskip

We now analyze the situation at an ambiguous midpoint $(\alpha,\gamma^+)$.
\begin{itemize} 
\item If $\sigma^+>0$ with $\partial_XQ(\alpha,\gamma^+)>0$ and $\partial_YQ(\alpha,\gamma^+)<0$, then $Q(x,\gamma^+)<0$ and $Q(x',\gamma^+)>0$ while 
$Q(\alpha,y)>0$ and $Q(\alpha,y')<0$.
\begin{itemize}
\item If $Q(\alpha,\gamma^+)>0$, then   $\xi<\alpha$ and  $\eta>\gamma^+$. So that 
$$\#H_{\gamma^+}^{<\alpha}\cap [x,x']=1,\#H_{\gamma^+}^{=\alpha}\cap [x,x']=\#V_\alpha^{>\beta}\cap [y,y']=0
$$
\item If $Q(\alpha,\gamma^+)<0$, then  $\xi>\alpha$ and $\eta<\gamma^+$. So that 
$$\#H_{\gamma^+}^{<\alpha}\cap [x,x']=\#H_{\gamma^+}^{=\alpha}\cap [x,x']=0,\#V_\alpha^{>\beta}\cap [y,y']=1
$$
\item If $Q(\alpha,\gamma^+)=0$,  then $\xi=\alpha$ and $\eta=\gamma^+$. So that
$$\#H_{\gamma^+}^{<\alpha}\cap [x,x']=0, \#H_{\gamma^+}^{=\alpha}\cap [x,x']=1,\#V_{\alpha}^{>\beta}\cap [y,y']=0
$$
\end{itemize}
In all cases, if $\sigma^+>0$ 
(with $\partial_XQ(\alpha^-,\gamma^+)>0$ and $\partial_YQ(\alpha^-,\gamma^+)<0$) then 
$$-\#H_{\gamma^+}^{<\alpha}\cap [x,x']-\#H_{\gamma^+}^{=\alpha}\cap [x,x']-\#V_{\alpha}^{>\beta}\cap [y,y']=-1.$$
\item If $\sigma^+<0$ with  $\partial_XQ(\alpha,\gamma^+)>0$ and $\partial_YQ(\alpha,\gamma^+)>0$, then $Q(x,\gamma^+)<0$ and $Q(x',\gamma^+)>0$ while 
$Q(\alpha,y)<0$ and $Q(\alpha,y')>0$.
\begin{itemize}
\item If $Q(\alpha,\gamma^+)>0$, then  the root $\xi$ is to the left of $\alpha$ and the root $\eta$ is under $\gamma^+$. So that 
$$\#H_{\gamma^+}^{<\alpha}\cap [x,x']=1,\#V_\alpha^{>\beta}\cap [y,y']=1
$$
\item If $Q(\alpha,\gamma^+)<0$,
then  the root $\xi$ is to the right of $\alpha$ and the root $\eta$ is above $\gamma^+$. So that 
$$\#H_{\gamma^+}^{<\alpha}\cap [x,x']=0,\#V_\alpha^{>\beta}\cap [y,y']=0
$$
\item If $Q(\alpha,\gamma^+)=0$ then $\xi=\alpha$ and $\eta=\gamma^+$. So that
$$\#H_{\gamma^+}^{< \alpha}\cap [x,x']=\#V_{\alpha}^{>\beta}\cap [y,y']=0
$$
\end{itemize}
In all cases,  when $\sigma^+<0$ 
 (with $\partial_XQ(\alpha^-,\gamma^+)>0$ and $\partial_YQ(\alpha^-,\gamma^+)>0$),
$$\#H_{\gamma^+}^{<\alpha}\cap [x,x']-\#V_{\alpha}^{>\beta}\cap [y,y']=0.$$
\item Details for the remaining cases $\sigma^+>0$ with $\partial_XQ(\alpha,\gamma^+)<0$ and $\partial_YQ(\alpha,\gamma^+)>0$ (resp. $\sigma^+<0$ with $\partial_XQ(\alpha,\gamma^+)<0$ and $\partial_YQ(\alpha,\gamma^+)<0$) are omitted.
\end{itemize}

Summarizing the situation in all the cases, the conclusion is as follows
\begin{itemize}
\item If $\sigma^+>0$ then
$$-\#H_{\gamma^+}^{<\alpha}\cap [x,x']-\#H_{\gamma^+}^{=\alpha}\cap [x,x']-\#V_{\alpha}^{>\beta}\cap [y,y']=-1
$$
\item If $\sigma^+<0$ then
$$\#H_{\gamma^+}^{<\alpha}\cap [x,x']-\#V_{\alpha}^{>\beta}\cap [y,y']=0
$$
\end{itemize}
so that the sign of $Q(\alpha,\gamma^+)$ has no influence on $\textsc{Left}$ 
given the formulas of Algorithm \ref{algoconnect}).

The analysis for the other ambiguous corners and midpoints is similar.

So we can conclude

\begin{proposition}\label{ambigousnoproblem}
The sign of  $Q$ at an ambiguous corner or midpoint has no influence on $\textsc{Left}$ and $\textsc{Right}$.
\end{proposition}

So we can decide arbitrarily that $Q$ is zero at all ambiguous corners and midpoints. In order to obtain the quantities
$$
\#V_{\alpha^-},\#H_{\gamma^-}^{= \alpha^-}, \#H_{\gamma^-}^{ < \alpha}, \#H_{\gamma^-}^{= \alpha}, \#H_{\gamma^+}^{= \alpha^-}, \#H_{\gamma^+}^{ < \alpha}, \#H_{\gamma^+}^{= \alpha}, \#V_{\alpha}^{>\beta}, \#V_{\alpha}^{<\beta}
$$
we now use the output of Theorem \ref{sing-fibers} and the extra information  that $Q$ is zero at all ambiguous corners and midpoints.
Finally, we use these quantities to compute correctly $\textsc{Left}$ 
according to the formulas of Algorithm \ref{algoconnect}. The situation is similar for $\textsc{Right}$.

All in all, we proved, using  Theorem \ref{sing-fibers} and Corollary \ref{cor:adjacency}:

\begin{proposition}\label{leftright}
Using $\tilde O(d^5\tau+d^6)$ bit-operations, we can compute $\textsc{Left}_{i,j}$ 
and  $\textsc{Right}_{i,j}$ for every $i,j$, $i=1,\ldots,N$, $j\in \textsc{CritInd}_i$.
\end{proposition}

\subsection{Vertical asymptotes}
\label{sec:asym}

We now explain how to deal with the vertical asymptotes.

The vertical asymptotes occur at values of $\alpha_i$ where $q(\alpha_i)=0$, i.e. $\deg(Q(\alpha_i,Y))<\deg_Y(Q)$.

Our aim is to compute numbers
$\textsc{Left}_{i,0}$ and $\textsc{Right}_{i,0}$ (resp $\textsc{Left}_{1,m_i+1}$ and $\textsc{Right}_{1,m_i+1}$) counting the number of asymptotes tending to $-\infty$ just before $\alpha_i$ and  just after $\alpha_i$ (resp. tending to $+\infty$ just before $\alpha_i$ and  just after $\alpha_i$ ).

There are two cases to consider.

As we have seen already in Section \ref{sec:basic}, when $\deg(Q(\alpha_i,Y))=\deg_Y(Q)-1$,
\begin{itemize}
\item if the degree of the coefficient of degree $\deg_Y(Q)-1$ evaluated at $\alpha_i$ i is positive, there is only one asymptote going to $-\infty$ just before $\alpha_i$ and to $+\infty$ just  after $\alpha_i$ and we define
$$\textsc{Left}_{i,0}=1,\textsc{Right}_{i,0}=0, \textsc{Left}_{1,m_i+1}=0,=\textsc{Right}_{1,m_i+1}=1.$$
\item if the degree of the coefficient of degree $\deg_Y(Q)-1$ evaluated at $\alpha_i$ is negative
there is only one asymptote going to  $+\infty$ before $\alpha_i$ and  $-\infty$ after $\alpha_i$ and we define
$$\textsc{Left}_{i,0}=0,\textsc{Right}_{i,0}=1, \textsc{Left}_{1,m_i+1}=1,=\textsc{Right}_{1,m_i+1}=0.$$
\end{itemize}

When $\deg(Q(\alpha_i,Y))<\deg_Y(Q)-1$, $\alpha_i$ is also a root of $D_X$. The indices $i = 1,\ldots,N$ such that $\deg(Q(\alpha_i,Y))<\deg_Y(Q)-1$ are part of the output of Theorem \ref{sing-fibers}.

For vertical asymptotes at $-\infty$, we isolate the roots of $Q(X,\gamma^-_1)=0$ and decide the sign slope at the roots of
$Q(X,\gamma^-_1)=0$. This can be done in $\tO(d^4\tau+d^5)$ bit-operations since 
$Q(X,\gamma^-_1)$ is a polynomial of degree at most $d$ and bit size  $O(d^2 \tau+d^3)$.
We also compare the roots of $Q(X,\gamma^-_1)$ with the roots of $D_X$. Note first that $Q(X,\gamma^-_1)$ and $D_X$  have no roots in common. Since the separators of $Q(X,\gamma^-_1)$ and $D_X$ are both bounded by $2^{\tilde O(d^3 \tau+d^4)}$ this can be done by Proposition
\ref{sagraloff-isolation-integer} with complexity  $\tO(d^5\tau+d^6)$.

On each open interval $(\alpha_i,\alpha_{i+1})$, $i=1,\ldots,N$ delimited by roots of $D_X$, let $\textsc{Right}_{i,0}$ the number of roots of $Q(X,\gamma^-_1)=0$  such that the slope sign is $>0$ and   $\textsc{Left}_{i+1,0}$ the number of roots of $Q(X,\gamma^-_1)=0$  such that the  slope sign is $<0$.
We also denote by  let $\textsc{Right}_{N,0}$ the number of roots of $Q(X,\gamma^-_1)=0$  such that the slope sign is $>0$ on $(\alpha_N,+\infty)$ and   $\textsc{Left}_{1,0}$ the number of roots of $Q(X,\gamma^+_M)=0$  such that the  slope sign is $<0$  on $(-\infty, \alpha_1)$.

Note that all the roots $Q(X,\gamma^-_1)=0$ on $(\alpha_i,\alpha_{i+1})$ with positive  slope sign are bigger than all the roots $Q(X,\gamma^-_1)=0$ on $H_i$ with negative slope sign. Note also that if $\alpha_i$ is such that $\deg(Q(\alpha_i,Y))=\deg_Y(Q)$, 
$\textsc{Left}_{i,0}=\textsc{Right}_{i,0}=0$.

The situation at $+\infty$ is entirely similar and we define 
 $\textsc{Right}_{i,m_i+1}$ the number of roots of $Q(X,\gamma^+_M)=0$ on $(\alpha_i,\alpha_{i+1})$ such that the slope sign is $>0$ and   $\textsc{Left}_{i+1,m_{i+1}+1}$ the number of roots of $Q(X,\gamma^+_M)=0$ on $(\alpha_i,\alpha_{i+1})$ such that the slope sign is $<0$.
 We also denote by  let $\textsc{Right}_{N,m_N+1}$ the number of roots of $Q(X,\gamma^+_M)=0$  such that the slope sign is $<0$ on $(\alpha_N,+\infty)$ and   $\textsc{Left}_{1,m_1+1}$ the number of roots of $Q(X,\gamma^+_M)=0$  such that the  slope sign is $>0$  on $(-\infty, \alpha_1)$.
 
 Finally we have

\begin{proposition}\label{asympt}
The number of asymptotic branches tending to $-\infty$ (resp $+\infty$) to the left of $\alpha_i$ is $\textsc{Left}_{i,0}$
(resp. $\textsc{Left}_{i,m_i+1}$) and the number of asymptotic branches tending to $-\infty$ (resp $+\infty$) to the right of  $\alpha_i$  is $\textsc{Right}_{i,0}$ (resp. $\textsc{Right}_{i,m_i+1}$).Moreover the complexity of computing these numbers is  $\tO(d^5\tau+d^6)$.
\end{proposition}

  \subsection{Final topology}\label{topology}

 For $(\alpha_i,\beta_{i,j})$ $X$-critical, i.e. $j\in \textsc{CritInd}_i$, we denote as before  by $\textsc{Left}_{i,j}$ (resp. $\textsc{Right}_{i,j}$) the number of 
curve segments 
 arriving at $(\alpha_i,\beta_{i,j})$ inside
$\bbox(\alpha,\beta)$ .
  Note that this information has been computed in Proposition \ref{leftright} using $\tO(d^5\tau+d^6)$ bit operations.
 For $j\notin \textsc{CritInd}_i$, we take $\textsc{Left}_{i,j}=\textsc{Right}_{i,j}=1$. 
  If $j \textsc{CritInd}_i\setminus \textsc{SingInd}_i$, denoting by $\mu_{i,j}$  the multiplicity of $\beta_{i,j}$ as a root of $Q(\alpha_i,Y)$ 
\begin{itemize}
\item if $\mu_{i,j}$  is even and $\partial_Y Q^{(\mu_{i,j})}(\alpha_i,\beta_{i,j}))>0$,  $\textsc{Left}_{i,j}=0,\textsc{Right}_{i,j}=2$
\item if $\mu_{i,j}$  is even and $\partial_Y  Q^{(\mu_{i,j})}(\alpha_i,\beta_{i,j}))<0$, $\textsc{Left}_{i,j}=2,\textsc{Right}_{i,j}=0$
\item if $\mu_{i,j}$  is odd, $\textsc{Left}_{i,j}=\textsc{Right}_{i,j}=1$.
\end{itemize}
  as discussed  in Section \ref{sec:basic}.
  
 We denote as before  by   $\textsc{Left}_{i,0}$  (resp.  $\textsc{Right}_{i,0}$) the number 
 of vertical asymptotes tending to $-\infty$ at the left (resp. the right) of $\alpha_i$ and by $\textsc{Left}_{i,m_i+1}$  (resp.  $\textsc{Right}_{i,m_i+1}$) the number  of vertical asymptotes tending to $+\infty$ at the left (resp. the right) of $\alpha_i$.  Note that this information has been already determined in Proposition \ref{asympt}.

The topology of $V_{\mathbb{R}}(Q)$ is encoded by the finite list $$\tilde{\mathcal{L}}(Q)=[m'_0,L_1,\ldots,L_N ,m'_N]$$ where
\begin{itemize}
\item[-] $L_i=[m_{i},[[\textsc{Left}_{i,j},\textsc{Right}_{i,j}], 0 \le j\le m_{i}+1]]$  for $i=1,\ldots,N$,
\end{itemize}

In the special case where  $\deg_X(Q(X,Y))=0$  and $V_{\mathbb{R}} (Q)$ is a finite number of horizontal lines , we  compute the number $m$ of real roots of $Q(X,Y)$ which is a polynomial in $Y$.
Similarly in the special case where $D_X(X)$ has no real root, we compute the number $m$ of real roots of $Q(0,Y)$. 

In both special cases, the topology of $V_{\mathbb{R}}(Q)$ is encoded by $\tilde{\mathcal{L}}(Q)=[m]$.

\begin{exemple}\label{exemple1}
Now we illustrate the previous result by taking an example, where 
$$P(X,Y)= (4X+1)(8X-1)(16X-1)(XY-1)(4Y^2-4X-1)(4Y^2+4X-1),$$
and
$$Q(X,Y)= (XY-1)(4Y^2-4X-1)(4Y^2+4X-1).$$
 We have 
$$D_X=-2^{28}\,X^4\,\left(4\,X-1\right)\,\left(1+4\,X\right)\,\left(4-
 X^2+4\,X^3\right)^2\,\left(-4+X^2+4\,X^3\right)^2.$$
 Since
$4-
 X^2+4\,X^3$ (resp. $-4+X^2+4\,X^3$) has one real root on $[-1,-1/2]$ (resp. on $[1/2,1]$), we have $N=5$.
 We obtain
$$\tilde{\mathcal{L}}(Q)=[3,L_1,3,L_2,5,L_3,5,L_4,3,L_5 ,3]$$ where
\begin{itemize}
\item[-] $L_1=[2,[[0,0],[2,2],[1,1],[0,0]]]$
\item[-] $L_2=[4,[[0,0],[1,1],[1,1],[0,2],[1,1],[0,0]]]$
\item[-] $L_3=[2,[[1,0],[2,2],[2,2],[0,1]]]$
\item[-] $L_4=[4,[[0,0],[1,1],[2,0],[1,1],[1,1],[0,0]]]$
\item[-] $L_5=[2,[[0,0],[1,1],[2,2],[0,0]]]$
\end{itemize}
For example the  list $L_3$ means that above $\alpha_3=0$ there is
one asymptote going to $-\infty$ at the left of $\alpha_3$, then two singular points with two branches to the left and two branches to the right and one asymptote going to $+\infty$ at the right of $\alpha_3$.
\begin{figure}[!htp]
      \hfill\hbox to 0pt{\hss\includegraphics[width=8cm,height=8cm]{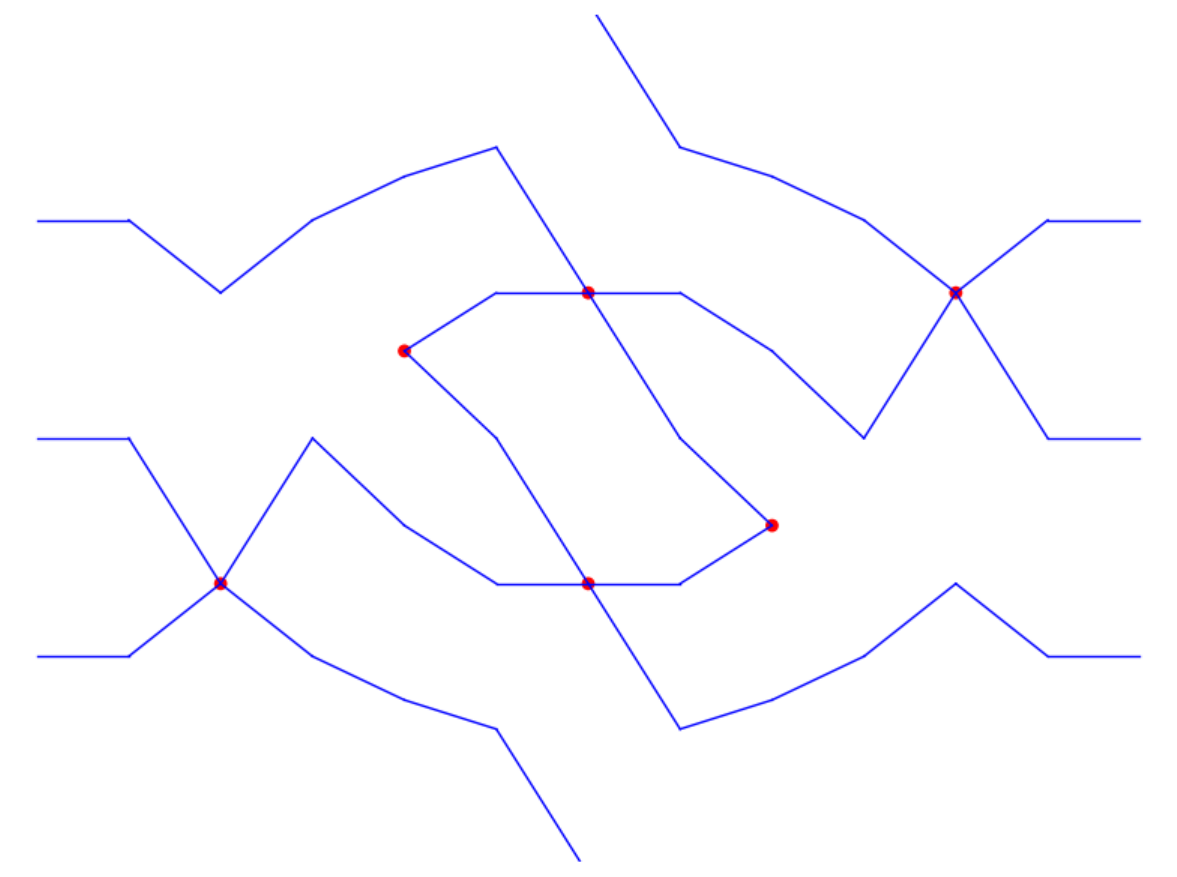}\hss}\hfill\null
\caption{\label{description_graph}Simple planar graph isotopic to the curve in  Example \ref{exemple1} without its vertical lines }
\end{figure}
\end{exemple}

A 
simple planar graph 
$\mathrm{Gr}(Q)$
 can be obtained as follows, 
\begin{itemize}
\item define $d'=\max(\max_{i=0,\ldots, N} m'_i,\max_{i=1,\ldots,N} m_i)$,
\item  for each $i=0,\ldots, N$ include the points $$I_{i,j}=\left(2i+1,\frac{j(d'+1)}{m'_i+1}\right)$$ for $j$ from $1$ to $m'_{i}$,
\item for each   $i=1,\ldots, N$ include the points $$P_{i,j}=\left(2i,\frac{j(d'+1)}{m_i+1}\right)$$ for $j$ from $0$ to $m_{j}+1$,
\item include the points $$P_{-\infty,j}=\left(0,\frac{j(d'+1)}{m'_0+1}\right)$$ for $j$ from $1$ to $m'_{0}$ as well as the points  $$P_{+\infty,j}=\left(2(N+1),\frac{j(d'+1)}{m'_N+1}\right)$$ for $j$ from $1$ to $m'_{N}$,
\item for $i=1,\ldots,N$, add the segment $I_{i-1,\ell}P_{i,j}$ (resp. $P_{i,j}I_{i,r}$) if $$\sum_{k=0}^{j-1} \textsc{Left}_{i,k}<\ell \le \sum_{k=0}^{j}\textsc{Left}_{i,k} (\mbox{resp.} \sum_{k=0}^{j-1} \textsc{Right}_{i,k}<r \le \sum_{k=0}^{j} \textsc{Right}_{i,k})$$
\item add the segments $P_{-\infty,j}I_{0,j}$ for $j=1,\ldots, m'_0$ and the segments
$I_{N,j} P_{\infty,j}$ for $j=1,\ldots, m'_N$
\end{itemize}
    
It is clear that
\begin{proposition}
$\mathrm{Gr}(Q)
\subset\mathbb (0,2 N)\times (0,d'+1)$
is homemorphic to $V_{\mathbb{R}}(Q)\subset\mathbb{R}^2$.
\end{proposition}

We finally explain how to add vertical lines back.
We remind that $c(X)$ is the gcd of all the coefficients $c_i(X)$ of $P(X,Y)$ written as an element of $\Z[X][Y]$
and we consider the  square-free part $c^\star(X)$ of $c(X)$.
 Define:
\begin{itemize}
\item $c_1(X):=$ gcd $(c^\star(X),D_X(X))$ and $c_2(X):= $quo$(c^\star(X),c_1(X))$, 
\item $\mathcal{V}_{1} := \{(x,y)\in \mathbb{R}^2 | c_{1}(x)=0\}$ and $\mathcal{V}_{2} := \{(x,y)\in \mathbb{R}^2 | c_{2}(x)=0\}$. 
\end {itemize}
Hence $\mathcal{V}_{1}$ is the subset of vertical lines of $V_{\mathbb{R}} ({P})$ passing through
zeroes of $D_XX$ 
while $\mathcal{V}_{2}$ is the subset of those passing between 
roots of $D_X$.
The computation of $c_{1}(X)$ and $c_{2}(X)$ has respectively bit complexities of $\tO(d^4 \tau + d^5)$ and $\tO(d\tau + d^2)$ according to Proposition \ref{gcd-comp} and  Proposition \ref{exact_division_comp}.
To add back the lines in $\mathcal{V}_{1}$ to $V_{\mathbb{R}} (Q)$, it suffices to identify the real roots of $c_{1}(X)$ as roots of $D_X(X)$, i.e. to decide whether the vertical line defined by $X=\alpha_i$ belongs to $V_{\mathbb{R}}(P)$. Such identification has bit complexity $\tO(d^5 \tau + d^6)$ according to Proposition \ref{comparingroots}. To add back the lines in $\mathcal{V}_{2}$ to $V_{\mathbb{R}} (Q)$, it suffices to 
order the roots of $c_2$ and the roots of $D_X$. This has bit complexity $\tO (d^5 \tau + d^6)$ according to Proposition \ref{comparingroots}.
\begin{proposition}
 Let $P \in \Z [X, Y]$ 
 a non-zero
 square-free polynomial of total degree $d$
 and integer coefficients of bitsize bounded by $\tau$. Adding back the vertical lines of $V_{\mathbb{R}} (P)$ to $V_{\mathbb{R}} (Q)$ has bit complexity $\tO
 (d^5 \tau + d^6)$. 
\end{proposition}

For a complete combinatorial description of  the topology of $V_{\mathbb{R}}(P)$ we define the  finite list $${\mathcal{L}}( P)=[N'_0,L'_1,\ldots,L'_\delta ,N'_\delta]$$ where
\begin{itemize}
\item[-] $L'_i=[[m_{i},w_i],[[\ell_{i,j},r_{i,j}], 0 \le j\le m_{i}+1]]$  for $i=1,\ldots,\delta$,
\item[-] $N'_i=[m'_i,v_i]$ for $i=0,\ldots,\delta$,
\end{itemize}
where $w_i=1$ if the line $X=\alpha_i$ belongs to  $V_{\mathbb{R}}(P)$ for $i=1,\ldots, \delta$ and  $w_i=0$ otherwise and $v_i$ is the number of distinct vertical lines $X=x$ with $\alpha_i<x<\alpha_{i+1}$ for $i=1,\ldots,\delta-1$  and $v_0$ (resp. $v_\delta$) is the number of distinct vertical lines $X=x$ with $x<\alpha_1$ (resp. $x>\alpha_\delta$).
Finally the
simple  planar graph 
$\mathrm{Gr}(P)$
is defined as $$\mathrm{Gr}(Q)\cup \bigcup_{i=1,\ldots,\delta\atop w_i=1} V_i \cup \bigcup_{i=0,\ldots,\delta\atop \ell=1,\ldots,v_i} V_{i,\ell}$$
with $V_i$ the vertical segment defined by $X=2i, 0<Y<d'+1$ and $V_{i,\ell}$ for $i=0,\ldots,\delta,\ell=1,\ldots,v_i$ is the vertical segment  defined by the equation
$$X=2i+\frac{2\ell}{v_i+1},0<Y<d'+1.$$

It is clear that
\begin{proposition}
$\mathrm{Gr}(P)\subset\mathbb (0,2\delta)\times (0,d'+1)$
is homemorphic to $V_{\mathbb{R}}(P)\subset\mathbb{R}^2$.
\end{proposition}

\begin{exemple}\label{exemple2}
Continuing Example \ref{exemple1},
we have 
three vertical lines of equation 
$$X=\frac{-1}{4},X=\frac{1}{8},X=\frac{1}{16}$$
 We obtain
$${\mathcal{L}}(P)=[N'_0,L'_1,N'_1,L'_2,N'_2,L'_3,N'_3,L'_4,N'_4,L'_5 ,N'_5]$$ where
\begin{itemize}
\item[-] $N'_0=[3,0]$
\item[-] $L_1=[[2,0],[[0,0],[2,2],[1,1],[0,0]]]$
\item[-] $N_1=[3,0]$
\item[-] $L_2=[[4,1],[[0,0],[1,1],[1,1],[0,2],[1,1],[0,0]]]$
\item[-] $N_2=[5,0]$
\item[-] $L_3=[[2,0],[[1,0],[2,2],[2,2],[0,1]]]$
\item[-] $N_3=[5,2]$
\item[-] $L_4=[[4,0],[[0,0],[1,1],[2,0],[1,1],[1,1],[0,0]]]$
\item[-] $N_4=[3,0]$
\item[-] $L_5=[[2,0],[[0,0],[1,1],[2,2],[0,0]]]$
\item[-] $N_5=[3,0]$
\end{itemize}
\begin{figure}[!htp]
      \hfill\hbox to 0pt{\hss\includegraphics[width=8cm,height=8cm]{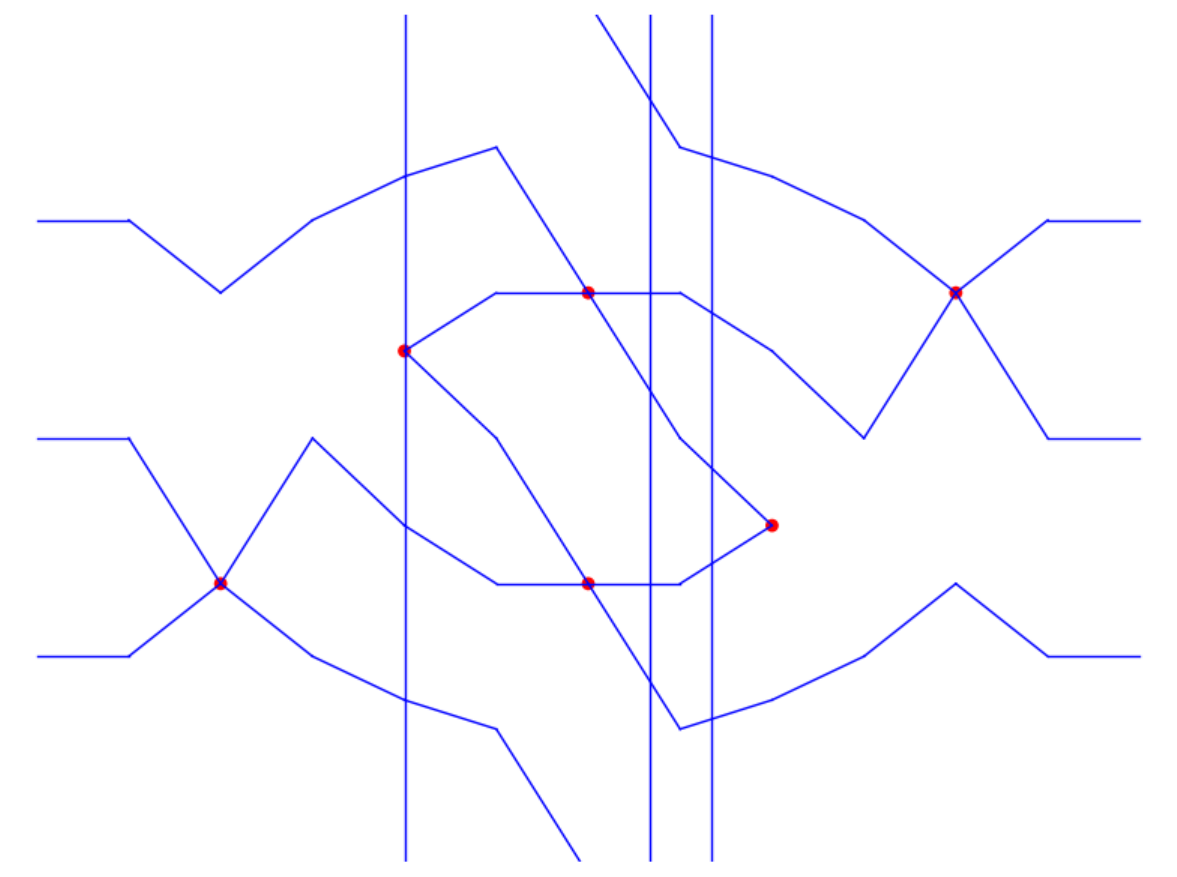}\hss}\hfill\null
\caption{\label{description_graphbis}Simple planar graph isotopic to the curve in  Example \ref{exemple1} }
\end{figure}
\end{exemple}

Finally, summarizing our results, we proved Theorem \ref{finaltopology}.

\section*{Acknowledgements}
We are grateful to the anonymous referees for their relevant remarks and suggestions.

\section*{Data sharing}
Data sharing not applicable to this article as no datasets were generated or analysed during the current study.

\newpage
 \bibliographystyle{plain} 
\bibliography{topology}

\def\cprime{$'$}
\begin{thebibliography}{10}

\bibitem{ST2018}
Elias P.~Tsigaridas Adam W.~Strzebonski.
\newblock Univariate real root isolation in an extension field and
  applications.
\newblock {\em Journal of Symbolic Computation}, 2018.

\bibitem{AM1}
Lionel Alberti and Bernard Mourrain.
\newblock Regularity criteria for the topology of algebraic curves and
  surfaces.
\newblock In {\em IMA Conference on the Mathematics of Surfaces}, pages 1--28,
  2007.

\bibitem{AM2}
Lionel Alberti and Bernard Mourrain.
\newblock Visualisation of implicit algebraic curves.
\newblock In {\em Pacific Conference on Computer Graphics and Applications},
  pages 303--312, 2007.

\bibitem{AMW}
Lionel Alberti, Bernard Mourrain, and Julien Wintz.
\newblock Topology and arrangement computation of semi-algebraic planar curves.
\newblock {\em Comptation Aided Geometric Design}, 25(8):631--651, 2008.

\bibitem{BaZa}
M.~Badrato and A.~Zanoni.
\newblock {\newblock} long integers and polynomial evaluation with estrin's
  scheme.
\newblock {\em In Proc. SYNACS'11}, pages 39--46, 2011.

\bibitem{BPRbook2}
S.~Basu, R.~Pollack, and M.-F. Roy.
\newblock {\em Algorithms in real algebraic geometry}, volume~10 of {\em
  Algorithms and Computation in Mathematics}.
\newblock Springer-Verlag, Berlin, 2006 (second edition).
\newblock Revised version of the second edition online at
  \url{http://perso.univ-rennes1.fr/marie-francoise.roy/}.

\bibitem{CIsolate}
R.~{Becker}, M.~{Sagraloff}, V.~{Sharma}, and C.~{Yap}.
\newblock {A Near-Optimal Subdivision Algorithm for Complex Root Isolation
  based on the Pellet Test and Newton Iteration}.
\newblock {\em ArXiv e-prints}, September 2015.

\bibitem{BEKS}
Eric Berberich, Pavel Emeliyanenko, Alexander Kobel, and Michael Sagraloff.
\newblock Exact symbolic–numeric computation of planar algebraic curves.
\newblock {\em Theoretical Computer Science}, 491:1 -- 32, 2013.

\bibitem{BLMPRS16}
Yacine Bouzidi, Sylvain Lazard, Guillaume Moroz, Marc Pouget, Fabrice
  Rouillier, and Michael Sagraloff.
\newblock Solving bivariate systems using rational univariate representations.
\newblock {\em J. Complexity}, 37:34--75, 2016.

\bibitem{BCGY}
M.~Burr, S.W.Choi, B.~Galehouse, and Chee Yap.
\newblock Complete subdivision algorithms, ii: Isotopic meshing of singular
  algebraic curves.
\newblock In {\em ISSAC}, 2008.

\bibitem{CLPPRT}
J.~Cheng, S.~Lazard, L.~Penaranda, M.~Pouget, F.~Rouillier, and E.~Tsigaridas.
\newblock On the topology of planar algebraic curves.
\newblock {\em Mathematics in Computer Science}, 1(14):113--137, 2011.

\bibitem{DMR}
Daouda~Niang Diatta, Bernard Mourrain, and Olivier Ruatta.
\newblock On the computation of the topology of a non-reduced implicit space
  curve.
\newblock In {\em ISSAC}, 2008.

\bibitem{EKW}
A.~Eigenwillig, M.~Kerber, and N.~Wolpert.
\newblock Fast and exact geometric analysis of real algebraic plane curves.
\newblock In {\em ISSAC}, pages 151--158, 2007.

\bibitem{GE1}
L.~Gonzalez-Vega and M.~El Kahoui.
\newblock An improved upper complexity bound for the topology computation of a
  real algebraic curve.
\newblock {\em Journal of Complexity}, 12:527--544, 1996.

\bibitem{GI}
L.~Gonzalez-Vega and I.~Necula.
\newblock Efficient topology determination of implicitly defined algebraic
  plane curves.
\newblock {\em Comput. Aided Geom. Design}, 19(9):719--743, 2004.

\bibitem{DBLP:phd/de/Kerber2009}
Michael Kerber.
\newblock {\em Geometric algorithms for algebraic curves and surfaces}.
\newblock PhD thesis, Saarland University, 2009.

\bibitem{KS}
Michael Kerber and Michael Sagraloff.
\newblock A worst-case bound for topology computation of algebraic curves.
\newblock {\em Journal of Symbolic Computation}, 47(3):239 -- 258, 2012.

\bibitem{KS15}
Michael Kerber and Michael Sagraloff.
\newblock Root refinement for real polynomials using quadratic interval
  refinement.
\newblock {\em Journal of Computational and Applied Mathematics}, 280:377 --
  395, 2015.

\bibitem{KoS}
Alexander Kobel and Michael Sagraloff.
\newblock On the complexity of computing with planar algebraic curves.
\newblock {\em Journal of Complexity}, 31(2):206 -- 236, 2015.

\bibitem{MSW1}
Kurt Mehlhorn, Michael Sagraloff, and Pengming Wang.
\newblock From approximate factorization to root isolation.
\newblock In {\em Proceedings of the 38th International Symposium on Symbolic
  and Algebraic Computation}, ISSAC '13, pages 283--290, New York, NY, USA,
  2013. ACM.

\bibitem{MSW2}
Kurt Mehlhorn, Michael Sagraloff, and Pengming Wang.
\newblock From approximate factorization to root isolation with application to
  cylindrical algebraic decomposition.
\newblock {\em Journal of Symbolic Computation}, 66:34 -- 69, 2015.

\bibitem{Pa}
Victor~Y. Pan.
\newblock Univariate polynomials: Nearly optimal algorithms for numerical
  factorization and root-finding.
\newblock {\em J. Symb. Comput.}, 5(33):701--733, 2002.

\bibitem{PT}
Victor~Y. Pan and Elias~P. Tsigaridas.
\newblock On the boolean complexity of real root refinement.
\newblock In {\em ISSAC}, pages 299--306, 2013.

\bibitem{SM16}
Michael Sagraloff and Kurt Mehlhorn.
\newblock Computing real roots of real polynomials.
\newblock {\em Journal of Symbolic Computation}, 73:46 -- 86, 2016.

\bibitem{GG}
J.~von~zur Gathen and J.~Gerhard.
\newblock {\em Modern computer algebra}.
\newblock Cambridge University Press, New York, 1999.

\bibitem{WM}
Julien Wintz and Bernard Mourrain.
\newblock A subdivision arrangement algorithm for semi-algebraic curves: An
  overview.
\newblock In {\em Pacific Conference on Computer Graphics and Applications},
  pages 449--452, 2007.

\end{thebibliography}

\end{document}